%% file: geometrictduality.tex
\documentclass[final,oneside,notitlepage,10pt]{article}

\input{styles}

\input{kobib}
\input{makros}

\input{hyphenation_en}
\input{theorems}

\usepackage{tikz-cd}

\def\poi{\mathcal{P}}
\def\GTB{\mathrm{T}\text{-}\mathrm{BG}^{\mathrm{geo}}}
\def\TTB{\mathrm{T}\text{-}\mathrm{BG}^{\mathrm{top}}}

\def\DTB{\mathrm{T}\text{-}\mathrm{BG}^{\mathrm{diff}}}
\def\GC{\mathrm{Corr}^{\mathrm{geo}}}
\def\TC{\mathrm{Corr}^{\mathrm{top}}}

\def\DC{\mathrm{Corr}^{\mathrm{diff}}}
\def\GTC{\mathrm{T}\text{-}\GC}
\def\TTC{\mathrm{T}\text{-}\TC}

\def\DTC{\mathrm{T}\text{-}\DC}
\def\DTP{\mathrm{TDP}}
\def\LDgeo{\mathrm{Loc}^{\mathrm{geo}}}
\def\LDdiff{\mathrm{Loc}^{\mathrm{diff}}}
\def\LDtop{\mathrm{Loc}^{\mathrm{top}}}

\title{Geometric T-duality: Buscher rules in general topology}

\author{Konrad Waldorf}
\email{konrad.waldorf@uni-greifswald.de}

\adress{
  Universität Greifswald\\
  Institut für Mathematik und Informatik\\
  Walther-Rathenau-Str. 47\\
  D-17487 Greifswald}

\msc{}
\dedication{Dedicated to Krzysztof Gaw\c edzki}
\pdfdaten

\denseversion
\nolinks
\allowdisplaybreaks

\usepackage{amstext}
\usepackage{amsmath}

\crefname{equation}{\unskip}{\unskip}

\begin{document}

\maketitle

\begin{abstract}
The classical Buscher rules describe T-duality for metrics and B-fields  in a topologically trivial setting. On the other hand, topological T-duality addresses aspects of non-trivial topology while  neglecting metrics and B-fields. In this article we develop a new unifying framework for both aspects. 

\end{abstract}

\mytableofcontents
\clearpage

\setsecnumdepth{1}

\section{Introduction}

\label{sec:introduction}

Mathematical models for  string theory are based on  \emph{geometric backgrounds} consisting of
\begin{itemize}

\item 
a smooth manifold $E$ (spacetime),

\item
a Riemannian metric $g$ on $E$ (gravity field),
and
\item
a bundle gerbe $\mathcal{G}$ with connection over $E$ (Kalb-Ramond field). 

\end{itemize}
A special class of Kalb-Ramond fields is given by B-fields, i.e., 2-forms $B \in \Omega^2(E)$; these are precisely the connections on the \emph{trivial} bundle gerbe. 
Geometric backgrounds (are supposed to) determine 2-dimensional  quantum field theories, and an important question is  when two geometric backgrounds determine the same theory.

In the context of T-duality, one  assumes that spacetimes $E$ have a toroidal symmetry: an action of the $n$-dimensional torus $\T^{n}$ on $E$, such that $g$ is invariant and $E$ is a principal $\T^{n}$-bundle over the quotient $X:=E/\T^{n}$.
We will use the terminology \emph{geometric T-background} for geometric backgrounds with toroidal symmetry. 
When are two geometric T-backgrounds $(E,g,\mathcal{G})$ and $(\hat E,\hat g,\hat{\mathcal{G}})$ T-dual, i.e., when do they determine the same quantum field theory? To the best of my knowledge, no general conditions are known -- unless the data of a geometric  T-background are simplified in one way or another. The purpose of the present paper is to propose such general conditions, implying those of   all simplified situations.

Buscher provided conditions for T-duality  \cite{Buscher1987} in a topologically trivial situation, where $E = X \times \T$ is the trivial circle bundle (i.e., $n=1$) over an open subset $X \subset \R^{s}$, and the bundle gerbe $\mathcal{G}$ is just a B-field $B \in \Omega^2(E)$. These conditions are the by now classical \emph{Buscher rules}:
\begin{align*}
&\hat g_{\theta\theta}=\frac{1}{g_{\theta\theta}},\qquad
\hat g_{\alpha\theta}= \frac{B_{\alpha\theta}}{g_{\theta\theta}},\qquad
\hat g_{\alpha \beta }=g_{\alpha \beta }-\frac{1}{g_{\theta\theta}}(g_{\alpha\theta}g_{\beta \theta}-B_{\alpha\theta}B_{\beta \theta})
\\
&\hat B_{\alpha\theta}=\frac{g_{\alpha\theta}}{g_{\theta\theta}},\qquad \hat B_{\alpha \beta }=B_{\alpha \beta }-\frac{1}{g_{\theta\theta}}(B_{\alpha\theta}g_{\beta \theta}-g_{\alpha\theta}B_{\beta \theta})
\end{align*}
Here, the indices label the coordinates of the direct product $E=X \times \T$, with $\alpha,\beta =1,...,s$ coordinates of $X$ and $\theta$ the single fibre coordinate. 
The Buscher rules can be generalized to arbitrary torus dimension $n$, see \cite{Giveon1994}.

A groundbreaking observation of Giveon et al. \cite{Giveon1993,Giveon1993b} and Alvarez et al. \cite{Alvarez1994}
was that (even in the case $n=1$) the Buscher rules  require a \emph{topology change} as soon as $X \subset \R^{s}$ is replaced by a topologically non-trivial manifold. The example studied in \cite{Giveon1993b} is when $E=S^3$ is the Hopf fibration over $X=S^2$,  $g$ is the round metric on $S^3$, and $B=0$. One can then cover $S^2$  by open subsets $U_i \subset S^2$ over which $E$ trivializes, and apply the Buscher rules on each patch to obtain locally defined  dual metrics $\hat g_i$ and a dual B-fields $\hat B_i$. The observation is then that these locally defined data do \emph{not} glue to a new metric     and B-field on the Hopf fibration, but rather to a new metric and B-field on the \emph{trivial} bundle $\hat E := S^2 \times \T$. In other words, spacetime changes its topology under T-duality! A second important development, due to   Hori \cite{hori1}, was a \quot{Fourier-Mukai} transformation for Ramond-Ramond charges on D-branes accompanying T-duality, involving  topological K-theory of spacetimes and the Poincaré bundle over $\T^{n} \times \T^{n}$.

The topology change and the relation to K-theory sparked the interest of mathematicians in T-duality, and the question emerged for a formulation of Buscher rules in (more) general topology. 
Basically at the same time, string theorists and mathematicians explored topological aspects of  B-fields. The first account in this direction was Gaw\c edzki's work on topological effects in 2-dimensional sigma models using Deligne cohomology \cite{gawedzki3}, and Alvarez' work on topological quantization \cite{alvarez1}. A major step was the invention of bundle gerbes by Murray \cite{murray} that unleashed a number of advances, e.g. a complete classification of WZW models on compact simple Lie groups \cite{gawedzki2}, a corresponding classification of  D-branes in these models \cite{gawedzki1,gawedzki4}, a discussion of D-branes in terms of twisted K-theory \cite{bouwknegt1} or a classification of WZW orientifolds \cite{schreiber1,gawedzki6}. 
Bundle gerbes with connection have an underlying topological part, measured by their \emph{Dixmier-Douady class} in $\h^3(E,\Z)$, and  a curvature, a 3-form $H\in \Omega^3(E)$ called \emph{H-flux}. If the Dixmier-Douady class vanishes, then they reduce -- up to isomorphism -- to a trivial bundle gerbe $\mathcal{I}_B$ carrying the former B-field $B$, such that $H=\mathrm{d}B$. 
Despite  these advances, the quite complicated interplay between metric and B-field, which is characteristic for the Buscher rules, did not have a straightforward generalization from B-fields to bundle gerbes with connection. 

     Bouwknegt-Evslin-Mathai observed in \cite{Bouwknegt,Bouwknegt2004b} that the topology change can also be observed by only looking at the H-flux, while discarding metrics and the remaining data of the bundle gerbe and its connection.   An important result of the work of Bouwknegt-Evslin-Mathai was to establish the Fourier-Mukai transformation in twisted de Rham cohomology,  an isomorphism
\begin{equation*}
\h^{\bullet}_{\mathrm{dR}}(E,H) \cong \h^{\bullet+1}_{\mathrm{dR}}(\hat E,\hat H)\text{.}
\end{equation*}
Another important observation in this context was the eventual non-existence of T-duals in case of torus dimension $n>1$. Mathai-Rosenberg explored these missing T-duals by invoking non-commutative geometry \cite{Mathai2006,Mathai2005a,Mathai2006a}. 

As the curvature $H$ of a bundle gerbe with connection represents the Dixmier-Douady class  only in \emph{real} cohomology, it neglects torsion. Bunke-Rumpf-Schick invented a framework of \emph{topological T-duality} that captures the full information of the two Dixmier-Douady classes, but now completely neglects connections and metrics \cite{Bunke2005a,Bunke2006a}. Their framework introduced a new and very enlightening aspect to T-duality. So far, T-duality was understood as a transformation, a map, taking one  T-background to another, T-dual one.  However, as mentioned above, some T-backgrounds do not have any T-duals. Even worse, if $n>1$, T-backgrounds have many different T-duals. Thus, T-duality is by no means a map. Bunke-Rumpf-Schick implemented this insight by describing T-duality as a \emph{relation} on  the space of topological  T-backgrounds (the latter consisting of a principal $\T^{n}$-bundle $E$ and a bundle gerbe $\mathcal{G}$ over $E$ without connection).   It might be good to remark that this relation is not an equivalence relation; it is only symmetric, but neither reflexive nor transitive. 
The relation is established by the existence of an isomorphism 
\begin{equation}
\label{eq:ci}
\pr^{*}\mathcal{G} \cong \hat\pr^{*}\hat{\mathcal{G}}
\end{equation}
between the pullbacks of the two bundle gerbes to the so-called \emph{correspondence space}, the fibre product
\begin{equation*}
\alxydim{@C=1em@R=2em}{& E \times_X \hat E \ar[dr]^{\hat\pr} \ar[dl]_{\pr} \\ E \ar[dr] && \hat E \ar[dl] \\ & X\text{.}}
\end{equation*}
Moreover, the isomorphism \cref{eq:ci} has to satisfy a certain \emph{Poincaré condition}, relating it to the Poincaré bundle over $\T^{n} \times \T^{n}$. Bunke-Rumpf-Schick then started to explore the space of \emph{topological T-duality correspondences}, consisting of two topological T-backgrounds $(E,\mathcal{G})$ and $(\hat E,\hat{\mathcal{G}})$, and an isomorphism  \cref{eq:ci}, in its dependence on $X$.

Bunke-Rumpf-Schick's new perspective on topological T-duality was accompanied by a number of important results \cite{Bunke2006a}: a precise criterion when a topological T-background $(E,\mathcal{G})$ admits a T-dual (the Dixmier-Douady class of $\mathcal{G}$ has to lie in the second step of the standard  filtration of $\h^3(E,\Z)$) and a parameterization of all possible T-duals (the group $\mathfrak{so}(n,\Z)$ of skew-symmetric integral $n\times n$-matrices acts freely and transitively on them). Moreover, Bunke-Rumpf-Schick obtained a version of the Fourier-Mukai transformation in topological twisted K-theory,
\begin{equation*}
K^{\bullet}(E,\xi) \cong K^{\bullet+1}(\hat E,\hat\xi)\text{,}
\end{equation*}   
where the twists $\xi,\hat\xi$ are the Dixmier-Douady classes of the bundle gerbes $\mathcal{G}$ and $\hat{\mathcal{G}}$.

A further approach towards a unification of the topological data of a bundle gerbe and the differential data of its connection was proposed by Kahle-Valentino \cite{Kahle}. It can be understood as reducing the full information of a geometric  T-background $(E,g,\mathcal{G})$ to the information of $E$, the full bundle gerbe with connection $\mathcal{G}$, and of the connection $\omega$ on $E$ obtained from the metric $g$ under  Kaluza-Klein reduction. This approach  remained rather unrelated to the previous approaches, in particular, to the Buscher rules, and also is formulated in a rather uncommon language of \quot{differential cohomology groupoids}. Nonetheless, we will show here that the approach of Kahle-Valentino is  very close to the formalism  we will present below as  \quot{geometric T-duality}. Kahle-Valentino also propose a very interesting generalization of Bunke-Rumpf-Schick's Fourier-Mukai transformation from twisted K-theory to twisted \emph{differential}  K-theory, 
\begin{equation*}
\hat K^{\bullet}(E,\mathcal{G}) \cong \hat K^{\bullet+1}(\hat E,\mathcal{G})\text{.}
\end{equation*}   
Unfortunately, at that time, no general theory for twisted differential K-theory was available, and so Kahle-Valentino proposed an axiomatic description under which the isomorphism was deduced. However, it seems to be unclear if these axioms are met by existing models, e.g. \cite{carey8}, or if they can be proved in a modern framework of twisted differential cohomology, e.g. \cite{Bunke2019,Grady2019}.

In this article, we propose a new formalism which we call \emph{geometric T-duality}. It is based on the full information of geometric T-backgrounds: a principal $\T^{n}$-bundle $E$ over an arbitrary smooth manifold $X$, an invariant Riemannian metric $g$ on $E$, and a bundle gerbe with connection $\mathcal{G}$ on $E$. The main point of our new formalism is a notion of  \emph{geometric T-duality correspondence} as a relation on the set of all such geometric T-backgrounds (\cref{def:gtdc}). The main ingredient is that the isomorphism \cref{eq:ci} on the correspondence space is now a connection-preserving isomorphism
\begin{equation}
\label{eq:ci2}
\pr^{*}\mathcal{G} \cong \hat\pr^{*}\hat{\mathcal{G}} \otimes \mathcal{I}_{\rho_{g,\hat g}}\text{,}
\end{equation}
where $\rho_{g,\hat g}$ is a certain 2-form produced from the metrics $g$ and $\hat g$.  The isomorphism \cref{eq:ci2} is then required to satisfy a differential version of Bunke-Rumpf-Schick's Poincaré condition.  
The most important result about our new geometric T-duality is that it indeed unifies all aspects investigated before separately.

\begin{theorem}
\label{th:main1}
Suppose two string backgrounds $(E,g,\mathcal{G})$ and $(\hat E,\hat g,\hat{\mathcal{G}})$ are in geometric T-duality correspondence. Then, the following statements are true:
\begin{enumerate}[\normalfont (1)]

\item 
\label{th:main1:1}
Locally, the Buscher rules are satisfied. More precisely, there exist local  trivializations $\varphi: U \times \T^{n} \to E$ and $\hat\varphi: U \times \T^{n} \to \hat E$ and bundle gerbe trivializations $\varphi^{*}\mathcal{G} \cong \mathcal{I}_B$ and $\hat\varphi^{*}\hat{\mathcal{G}} \cong \mathcal{I}_{\hat B}$ such that $(g,B)$ and $(\hat g,\hat B)$ satisfy the Buscher rules. 

\item
\label{th:main1:2}
Discarding metrics and bundle gerbe connections,  $(E,\mathcal{G})$ and $(\hat E,\hat{\mathcal{G}})$ are in topological T-duality correspondence in the sense of Bunke-Rumpf-Schick. 

\item
\label{th:main1:3}
Discarding metrics, and replacing the bundle gerbes by their curvature 3-forms, $(E,H)$ and $(\hat E,\hat H)$ are T-dual as backgrounds with H-flux in the sense of Bouwknegt-Evslin-Mathai. 

\item
\label{th:main1:4}
Replacing the metrics $g$ and $\hat g$ by their Kaluza-Klein connections $\omega$ and $\hat\omega$, respectively, $(E,\omega)$ and $(\hat E,\hat\omega)$ form a differential T-duality pair in the sense of Kahle-Valentino.

\end{enumerate}
\end{theorem}  

The proof of \cref{th:main1:1*} consists of some computations with differential forms, metrics, and connections performed in \cref{sec:buscher}; the statement is \cref{lem:Bs} in the main text.  \cref{th:main1:2*,th:main1:3*} follow  directly from the definitions, see \cref{prop:geototop,prop:geotoH}. The proof of \cref{th:main1:4*} is rather involved due to the very different settings. In order to prove \cref{th:main1:4*}, we introduce in \cref{sec:differentialTduality} another formalism that we call \quot{differential T-duality}; we then show in \cref{prop:geotodiff} that it is a consequence of geometric T-duality, and prove in \cref{prop:dttanddtc} that it is equivalent to Kahle-Valentino's setting.   
 
We remark that our terminology  \quot{geometric} does not refer to the question whether or not dual T-backgrounds can be modelled on ordinary torus bundles, as opposed to the non-commutative ones of Mathai-Rosenberg. Instead, it will be used here in order to distinguish our setting from \quot{topological} T-duality and \quot{differential} T-duality.
\cref{fig:graph} expresses the implications between the various notions of T-duality.

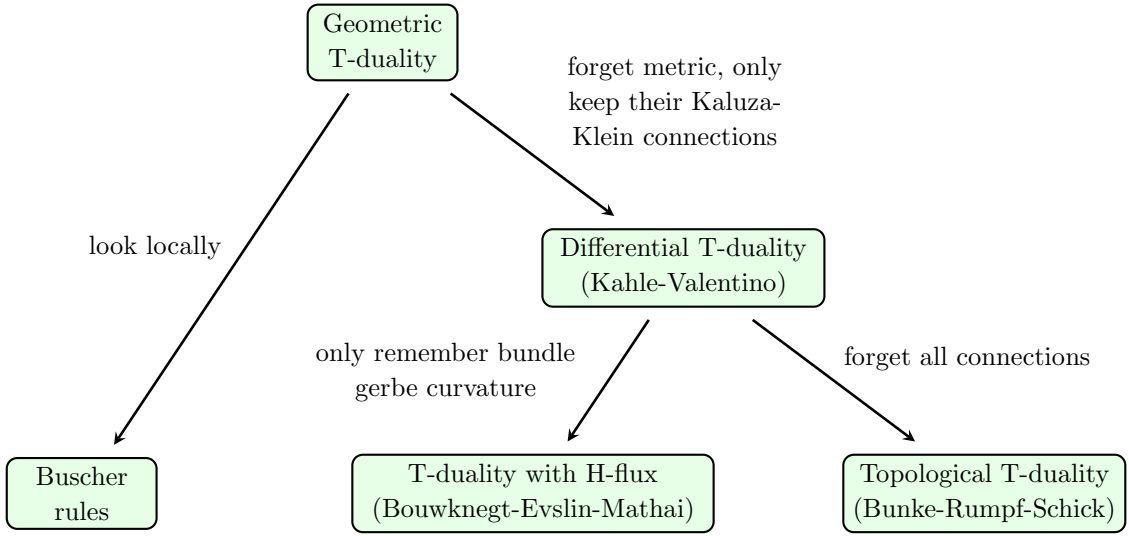
\begin{figure}[h]
\begin{center}
\begin{tikzpicture} 
[block/.style ={rectangle,outer sep=5, draw=black, thick, fill=green!10, text width=5em,align=center, rounded corners, minimum height=2em},
arrow/.style={-stealth, line width = 1pt}]
 
\node[block] (G) at (4,6)  {Geometric T-duality};
\node[block] (B) at (0,0)  {Buscher rules};
\node[block, text width=10em] (D) at (8,3)  {Differential T-duality (Kahle-Valentino)};
\node[block, text width=13em] (H) at (6,0)  {T-duality with H-flux (Bouwknegt-Evslin-Mathai)};
\node[block, text width=10em] (T) at (12,0)  {Topological T-duality (Bunke-Rumpf-Schick)};
 
\draw[arrow]  (G) -- (B) node[midway,above left,text width=5em] {look locally}; 
\draw[arrow] (G) -- (D)  node[midway,above right,text width=10em,text centered] {forget metric, only keep their Kaluza-Klein connections}; 
\draw[arrow] (D) -- (H)  node[near end,above left,text width=10em,text centered] {only remember bundle gerbe curvature}; 
\draw[arrow] (D) -- (T)  node[midway,above right,text width=10em] {forget all connections}; 
 
\end{tikzpicture}
\end{center}
\caption{A schematic overview about the various versions of T-duality considered here, and how  they imply each other.}
\label{fig:graph}
\end{figure}

\cref{th:main1} says that geometric T-duality reduces to several known forms of T-duality. We also consider  the opposite question: can these other formulations of T-duality  be upgraded to full geometric T-duality? 

\begin{theorem}
\label{th:main2}
\begin{enumerate}[\normalfont (a)]

\item 
\label{th:main2:1}
Locally, geometric T-duality is equivalent to the Buscher rules. More precisely, suppose $(g,B)$ and $(\hat g,\hat B)$ satisfy the Buscher rules. Then, the  geometric T-backgrounds $(X \times \T^{n},g,\mathcal{I}_B)$ and $(X \times \T^{n},\hat g,\mathcal{I}_{\hat B})$ are in geometric T-duality correspondence. 

\item
\label{th:main2:4}
Every topological T-duality correspondence can be lifted to a geometric T-duality correspondence. More precisely, 
suppose $(E,\mathcal{G})$ and $(\hat E,\hat{\mathcal{G}})$ are topological T-backgrounds, and suppose $\mathcal{D}$ is a topological T-duality correspondence. Then, there exist $\T^{n}$-equivariant metrics $g$ and $\hat g$ on $E$ and $\hat E$,  and connections on $\mathcal{G}$ and $\hat{\mathcal{G}}$ such that $\mathcal{D}$ is a geometric T-duality correspondence between $(E,g,\mathcal{G})$ and $(\hat E,\hat g,\hat{\mathcal{G}})$. 

\item
\label{th:main2:3}
Every differential T-duality pair can be lifted to a geometric T-duality correspondence. The precise statement is in \cref{prop:upgradefromdifftogeo}.

\item
\label{th:main2:2}
Every topological T-duality correspondence can be lifted to a differential T-duality pair. The precise statement is in \cref{prop:upgradefromtoptodiff}.

\end{enumerate}
\end{theorem}

The proof of \cref{th:main2:1*} is rather straightforward, see \cref{prop:gtdclocally}. \cref{th:main2:4*} follows from \cref{th:main2:3*,th:main2:2*}, see \cref{prop:upgradefromtoptogeo}. \cref{th:main2:3*} is a direct consequence of close relationship between geometric and differential T-duality. The proof of \cref{th:main2:2*} is the hardest part, see  \cref{prop:upgradefromtoptodiff}. In order to prove it, we introduce in \cref{sec:localformalism} a \emph{local formalism} for geometric T-duality, i.e., we introduce a complete description in terms of functions and differential forms w.r.t. an open cover. Locally, on an open set $U_i$ this formalism gives precisely the Buscher rules. Additionally, it contains data and conditions on double, triple, and quadruple overlaps -- \emph{higher order Buscher rules}. To the best of my knowledge, these higher order Buscher rules have not been described before. 
\cref{fig:localdata} summarizes our local description. For a more detailed explanation of these data and conditions we refer to \cref{sec:localdata}.       

\begin{figure}[h]
\begin{center}
\small
\begin{tabular}{cccc}
\begin{minipage}[c]{4em}
\center
\textbf{inter\-sections}\end{minipage}\medskip\hspace{-0.9em}
& \begin{minipage}[c]{10em}
\center
\textbf{background }\end{minipage} & \begin{minipage}[c]{16em}
\center
\textbf{geometric T-duality correspondence }
\end{minipage} &  
\begin{minipage}[c]{10em}
\center
\textbf{dual background }
\end{minipage}
 \\\hline
\begin{minipage}[c]{3em}
\raggedleft
1-fold
\end{minipage} & \medskip\begin{minipage}[c]{10em}
\medskip\center 2-form $B_i$\\metric $g_i$
\end{minipage} & ordinary Buscher rules & \begin{minipage}[c]{10em}\medskip
\center 2-form $\hat B_i$\\metric $\hat g_i$
\end{minipage} \\ \hline 
\begin{minipage}[c]{3em}
\raggedleft
2-fold
\end{minipage} & \begin{minipage}[c]{10.3em}
\medskip\center
$\R^{n}$-valued functions 
$a_{ij}$\\such that $a_{ij}^{*}g_j = g_i$
\\\smallskip
gauge potentials $A_{ij}$ s.t. $a_{ij}^{*}B_j = B_i + \mathrm{d} A_{ij}$
\end{minipage} & \begin{minipage}[c]{16em}
\begin{center}
\emph{New:} second order Buscher rules\\\smallskip
$\hat A_{ij}+a_{ij}\hat\theta = A_{ij}+\hat a_{ij}\theta-a_{ij}\mathrm{d}\hat a_{ij}$
\end{center}
\end{minipage}  & \begin{minipage}[c]{10.3em}
\medskip\center
$\R^{n}$-valued functions 
$\hat a_{ij}$\\such that $\hat a_{ij}^{*}\hat g_j = \hat g_i$
\\\smallskip
gauge potentials $\hat A_{ij}$ s.t. $\hat a_{ij}^{*}\hat B_j = \hat B_i + \mathrm{d} \hat A_{ij}$
\end{minipage} \\[2.7em]\hline
\begin{minipage}[c]{3em}
\raggedleft
3-fold\end{minipage} & 
\begin{minipage}[c]{10em}
\medskip
\center
winding numbers\\
$a_{ik}=m_{ijk}+a_{ij}+a_{jk}$
\\
\smallskip
gauge transformations\\$c_{ijk}$ such that
\\ 
$A_{ik} = a_{ij}^{*} A_{jk}+  A_{ij}  + c_{ijk}^{*}\theta$
\end{minipage} & \begin{minipage}[c]{16em}
\medskip\begin{center}
\emph{New:} third order Buscher rules\\\smallskip $\hat c_{ijk}(x,\hat a)+m_{ijk}(\hat a_{ik}(x)+\hat a)$\\$= c_{ijk}(x,a)  +\hat m_{ijk}a$\\$\hspace{4em}-a_{ij}(x)\hat a_{jk}(x)$
\end{center}
\end{minipage} & \begin{minipage}[c]{10em}
\medskip\center
winding numbers\\
$\hat a_{ik}=\hat m_{ijk}+\hat a_{ij}+\hat a_{jk}$
\\
\smallskip
gauge transformations \\$\hat c_{ijk}$ such that
\\$\hat A_{ik} = \hat a_{ij}^{*} \hat A_{jk} +  \hat A_{ij} + \hat c_{ijk}^{*}\theta$
\end{minipage} \\[4em]\hline
\begin{minipage}[c]{3em}
\raggedleft
4-fold
\end{minipage}
 & \begin{minipage}[c]{10em}
 \center
\medskip cocycle condition\\$a_{ij}^{*}c_{jkl}\cdot c_{ijl} = c_{ijk}\cdot c_{ikl}$\end{minipage} &  & \begin{minipage}[c]{10em}
\medskip\center cocycle condition\\$\hat a_{ij}^{*}\hat c_{jkl}\cdot \hat c_{ijl} = \hat c_{ijk}\cdot \hat c_{ikl}$\end{minipage}
\end{tabular}
\end{center}
\caption{Local data for geometric T-backgrounds and geometric T-duality correspondences. The first line is the well-known local (topologically trivial) situation. The columns \quot{background} and \quot{dual background} each lists separately  the  local data from which one can glue a principal $\T^{n}$-bundle, an invariant metric, and a bundle gerbe with connection. The transition functions $a_{ij}$  and $\hat a_{ij}$ are taken to be $\R^{n}$-valued, revealing winding numbers $m_{ijk}$ and $\hat m_{ijk}$, respectively. The middle column  shows how the (higher) Buscher rules mix these local data from both sides.  }
\label{fig:localdata}
\end{figure}

We summarize the local data described in \cref{fig:localdata} (up to a certain notion of equivalence, and in the direct limit over refinements of open covers) in a set $\LDgeo(X)$. We also look at slightly  smaller versions:
\begin{itemize}

\item 
$\LDdiff(X)$, where the metrics are replaced by their Kaluza-Klein connections. 

\item
$\LDtop(X)$, where all metrics and differential forms, and all conditions involving them, are removed. 

\end{itemize}
These slightly smaller versions are very illuminating and important, not only for our proofs, but also because they can be related to another interesting quantity, namely the \emph{non-abelian differential cohomology} with values in the T-duality 2-group $\mathbb{TD}$, $\hat \h^1(X,\mathbb{TD})$. More precisely, it is its \emph{adjusted} version $\hat\h^1(X,\mathbb{TD}_{\kappa})$ in the sense of Kim-Saemann \cite{Kim2020,Kim2022} that becomes relevant here.  The 2-group $\mathbb{TD}$ has been introduced in \cite{Nikolause}, where we proved that the (non-differential) non-abelian cohomology $\h^1(X,\mathbb{TD})$ classifies \emph{topological} T-duality correspondences. The following result, in particular, extends this classification to differential and geometric T-duality correspondences. We denote by $\GTC(X)$, $\DTC(X)$, and $\TTC(X)$ the sets of equivalence classes of geometric, differential, and topological T-duality correspondences, respectively.

\begin{theorem}
\label{th:main:3}
There is a commutative diagram
\begin{equation*}
\alxydim{}{\GTC(X) \ar[d]^{\cong}_{(b)} \ar@{->>}[r]^-{(a)} & \DTC(X) \ar[d]^{\cong}_{(c)} \ar@{->>}[r] & \TTC(X) \ar[d]^{\cong}_{(d)} &\text{(global level)}\;\qquad \\ \LDgeo(X) \ar@{->>}[r] & \LDdiff(X) \ar[d]^{\cong}_{(e)} \ar@{->>}[r] & \LDtop(X) \ar[d]^{\cong}_{(f)} & \text{(local level)}\;\;\qquad \\ &\hat\h^1(X,\mathbb{TD}_{\kappa}) \ar@{->>}[r]_{(g)} & \h^1(X,\mathbb{TD}) & \text{(cohomology level)}}
\end{equation*}
in which all vertical arrows are bijections, and all horizontal arrows are surjections.
\end{theorem}

The surjectivity of the map (a) follows from \cref{th:main2:3}.
The most laborious part in \cref{th:main:3} is the construction of the map (b), establishing the relation between the global geometric formalism and the local formalism, and the proof that (b) is a bijection. This is undertaken in \cref{sec:reconstruction,sec:welldefinednessofreconstruction,sec:localtoglobal}, culminating in \cref{th:localtoglobalequivalence}. That the map (c) is a bijection can then easily be deduced from the bijectivity of (b), see \cref{th:localtoglobalequivalencediff}.
Construction and a proof of  bijectivity of the maps (e) and (f) are rather tedious calculations with local data and $\mathbb{TD}$-cocycles, and are performed in \cref{prop:ldequivtop,prop:LDdifftohatHbij}. The bijectivity of (f) together with above-mentioned classification result of \cite{Nikolause} imply  the bijectivity of (d), see \cref{prop:localtop}.
The final statement, the surjectivity of the forgetful map (g) from differential to non-differential non-abelian cohomology, is then a rather short -- though important -- calculation, performed in \cref{prop:lift}. Via the bijections (c) to (f), we obtain then the proof of \cref{th:main2:2}.

Apart from the results described above, we consider an interesting action of the (abelian) differential cohomology $\hat \h^3(X)$ on the set $\GTC(X)$ of all geometric T-duality correspondences. This action  has counterparts in the setting of differential and topological T-duality correspondences, and has also been studied by Bunke-Rumpf-Schick \cite{Bunke2006a}, see \cref{lem:action,prop:actionfreeandtransitive}.

Finally, we remark that \cref{th:main2:2} guarantees the existence of many examples of geometric T-duality correspondences. In \cref{sec:examples} we describe explicitly two full examples of  geometric T-duality correspondences. The first concerns a geometric T-background of the form $(E,g,\mathcal{I}_0)$, i.e., an arbitrary principal $\T^{n}$-bundle $E$ with an arbitrary metric $g$ and trivial B-field. Reducing this to the case in which $E=S^3 \to S^2$ is the Hopf fibration, and $g$ is the round metric on $S^3$, we reproduce the example of Alvarez et al.  \cite{Alvarez1994} and the observation of a topology change, now in the full setting of geometric T-duality. The second example is again the Hopf fibration and the round metric, but now equipped with the \quot{basic} gerbe of $S^3\cong \su 2$. It was known in the setting of T-duality with H-flux that this T-background is self-dual. We confirm that self-duality persists in the full setting of geometric T-duality, see \cref{prop:selfdual}. In particular, it follows from \cref{th:main1} that self-duality holds  in pure topological T-duality, and that the Buscher rules are satisfied locally.

\paragraph{Acknowledgements.} I would like to thank Ines Kath, Christian Saemann, and Tilmann Wurzbacher for helpful discussions and comments.  The beginning of this work was the PhD project of Malte Kunath, whose thesis \cite{Kunath2021} treats the case $n=1$ and -- for this case -- derives parts 1,2, and 3 of \cref{th:main1}.

\setsecnumdepth{1}

\section{Preliminaries} 

In this section we recall  structures, terminology, and conventions that will be used throughout this article.  
To start with, we recall that a connection on a principal $H$-bundle $E$ over a smooth manifold $M$ is a 1-form $\omega\in \Omega^1(E,\mathfrak{g})$ such that
\begin{equation*}
R^{*}\omega = \mathrm{Ad}_{h}^{-1}(p^{*}\omega) + h^{*}\theta
\end{equation*}
holds over $E \times H$, where $R$ denotes the principal action, $p$ the projection to $E$,  $h$ the projection to $H$, and $\theta$ is the left-invariant Maurer-Cartan form on $H$. If $H$ is abelian, we identify the curvature of $\omega$ with the unique 2-form $F \in \Omega^2(M,\mathfrak{h})$ such that $\pi^{*}F=\mathrm{d}\omega$, where $\pi:E \to M$ denotes the bundle projection.

We denote by $\trivlin := M \times H$ the trivial principal $H$-bundle over a smooth manifold $M$.
We may identify connections $\omega$ on $\trivlin$ with $\mathfrak{h}$-valued 1-forms $A\in \Omega^1(M,\mathfrak{h})$ in the usual way, i.e., 
\begin{equation}
\label{eq:conntrivbun}
\omega = \mathrm{Ad}_h^{-1}(p^{*}A) + h^{*}\theta
\end{equation}
where $p:M \times H \to M$ and $h:M \times H \to H$ are the projections. 
 We write $\trivlin_A$ for the trivial bundle equipped with the connection \cref{eq:conntrivbun}.  If $H$ is abelian, and $A_1,A_2\in \Omega^1(M,\mathfrak{h})$, there is a bijection
\begin{equation}
\label{eq:trivialbundleisos}
\bigset{10em}{Connection-preserving bundle isomorphisms $\varphi: \trivlin_{A_1} \to \trivlin_{A_2}$} \cong \bigset{9em}{Smooth maps $f:M \to H$ such that $A_1 = A_2 + f^{*}\theta$}
\end{equation} 
under which $\varphi(x,h) = (x,h+f(x))$. 
If $E$ is a principal $H$-bundle over $M$ with connection $\omega$, and $s_i: U_i \to E$ are local sections, then $\tau_i:\trivlin_{s_i^{*}\omega} \to E|_{U_i}: (x,a) \mapsto s_i(x)\cdot a$ is a connection-preserving bundle isomorphism.
On an overlap $U_i \cap U_j$, we consider the transition function $g_{ij}:U_i \cap U_j \to H$ defined by $s_i(x) = s_j(x)\cdot g_{ij}(x)$. 
In particular, we have $s_i^{*}\omega = s_j^{*}\omega + g_{ij}^{*}\theta$. 

\setsecnumdepth{2}

\subsection{Bundle gerbes with connection}

We use the definitions and conventions of \cite{waldorf1}. The reader familiar with bundle gerbes can safely skip this subsection. We write $\T:=\ueins=\R/\Z$. 

\begin{definition}
A \emph{bundle gerbe $\mathcal{G}$ with connection} over a smooth manifold $M$ consists of the following structure:
\begin{enumerate}

\item 
A surjective submersion $\pi:Y \to M$, and a 2-form $B \in \Omega^2(Y)$ called \quot{curving}.

\item
A principal $\T$-bundle $P$ with connection over the double fibre product $Y^{[2]}=Y \times_M Y$, whose curvature is $F_P = \pr_2^{*}B-\pr_1^{*}B$. 

\item
A connection-preserving bundle isomorphism $\mu: \pr_{12}^{*}P \otimes \pr_{23}^{*}P \to \pr_{13}^{*}P$ over $Y^{[3]}$, called \quot{bundle gerbe product}.

\end{enumerate}
It it required that over a point $(y_1,y_2,y_3,y_4) \in Y^{[4]}$ the following associativity condition holds:
\begin{equation*}
\alxydim{@C=6em}{P_{y_1,y_2} \otimes P_{y_2,y_3} \otimes P_{y_3,y_4} \ar[r]^-{\mu_{y_1,y_2,y_3} \otimes \id} \ar[d]_{\id \otimes \mu_{y_2,y_3,y_4}} & P_{y_1,y_3} \otimes P_{y_3,y_4} \ar[d]^{\mu_{y_1,y_3,y_4}} \\ P_{y_1,y_2} \otimes P_{y_2,y_4} \ar[r]_-{\mu_{y_1,y_2,y_4}} &  P_{y_1,y_4}}
\end{equation*}
The \emph{curvature} of $\mathcal{G}$ is the unique 3-form $H\in \Omega^3(M)$ such that $\pi^{*}H=\mathrm{d}B$.
\end{definition}

If $B \in \Omega^2(M)$ is a 2-form, then there is a \quot{trivial} bundle gerbe with connection $\mathcal{I}_B$, with surjective submersion $\pi=\id_M$, the trivial $\T$-bundle with connection $P=\trivlin_0$, and the trivial bundle isomorphism $\trivlin_0 \otimes \trivlin_0 \cong \trivlin_0$. Its curvature is $H=\mathrm{d}B$.

If $\{U_i\}_{i\in I}$ is an open cover of $M$ admitting smooth local sections $s_i: U_i \to Y$, then we may define the 2-forms $B_i := s_i^{*}B\in \Omega^2(U_i)$. If we further assume that the non-empty double intersections $U_i \cap U_j$ are contractible, we may choose sections $s_{ij}:U_i \cap U_j \to (s_i,s_j)^{*}P$, inducing  1-forms $A_{ij}\in \Omega^1(U_i\cap U_j)$ satisfying $B_j = B_i + \mathrm{d}A_{ij}$. Next, there exists a unique smooth map $c_{ijk}: U_i \cap U_j \cap U_k \to \T$ such that
\begin{equation*}
\mu(s_{ij}(x) \otimes s_{jk}(x)) \cdot c_{ijk}(x)=s_{ik}(x)\text{.}
\end{equation*}
This implies an equality $A_{ik}=A_{ij}+A_{jk}+c_{ijk}^{*}\theta$. 
Finally, the associativity condition for $\mu$ implies a \v Cech cocycle condition $c_{jkl}\cdot c_{ijl} = c_{ijk}\cdot c_{ikl}$. The \quot{local data} $(B_i,A_{ij},c_{ijk})$ yield a degree-2-cocycle in Deligne cohomology, and thus represent a class in degree three differential cohomology $\hat\h^3(M)$ of $M$.

It will be important to consider the full bicategorical structure of bundle gerbes with connection. 

\begin{definition}
Suppose $\mathcal{G}$ and $\mathcal{G}'$ are bundle gerbes with connection. A \emph{connection-preserving isomorphism} $\mathcal{A}:\mathcal{G} \to \mathcal{G}'$ consists of the following structure:
\begin{enumerate}

\item 
A surjective submersion $\zeta: Z \to Y \times_M Y'$, and a principal $\T$-bundle $Q$ with connection over $Z$ whose curvature is $F_Q=\zeta^{*}\pr_{Y'}^{*}B' - \zeta^{*}\pr_Y^{*}B$.

\item
A connection-preserving bundle isomorphism $\chi: (\xi^{[2]})^{*}P \otimes \pr_2^{*}Q \to \pr_1^{*}Q \otimes (\xi'^{[2]})^{*}P'$ over the double fibre product $Z^{[2]}=Z \times_M Z$, where $\xi := \pr_Y\circ \zeta: Z \to Y$ and $\xi':= \pr_{Y'} \circ \zeta: Z \to Y'$, and $\xi^{[2]}$ and $\xi'^{[2]}$ denote the induced maps on double fibre products. 

\end{enumerate}
It is required that the following compatibility condition holds for all $(z_1,z_2,z_3)\in Z^{[3]}$, for which we set $\zeta(z_i)=: (y_i,y_i')$:
\begin{equation*}
\alxydim{@C=10em}{P_{y_1,y_2} \otimes P_{y_2,y_3} \otimes Q_{z_3} \ar[r]^-{\mu_{y_1,y_2,y_3} \otimes \id} \ar[d]_{\id \otimes \chi_{z_2,z_3}} & P_{y_1,y_3} \otimes Q_{z_3} \ar[dd]^{\chi_{z_1,z_3}} \\ P_{y_1,y_2} \otimes Q_{z_2} \otimes P'_{y_2',y_3'} \ar[d]_{\chi_{z_1,z_2} \otimes \id}  \\ Q_{z_1} \otimes P'_{y_1',y_2'} \otimes P'_{y_2',y_3'} \ar[r]_-{\id \otimes \mu'_{y_1',y_2',y_3'}} & Q_{z_1} \otimes P'_{y_1',y_3'}}
\end{equation*}
\end{definition}

We remark that the curvature of $\mathcal{G}$ and $\mathcal{G}'$ coincide if there exists a connection-preserving isomorphism. The set of isomorphism classes of bundle gerbes with connections over $M$ is denoted by $\mathrm{Grb}^{\nabla}(M)$.
This set is actually a group, whose multiplication is given by the tensor product of bundle gerbes, see \cite{waldorf1}.

Suppose we have chosen sections $s_{i}$ and $s_{ij}$ for $\mathcal{G}$ as above, and similar sections $s_i'$ and $s'_{ij}$ for $\mathcal{G}'$, with corresponding local data $(B_i,A_{ij},c_{ijk})$ and $(B_i',A'_{ij},c'_{ijk})$.
After a further refinement, we may assume that $(s_{i},s_i'):U_i  \to Y \times_M Y'$ lifts to $Z$, i.e., we may choose $t_i: U_i \to Z$ such that $\zeta\circ t_i=(s_i,s_i')$. We may then assume that $t_i^{*}Q$ admits a local section $u_i$, with corresponding  1-forms $C_i$. Note that $B_i'=B_i+\mathrm{d}C_i$.
There exists a unique smooth map $d_{ij}:U_i \cap U_j \to \T$ such that
\begin{equation*}
\chi(s_{ij}(x)\otimes u_j(x))\cdot d_{ij}(x) = u_i(x) \otimes s'_{ij}(x)\text{.}
\end{equation*}
This implies an equality
\begin{equation*}
  A'_{ij} = A_{ij}+C_j-C_i+d_{ij}^{*}\theta\text{.} 
\end{equation*}
Finally, the compatibility condition yields an equality
\begin{equation*}
d_{ik}\cdot c_{ijk}=d_{ij}\cdot d_{jk}\cdot c_{ijk}'\text{.}
\end{equation*}
The data $(C_i,d_{ij})$ constitute an equivalence between the Deligne 2-cocycles $(B_i,A_{ij},c_{ijk})$ and $(B'_i,A'_{ij},c'_{ijk})$. 
This establishes an isomorphism $\mathrm{Grb}^{\nabla}(M) \cong \hat \h^3(M)$ between the set of isomorphism classes of  bundle gerbes with connection and degree three differential cohomology \cite{murray2,stevenson1}.

\begin{definition}
Suppose $\mathcal{G}$ and $\mathcal{G}'$ are bundle gerbes with connection, and suppose that $\mathcal{A}_1,\mathcal{A}_2: \mathcal{G} \to \mathcal{G}'$ are connection-preserving isomorphisms. A \emph{connection-preserving 2-isomorphism} $\eta: \mathcal{A}_1 \Rightarrow \mathcal{A}_2$ is an equivalence class of triples $(W,\omega,\eta)$, where $\omega:W \to Z_1 \ttimes{\zeta_1}{\zeta_2} Z_2$ is a surjective submersion, and $\eta: \omega^{*}\pr_{Z_1}^{*}Q_1 \to \omega^{*}\pr_{Z_2}^{*}Q_2$ is a connection-preserving bundle isomorphism. It is required that for all $(w,w')\in W \times_M W$ the following diagram is commutative:
\begin{equation*}
\alxydim{@C=5em}{P_{y_1,y_2} \otimes Q_1|_{z_1'} \ar[d]_{\id \otimes \eta_{w'}} \ar[r]^-{\chi_1|_{z_1,z_1'}} & Q_1|_{z_1} \otimes P'_{y_1',y_2'} \ar[d]^{\eta_{w} \otimes \id} \\ P_{y_1,y_2} \otimes Q_2|_{z_2'} \ar[r]_-{\chi_2|_{z_2,z_2'}} & Q_2|_{z_2} \otimes P'_{y_1',y_2'}\text{;}}
\end{equation*} 
where $\omega(w)=:(z_1,z_2)$ and $\omega(w')=:(z_1',z_2')$, as well as $\zeta_i(z_i)=:(y_1,y_1')$ and $\zeta_i(z_i')=:(y_2,y_2')$. Two triples are equivalent if their bundle isomorphisms coincide when pulled back to a common refinement. 
\end{definition}

Concerning local data, we may assume that  the sections $t_{1,i}:U_i \to Z_1$ and $t_{2,i}: U_i \to Z_2$ lift to $W$, i.e., that there  are sections $v_i: U_i \to W$ such that $\omega\circ v_i=(t_{1,i},t_{2,i})$. Then, $v_i^{*}\eta: t_{1,i}^{*}Q_1 \to t_{2,i}^{*}Q_2$ is a connection-preserving bundle isomorphism, and there exists a unique smooth map $z_i: U_i \to \T$ such that $v_i^{*}\eta(u_{1,i}(x))\cdot z_i(x)=u_{2,i}(x)$. This yields an equality $C_{2,i} = C_{1,i} + z_i^{*}\theta$. The diagram leads to $d_{1,ij} \cdot z_i = z_j \cdot d_{2,ij}$. 

The (vertical) composition of connection-preserving 2-isomorphisms is obtained by going to a common refinement and composing the bundle isomorphisms there. This way, we obtain a category $\hom(\mathcal{G},\mathcal{G}')$. There is a (horizontal) composition functor
\begin{equation*}
\hom(\mathcal{G}'',\mathcal{G}') \times \hom(\mathcal{G}',\mathcal{G}) \to \hom(\mathcal{G},\mathcal{G}'')
\end{equation*} 
which turns bundle gerbes with connection into a bicategory. The following statement about the morphism category between trivial bundle gerbes will be very important later.

\begin{proposition}{{\normalfont\cite[Prop. 4]{waldorf1}}}
\label{lem:gerbehombundle}
There is a canonical equivalence of categories
\begin{equation*}
\hom(\mathcal{I}_{B_1},\mathcal{I}_{B_2}) \cong \buncon{\T}X^{B_2-B_1}\text{,}
\end{equation*}
where the right hand side denotes the category of principal $\T$-bundles with connection of fixed curvature $F=B_2-B_1$.
Under this equivalence, the composition of connection-preserving isomorphisms corresponds to the tensor product of bundles with connection.  
\end{proposition}

\subsection{Poincaré bundles and equivariance}

\label{sec:poincare}

We summarize some required facts about the Poincaré bundle, also see \cite[Appendix B]{Nikolause}. We write $\T^{n} := \R^{n}/\Z^{n}$ additively, and identify its Lie algebra with $\R^{n}$, and again  $\T=\T^1=\R/\Z$.
The $n$-fold Poincaré bundle is the following principal $\T$-bundle $\poi$ over $\T^{2n}=\T^{n} \times \T^{n}$. Its total space is
\begin{equation*}
\poi := (\R^{n} \times \R^{n} \times \T) \;/\; \sim
\end{equation*}
with $(a,\hat a,t)\sim (a+m,\hat a+\hat m,m\hat a+t)$ for all $m,\hat m\in \Z^{n}$ and $t\in \T$, and $m\hat a$ is the standard inner product.  The bundle projection is $(a,\hat a,t) \mapsto (a,\hat a)$, and the $\T$-action is $(a,\hat a,t)\cdot s := (a,\hat a,t+s)$.  

For   maps $\T^{p} \to \T^q$ between different tori we use a notation
\begin{equation*}
m_{1+2,-3,5}: \T^{5} \to \T^{3},\quad(a_1,a_2,a_3,a_4,a_5) \mapsto (a_1+a_2,-a_3,a_5)\text{.}
\end{equation*}
For pullbacks we then write $(...)_{1+2,-3,5}$ instead of $m_{1+2,-3,5}^{*}(...)$. 
The following structures and properties are straightforward to check.
\begin{enumerate}[(a),leftmargin=2em]

\item
The following maps are well-defined 
bundle isomorphisms over $\T^{3n}$:
\begin{eqnarray*}
\varphi_l:\poi_{1,3} \otimes \poi_{2,3} \to \poi_{1+2,3} &:& (a,c,t) \otimes (b,c',s) \mapsto  (a+b,c,s+t)
\\
\varphi_r:\poi_{1,2} \otimes \poi_{1,3} \to \poi_{1,2+3}&:& (a,b,t) \otimes (a',c,s) \mapsto  (a,b+c,(a-a')c+s+t)
\end{eqnarray*}
They express that the Poincaré bundle is \quot{bilinear} in the two factors $\T^{n} \times \T^{n}$.
Using the given formulas, one can check that $\varphi_l$ satisfies the following associativity condition:
\begin{equation}
\label{eq:varphiass}
\alxydim{@C=3em}{\poi_{1,4} \otimes \poi_{2,4} \otimes \poi_{3,4} \ar[r]^-{\id \otimes \varphi_l} \ar[d]_{\varphi_l \otimes \id} & \poi_{1,4} \otimes \poi_{2+3,4} \ar[d]^{\varphi_l} \\ \poi_{1+2,4} \otimes \poi_{3,4} \ar[r]_-{\varphi_l} & \poi_{1+2+3,4}\text{.}}
\end{equation}
An analogous condition holds for $\varphi_r$. Another compatibility condition that one can easily check is the commutativity of the following pentagon diagram:
\begin{equation}
\label{eq:komplr}
\alxydim{@C=-1cm}{&&\poi_{1,3} \otimes \poi_{1,4} \otimes \poi_{2,3} \otimes \poi_{2,4} \ar[rrd]^-{\varphi_r \otimes \varphi_r} \ar[dll]_{\id \otimes \,braid\, \otimes \id\quad} \\  \mquad\poi_{1,3} \otimes \poi_{2,3} \otimes \poi_{1,4} \otimes \poi_{2,4}\mqquad \ar[dr]_{\varphi_l \otimes \varphi_l} &&&& \qquad \poi_{1,3+4} \otimes \poi_{2,3+4}\quad \ar[dl]^{\varphi_l} \\ & \poi_{1+2,3} \otimes \poi_{1+2,4} \ar[rr]_-{\varphi_r} && \poi_{1+2,3+4}\text{.}}
\end{equation}

\item 
The map 
\begin{equation}
\label{eq:chil}
\chi_l: \R^{n} \times \T^{n} \to \poi: (a,\hat a) \mapsto (a,\hat a,0)
\end{equation}
is a well-defined section along $\R^{n} \times \T^{n} \to \T^{n} \times \T^{n}$, and the map
\begin{equation*}
\chi_r: \T^{n} \times \R^{n} \to \poi: (a,\hat a) \mapsto (a,\hat a,a\hat a)
\end{equation*}
is a well-defined section along  $\T^{n} \times \R^{n} \to \T^{n} \times \T^{n}$. 
These restrict further to sections along the inclusion  $\T^{n} \to \T^{n} \times \T^{n}$ into one of the factors.
The transition functions w.r.t. these sections are the following. Suppose $a,a'\in \R^{n}$ such that $m:=a'-a\in \Z^{n}$ and $\hat a \in \T^{n}$. Then, 
\begin{equation}
\label{eq:transitionfunctionleftsection}
\chi_l(a,\hat a)  = \chi_l(a',\hat a)\cdot m\hat a\text{.} 
\end{equation}
If $a\in \T^{n}$ and $\hat a, \hat a'\in \R^{n}$ with $\hat m:=\hat a'-\hat a\in \Z^{n}$, then we have
\begin{equation*}
\chi_r(a,\hat a)=\chi_r(a,\hat a')\cdot a\hat m\text{.} 
\end{equation*}

Over $\R^{n} \times \R^{n}$ the sections $\chi_l$ and $\chi_r$ do \emph{not} coincide, but differ by the $\T$-valued function $m:\R^{n} \times \R^{n} \to \T$, $m(a,\hat a) := a\hat a$, i.e., $\chi_r = \chi_l \cdot m$.
They do coincide when pulled back to $\Z^{n} \times \Z^{n}$.

\item
We recall that the dual $\poi^{\vee}$ of a principal $\T$-bundle $\poi$ has the same total space but $\T$ acting through inverses. 
The map
\begin{equation*}
\lambda:\poi \to \poi_{2,1}^{\vee}:(a,\hat a,t) \mapsto (\hat a,a,a\hat a-t)
\end{equation*}
is a well-defined bundle isomorphism over the identity of $\T^{n} \times \T^{n}$. It expresses that the Poincaré bundle is \quot{skew-symmetric}. 
We have $\lambda^2=\id$.
Moreover, the isomorphism $\lambda$ exchanges $\chi_l$ with $\chi_r$.

\item
The Poincaré bundle $\poi$ carries a canonical connection $\omega$, which descends from the 1-form $\tilde \omega\in \Omega^1(\R^{n} \times \R^{n} \times \T)$ defined by $\tilde\omega := -a \mathrm{d}\hat a+\mathrm{d}t$. 
It is straightforward to see that the isomorphisms $\varphi_l$ and $\varphi_r$ are connection-preserving.
Moreover, the sections $\chi_l$ and $\chi_r$ have covariant derivatives 
\begin{equation}
\label{eq:covderpoi}
\chi_l^{*}\omega =-a\mathrm{d}\hat a
\quand 
\chi_r^{*}\omega = -a\mathrm{d}\hat a+\mathrm{d}(a\hat a)=\hat a\mathrm{d}a\text{.}
\end{equation}

In other words, they establish trivializations $\chi_l^{*}\poi \cong \trivlin_{-a\mathrm{d}\hat a}$ over $\R^{n} \times \T^{n}$ and $\chi_r^{*}\poi\cong \trivlin_{\hat a\mathrm{d}a}$ over $\T^{n} \times \R^{n}$. 
\end{enumerate}

\begin{remark}
\label{re:curvpoi}
\begin{enumerate}[(i)]

\item 
The curvature of $\omega$ is 
\begin{equation*}
\Omega :=\theta_2 \dot \wedge \theta_1 \in \Omega^2(\T^{2n})\text{,}
\end{equation*}
where $\theta\in\Omega^1(\T^{n},\R^{n})$ is the Maurer-Cartan form, and $\dot\wedge$ denotes the wedge product of $\R^{n}$-valued forms using the standard inner product of $\R^{n}$ on the values. 

\item
\label{lem:propOmega}
\label{lem:propOmega:a}
Note that $\Omega_{1,3}+\Omega_{2,3}=\Omega_{1+2,3}$ and $\Omega_{1,2}+\Omega_{1,3}=\Omega_{1,2+3}$, as well as $\Omega_{1,-2}=-\Omega_{1,2}=\Omega_{-1,2}$.

\item
An identity expressing the 2-form $\Omega$ in terms of the Maurer-Cartan form on $\T$ is
\begin{equation*}
\Omega = \sum_{i=1}^{n} \pr_{i+n}^{*}\theta \wedge \pr_{i}^{*}\theta\in \Omega^2(\T^{2n})\text{.}
\end{equation*}
Since $\h^{*}(\T^{2n},\Z)$ is torsion free, this shows  that the first Chern class of $\poi$ is 
\begin{equation*}
\sum_{i=1}^{n}\pr_{i+n} \cup \pr_i \in \h^2(\T^{2n},\Z)
\end{equation*}
where $\pr_i: \T^{2n} \to \T$ is regarded as a representative for $[\T^{2n},\T]=\h^1(\T^{2n},\Z)$. 

\end{enumerate}
\end{remark}

The following discussion concerns the quite difficult equivariance properties of the Poincaré  bundle and its connection. They will be used only in \cref{sec:localformalism}, so that the reader may also continue  with \cref{sec:buscher,sec:gtd} first.

We remark  that the curvature form $\Omega$ is $\T^{2n}$-invariant. However, the Poincaré bundle itself is \emph{not} equivariant with respect to left or right multiplication of $\T^{2n}$ on itself. We construct a connection-preserving isomorphism
\begin{equation}
\label{eq:isorl}
R^{l}: \poi_{1+3,2+4} \to \poi_{3,4} \otimes \trivlin_{\psi^{l}}
\end{equation}
over $\R^{2n} \times \T^{2n}=\R^{n} \times \R^{n} \times \T^{n} \times \T^n$ in the following way, where $\psi^{l}\in \Omega^1(\R^{2n}\times\T^{2n})$ is given at a point $((x,\hat x),(a,\hat a))$ by
\begin{equation}
\label{eq:psi}
\psi^{l} := -x\mathrm{d}\hat x - x\mathrm{d}\hat a+\hat x\mathrm{d}a\text{.}
\end{equation}
Indeed, we have, using $\varphi_l$ and $\varphi_r$, an isomorphism
\begin{equation*}
\poi_{1+3,2+4} \cong \poi_{1,2} \otimes \poi_{1,4} \otimes \poi_{3,2} \otimes \poi_{3,4}
\end{equation*}
and due to \cref{eq:komplr} it does not matter how in which order these are used. In the next step we use the sections $\chi_l$ and $\chi_r$, and now we see that for the tensor factor $\poi_{1,2}$ over $\R^{2n}$ it does matter how to trivialize it. Using $\chi_l$, we obtain
\begin{equation*}
\poi_{x,\hat x} \otimes \poi_{x,\hat a} \otimes \poi_{a,\hat x} \otimes \poi_{a,\hat a}\cong \trivlin_{-x\mathrm{d}\hat x - x\mathrm{d}\hat a+\hat x\mathrm{d}a} \otimes \poi_{a,\hat a}= \trivlin_{\psi^{l}}|_{(x,\hat x),(a,\hat a)} \otimes \poi_{a,\hat a}\text{,} 
\end{equation*}
all together resulting in the isomorphism \cref{eq:isorl}. Explicitly, 
\begin{equation}
\label{eq:Rexplicit}
R^{l}_{(x,\hat x),(a,\hat a)}(x+a,\hat x+\hat a,t)= (a,\hat a,t-a\hat x)\text{.}
\end{equation}
One might be tempted to assume that $R^{l}$ establishes an action of $\R^{2n}$ on $\poi$ covering the action of $\R^{2n}$ on $\T^{2n}$, but this is not true. Instead,  \quot{acting twice} gives the formula
\begin{equation}
\label{eq:poieq:1}
R^{l}_{(x_1,\hat x_1),(a,\hat a)} \circ R^{l}_{(x_2,\hat x_2),(x_1+a,\hat x_1+\hat a)} =R^{l}_{(x_2+x_1,\hat x_2+\hat x_1),(a,\hat a)} \cdot (x_1\hat x_2)^{-1} 
\end{equation}
exhibiting an error term which should no be present when $R^{l}$ were an action.
If one uses $\chi_r$ instead of $\chi_l$, one obtains an isomorphism
\begin{equation}
\label{eq:isorr}
R^{r}: \poi_{1+3,2+4} \to \trivlin_{\psi^{r}} \otimes \poi_{3,4} 
\end{equation}
for
\begin{equation*}
\psi^{r} := \hat x\mathrm{d} x - x\mathrm{d}\hat a+\hat x\mathrm{d}a= \psi^{l} + \mathrm{d}(x \hat x)\text{.}
\end{equation*}

\begin{remark}
Restricting to $\Z^{n}\subset \R^{n}$, we have 
\begin{equation*}
\psi^{l}_{(m,\hat m),(a,\hat a)} =\psi^{r}_{(m,\hat m),(a,\hat a)}
\quand
R^{l}_{(m,\hat m),(a,\hat a)}=R^{r}_{(m,\hat m),(a,\hat a)}
\end{equation*}
the latter corresponding to the automorphism of $\poi$ given by multiplication with the map
\begin{equation*}
f_{m,\hat m}:\T^{2n} \to \T: (a, \hat a) \mapsto m\hat a-\hat m a\text{.}
\end{equation*}  
In particular, we obtain from \cref{eq:poieq:1} 
\begin{equation}
\label{eq:poieq:2}
R^{l}_{(x+m,\hat x+\hat m),(a,\hat a)}= R^{l}_{(x,\hat x),(m+a,\hat m+\hat a)} \cdot f_{m,\hat m}\cdot (m\hat x)\text{.}
\end{equation}
\end{remark}

\subsection{Invariant metrics on principal bundles}

We review the mathematical basis of Kaluza-Klein theory, summarized in the following \cref{th:kaluzaklein}. Its relevance for T-duality has been recognized already in Buscher's first paper \cite{Buscher1987} on the subject. General discussions can be found in \cite[\S 9.3]{bleecker1}, \cite{Coquereaux1988} or \cite[\S 3.4]{Coquereaux1983}; the latter reference contains a complete proof.

\begin{theorem}
\label{th:kaluzaklein}
Suppose that $E$ is a principal $H$-bundle over $X$.  Then, there is a bijection
\begin{equation*}
\bigset{7em}{$H$-invariant Riemannian metrics $g$ on $E$} 
\cong
\bigset{25em}{Triples $(\omega,g',h)$ consisting of a connection $\omega$ on $E$, a Riemannian metric $g'$ on $X$, and a smooth family $h_x$ of $\mathrm{Ad}$-invariant inner products on $\mathfrak{h}$ parameterized by $x\in X$}
\end{equation*}
under which $g$ corresponds to $(\omega,g',h)$ if and only if
\begin{equation}
\label{def:tildeg}
g_e = \begin{pmatrix}
g'_{x} & 0 \\
0 & h_x \\
\end{pmatrix}\text{,}
\end{equation}
where $e\in E$ sits in the fibre over $x\in X$, and the matrix  on the right hand side refers to the  decomposition $T_eE \cong T_xM \oplus \mathfrak{h}$  induced by the connection $\omega$.   
\end{theorem}

The connection $\omega$ that appears on the right hand side will be called the \emph{Kaluza-Klein connection} associated to the metric $g$.

\begin{remark}
\label{re:isometries}
\cref{th:kaluzaklein} is compatible with bundle isomorphisms:
for a bundle isomorphism $\varphi:E_1 \to E_2$ it is equivalent to be isometric for the metrics on $E_1$ and $E_2$ or to be connection-preserving w.r.t. the Kaluza-Klein connections  $\omega_1$ and $\omega_2$.
\end{remark}

\begin{remark}
\label{rem:kaluzakleintrivialbundle}
If the principal $H$-bundle is trivial, $E=\trivlin= X \times H$, we  identify connections $\omega$ on $E$ with $\mathfrak{h}$-valued 1-forms $A\in \Omega^1(X,\mathfrak{h})$ as in \cref{eq:conntrivbun}. Under this identification, a metric $g$ corresponds to a triple $(A,g',h)$ under the bijection of  \cref{th:kaluzaklein} if and only if
\begin{equation}
\label{eq:metricandlocalconnection}
g_{x,h}  = \begin{pmatrix}
g' + h_x(\mathrm{Ad}_h^{-1}(A_x(-)),\mathrm{Ad}_h^{-1}(A_x(-))) & h_x(\mathrm{Ad}_h^{-1}(A_x(-)),-) \\
h_x(-,\mathrm{Ad}_h^{-1}(A_x(-))) & h_x(-,-) \\
\end{pmatrix}\text{,}
\end{equation}
where now the  decomposition on the right hand side refers to the equality $T_{(x,h)}E = T_xM \oplus \mathfrak{h}$ induced by the direct product structure of $E$ and the identification $T_    hH \cong \mathfrak{h}$ via left multiplication, see \cite[p. 101]{Coquereaux1988}. 
\end{remark}

\begin{remark}
\label{re:kaluzakleinforabelian}
In later applications of \cref{th:kaluzaklein}, $H$ will  be a torus, $H=\T^{n}$. In particular, $H$ is abelian. In this case, the $\mathrm{Ad}$-invariance in \cref{th:kaluzaklein} is vacuous. Moreover,  \cref{eq:metricandlocalconnection} reduces to
\begin{equation*}
g_{x,h}  = \begin{pmatrix}
g'_x + h_x(A_x(-),A_x(-)) & h_x(A_x(-),-) \\
h_x(-,A_x(-)) & h_x(-,-) \\
\end{pmatrix}\text{.}
\end{equation*}
In the literature, this is sometimes written as 
\begin{equation*}
g = g' + A \odot A\text{,}
\end{equation*}
where, unfortunately, $h_x$ is suppressed or assumed to be constant.  
\end{remark}

\section{Buscher rules revisited}

\label{sec:buscher}

The Buscher rules formulate the local behaviour of metrics and B-fields under T-duality. A priori, they only  apply to trivial torus bundles over Euclidean space, for instance, over a coordinate patch. We first review the Buscher rules for higher tori, give then a reformulation when the metrics are replaced by their Kaluza-Klein connections, and finally produce  a completely coordinate-free reformulation. It is this latter formulation that we generalize in \cref{sec:gtd}. 
 
\subsection{Buscher rules for toroidal symmetries}

\label{sec:buschern}

We first recall the classical Buscher rules for a $\T^n$-symmetry. The case $n=1$ has been treated by Buscher in \cite{Buscher1987}. Rules for higher $n$ are described, e.g., in \cite{Giveon1994}, also see \cite{Bouwknegt2010} for a review. 

The Buscher rules apply to a manifold $E=\R^s \times \T^n$. On $E$ we consider a Riemannian $\T^n$-invariant metric $g$ and a $\T^n$-invariant 2-form $B$, the \quot{B-field}.
 With respect to the standard coordinates, we may identify $g$ with a block matrix
\begin{equation*}
g= \begin{pmatrix}
g_{bas} & g_{mix} \\
g_{mix}^{\mathrm{tr}} & g_{fib} \\
\end{pmatrix}
\end{equation*}
with symmetric matrices $g_{bas}\in C^{\infty}(\R^{s})^{s \times s}$ and $g_{fib}\in C^{\infty}(\R^{s})^{n \times n}$, and an arbitrary matrix $g_{mix}\in C^{\infty}(\R^{s})^{s \times n}$. 
We also identity $B$ with a block matrix 
\begin{equation*}
B=\begin{pmatrix}
B_{bas} & B_{mix} \\
-B_{mix}^{\mathrm{tr}} & B_{fib} \\
\end{pmatrix}
\end{equation*}  
where $B_{bas}\in C^{\infty}(\R^s)^{s\times s}$ and $B_{fib}\in C^{\infty}(\R^{s})^{n \times n}$ are skew-symmetric and $B_{mix}\in C^{\infty}(\R^s)^{s \times n}$ is arbitrary. We make the assumption that $B_{fib}=0$, so that the B-field must not have components purely in fibre direction. Note that for $n=1$ this is automatic. Pairs $(g,B)$ with the condition $B_{fib}=0$ will be called \emph{Buscher pairs}.

To proceed, we form the \quot{background} matrix 
\begin{equation*}
Q = \begin{pmatrix}
Q_{bas} & Q_{mix} \\
Q_{mix}' & Q_{fib} \\
\end{pmatrix}  :=  g+B\text{.}
\end{equation*}
The dualization process requires to form the \quot{dual} matrix
\begin{equation*}
\hat Q := \begin{pmatrix}
Q_{bas} - Q_{mix}Q_{fib}^{-1}Q'_{mix} & Q_{mix}Q_{fib}^{-1} \\
-Q_{fib}^{-1}Q'_{mix} & Q_{fib}^{-1} \\
\end{pmatrix}\text{.}
\end{equation*} 
Dual metric and B-field are now obtained by taking the symmetric and anti-symmetric parts of $\hat Q$, respectively, i.e.
\begin{equation*}
\hat g := \frac{1}{2}(\hat Q + \hat Q^{\mathrm{tr}})
\quand
\hat B := \frac{1}{2}(\hat Q - \hat Q^{\mathrm{tr}})\text{.}
\end{equation*}
A standard calculation shows that $(\hat g,\hat B)$ is again a Buscher pair, and that the relation between the Buscher pairs $(g,B)$ and $(\hat g,\hat B)$ is described by the following equations: 
\begin{align}
\label{eq:BR:1}
\hat g_{bas} &= g_{bas} -g_{mix} g_{fib}^{-1}g_{mix}^{\mathrm{tr}} + B_{mix}g_{fib}^{-1}B_{mix}^{\mathrm{tr}} \\ 
\label{eq:BR:2}
\hat g_{mix} &= B_{mix}g_{fib}^{-1}
\\
\label{eq:BR:3}
\hat g_{fib} &= g_{fib}^{-1}
\\
\label{eq:BR:4}
\hat B_{bas} &= B_{bas} - B_{mix}g_{fib}^{-1}g_{mix}^{\mathrm{tr}}+g_{mix} g_{fib}^{-1}B_{mix}^{\mathrm{tr}}
\\
\label{eq:BR:5}
\hat B_{mix} &= g_{mix}g_{fib}^{-1}
\end{align}
It is straightforward to see that these rules reduce in the case of $n=1$ to the usual Buscher rules.
It is also straightforward to see that $\hat {\hat Q} = Q$, implying that the Buscher rules are symmetric in the data. For completeness, let us fix the following definition. 

\begin{definition}
\label{def:BR}
Two Buscher pairs $(g,B)$ and $(\hat g,\hat B)$ \emph{satisfy the Buscher rules} if   \cref{eq:BR:1,eq:BR:2,eq:BR:3,eq:BR:4,eq:BR:5} are satisfied. 
\end{definition}

\subsection{Buscher rules in terms of Kaluza-Klein connections}

We consider again a metric $g$ on $E=\R^{s} \times \T^{n}$, and consider $E$ as a principal $\T^{n}$-bundle over $\R^{s}$. We apply \cref{th:kaluzaklein,rem:kaluzakleintrivialbundle}, to obtain a triple $(A,g',h)$ consisting of a Riemannian metric $g'$ on $\R^{s}$, a 1-form $A\in \Omega^1(\R^{s},\R^{n})$, and a family $h$ of inner products on $\R^{n}$ parameterized by $\R^{s}$. Now we consider \emph{Buscher quadruples} $(A,g',h,B)$ instead of Buscher pairs $(g,B)$. By \cref{th:kaluzaklein} there is a bijection between Buscher quadruples and Buscher pairs. 

The expression for the metric $g$ given in \cref{re:kaluzakleinforabelian} now reads
\begin{equation*}
g  = \begin{pmatrix}
g'+ A^{\mathrm{tr}}hA & A^{\mathrm{tr}}h \\
hA & h \\
\end{pmatrix}\text{.}
\end{equation*}
In other words, we have
\begin{align*}
g_{bas} = g' + A^{\mathrm{tr}}hA
\quomma
g_{mix}  = A^{\mathrm{tr}}h
\quand
g_{fib} = h\text{.}
\end{align*}
We employ the same procedure on the dual side, getting 
\begin{align*}
\hat g_{bas} = \hat g' + \hat A^{\mathrm{tr}}\hat h\hat A
\quomma
\hat g_{mix} = \hat A^{\mathrm{tr}}\hat h
\quand
\hat g_{fib} = \hat h\text{.}
\end{align*}
The Buscher rules now attain the following  simple form:
\begin{align}
\label{eq:BRq:1}
\hat g' &= g'  
\\ 
\label{eq:BRq:2}
\hat A^{\mathrm{tr}}&= B_{mix}
\\
\label{eq:BRq:3}
\hat h&= h^{-1}
\\
\label{eq:BRq:4}
\hat B_{bas} &= B_{bas} - B_{mix}A+(B_{mix}A)^{\mathrm{tr}}
\\
\label{eq:BRq:5}
\hat B_{mix} &= A^{\mathrm{tr}}
\end{align}

Again for completeness, we fix the following definition and result. 

\begin{definition}
\label{def:BRq}
Two Buscher quadruples $(A,g',h,B)$ and $(\hat A,\hat g',\hat h,\hat B)$ \emph{satisfy the Buscher rules}, if the  \cref{eq:BRq:1,eq:BRq:2,eq:BRq:3,eq:BRq:4,eq:BRq:5} are satisfied. 
\end{definition}

\begin{lemma}
\label{lem:br:equiv}
Under the bijection between Buscher pairs and Buscher quadruples, the Buscher rules of \cref{def:BR,def:BRq} are equivalent.
\end{lemma}

\subsection{Buscher rules in terms of Poincaré forms}

Next we want to give a coordinate-independent description of the Buscher rules of \cref{def:BRq}, which will again make them simpler. Let $\omega,\hat\omega \in \Omega^1(\R^{s} \times \T^{n},\R^{n})$ be the Kaluza-Klein connections  on $E=\R^{s}\times \T^{n}$ corresponding to $A$ and $\hat A$, respectively, i.e., $\omega := A_1 + \theta_2$  and $\hat \omega := \hat A_1 + \theta_2$. Here, the indices refer to the pullback along the projections to the two factors, as explained in \cref{sec:poincare}. We introduce the 2-form
\begin{equation*}
\rho := \hat\pr^{*}\hat\omega \dot\wedge \pr^{*}\omega \in \Omega^2(\R^{s} \times \T^{2n})\text{,}
\end{equation*}
where the symbol $\dot\wedge$ means that the standard scalar product on $\R^{n}$ is used in the values the forms.

\begin{lemma}
\label{lem:buscherandpoincare}
Buscher quadruples $(g',A,h,B)$ and $(\hat g',\hat A,\hat h,\hat B)$ satisfy the Buscher rules of \cref{def:BRq} if and only if the following conditions are satisfied:
\begin{enumerate}[\normalfont(a)]

\item 
$\hat g'=g'$

\item
$\hat h = h^{-1}$

\item
\label{lem:buscherandpoincare:c}
$\hat\pr^{*}\hat B - \pr^{*}B  =  \pr_{\T^{2n}}^{*}\Omega-\rho$.

\end{enumerate}
\end{lemma}

\begin{proof}
(a) and (b) are \cref{eq:BRq:1,eq:BRq:3}. We have
\begin{equation*}
-\rho_{x,t,\hat t} = \omega_{x,t} \dot \wedge \hat\omega_{x,\hat t} =( A_{x} + \theta_t^{\T^{n}})\dot\wedge ( \hat A_{x} + \theta_{\hat t}^{\T^{n}})=A_x \dot \wedge\hat  A_x + A_x\dot\wedge  \theta_{\hat t}^{\T^{n}} + \theta_t^{\T^{n}}\dot\wedge \hat A_x +\theta_t^{\T^{n}} \dot\wedge \theta_{\hat t}^{\T^{n}}\text{.} 
\end{equation*}
We change to coordinates w.r.t. $\R^{s} \times \T^{n} \times \T^{n}$, which we label by $i$, $\mu$, and $\hat \mu$. 
Then, we obtain
\begin{equation*}
-\rho(e_i,e_j)=A_i\hat A_j -  A_j\hat A_i
\quomma
-\rho(e_i,e_{\mu}) = -\hat A_{\mu i}
\quomma
-\rho(e_i,e_{\hat\mu}) =A_{\hat\mu i}
\quand 
-\rho(e_{\mu},e_{\hat\nu})=\delta_{\mu\hat \nu} 
\end{equation*}
Note that
\begin{equation*}
(B_{bas})_{ij} = (\pr^{*}B)(e_i,e_j)
\quomma
(B_{mix})_{i\mu} = (\pr^{*}B)(e_i,e_{\mu})
\quand
(B_{fib})_{\mu\nu} = (\pr^{*}B)(e_{\mu},e_{\nu}) 
\end{equation*}
and similarly
\begin{equation*}
(\hat B_{bas})_{ij} = (\hat\pr^{*}\hat B)(e_i,e_j)
\quomma
(\hat B_{mix})_{i\hat\mu} = (\hat\pr^{*}\hat B)(e_i,e_{\hat\mu})
\quand
(\hat B_{fib})_{\hat\mu\hat\nu} = (\hat\pr^{*}\hat B)(e_{\hat\mu},e_{\hat\nu}) \text{,}
\end{equation*}
and all other components vanish.
Further, we have
\begin{equation*}
(\pr_{\T^{2n}}^{*}\Omega)(e_{\mu},e_{\hat\nu}) = -\delta_{\mu\hat\nu}\text{,}
\end{equation*}
with again all other components vanishing.
Thus, (c) is equivalent to the following set of equations:
\begin{align*}
(\hat B_{bas})_{ij} - (B_{bas})_{ij} &= A_i\hat A_j -  A_j\hat A_i
\\
(B_{mix})_{i\mu} &=\hat A_{\mu i}
\\
(\hat B_{mix})_{i\hat\mu}  &= A_{\hat\mu i}
\end{align*}
The second and third equation are \cref{eq:BRq:2,eq:BRq:5}. The first equation, using second and third, is equivalent to
\begin{equation*}
\hat B_{bas} - B_{bas} = (B_{mix}A)^{\mathrm{tr}} - B_{mix}A
\end{equation*}
and this is precisely \cref{eq:BRq:4}.
\end{proof}

A straightforward computation using \cref{lem:buscherandpoincare:c} shows the following.

\begin{lemma}
\label{lem:buscherrulesBequiv}
Suppose  $(A,g',h,B)$ and $(\hat A,\hat g',\hat h,\hat B)$ are Buscher quadruples satisfying the Buscher rules of \cref{def:BRq}. Then, we have
\begin{align}
\label{eq:equi:B}
B_{1,2+3} &= B_{1,2}+ \hat A_1\wedge \theta_3
\\
\label{eq:equi:hatB}
\hat B_{1,2+3} &=  \hat B_{1,2}+ A_1 \wedge \theta_3
\end{align}
over $\R^{s} \times \T^{n} \times \T^{n}$. In particular, $B$ and $\hat B$ are $\T^{n}$-invariant. 
\end{lemma}

\setsecnumdepth{2}

 \section{Geometric T-duality}

\label{sec:gtd}

In this section we give the central definitions of this article: we introduce geometric T-backgrounds (\cref{def:tb}) and geometric T-duality correspondences between them (\cref{def:gtdc}). We deduce a number of first consequences; in particular, we relate geometric T-duality to T-duality with H-flux and to topological T-duality.

\subsection{Basic definitions}

\label{sec:basicdefinitions}

\begin{definition}
\label{def:tb}
A \emph{geometric T-background} over a smooth manifold $X$ is a triple $(E,g,\mathcal{G})$ consisting of a principal $\T^{n}$-bundle $E$ over $X$, a $\T^{n}$-invariant Riemannian metric $g$ on $E$, and a bundle gerbe $\mathcal{G}$  over $E$ with connection. Two geometric T-backgrounds $(E_1,g_1,\mathcal{G}_1)$ and $(E_2,g_2,\mathcal{G}_2)$ over $X$ are equivalent, if there exists a bundle isomorphism $ f:E_1 \to E_2$ that is isometric with respect to the metrics $g_1$ and $g_2$, and a connection-preserving bundle gerbe isomorphism $\mathcal{G}_1 \cong  f^{*}\mathcal{G}_2$. The set of equivalence classes of geometric T-backgrounds over $X$ is denoted by $\GTB(X)$. 
\end{definition}

As every bundle gerbe with connection has a curvature 3-form, every geometric T-background carries a 3-form $H\in \Omega^3(E)$, the \emph{H-flux}. Note that $H$ is closed, but in general not exact. The H-fluxes of equivalent geometric T-backgrounds satisfy $H_1= f^{*}H_2$.  

If 
$(E,g,\mathcal{G})$ and $(\hat E, \hat g, \hat{\mathcal{G}})$ are geometric T-backgrounds over the same manifold $X$, then the principal $\T^{2n}$-bundle $E \times_X \hat E$ is called the \emph{correspondence space}. It fits into an important commutative diagram:
\begin{equation*}
\alxydim{}{& E \times_X \hat E \ar[dr]^{\hat\pr} \ar[dl]_{\pr} \\ E \ar[dr]_-{p} && \hat E \ar[dl]^-{\hat p} \\ & X}
\end{equation*}
Let $\omega\in\Omega^1(E,\R^{n})$ and $\hat\omega\in\Omega^1(\hat E,\R^{n})$ be the Kaluza-Klein connections of the metrics $g$ and $\hat g$, respectively, under \cref{th:kaluzaklein}. Then, we consider the 2-form
\begin{equation}
\label{def:rhogg}
\rho_{g,\hat g} := \hat\pr^{*}\hat\omega \dot\wedge \pr^{*}\omega \in \Omega^2(E \times_X \hat E)\text{,}  
\end{equation}
where $\dot\wedge$ denotes the wedge product of $\R^{n}$-valued forms w.r.t. the standard inner product. Since $\omega$ and $\hat\omega$ are $\T^{n}$-invariant (they are connections on a principal bundle with abelian structure group), the 2-form $\rho_{g,\hat g}$ is $\T^{2n}$-invariant. We remark that the 2-form $\rho_{g,\hat g}$ also appeared in \cite{hori1,Bouwknegt}.

\begin{definition}
\label{def:gc}
A \emph{geometric correspondence} over $X$ consists of two geometric T-backgrounds
$(E,g,\mathcal{G})$ and $(\hat E, \hat g, \hat{\mathcal{G}})$ over $X$, and  a connection-preserving bundle gerbe isomorphism
\begin{equation*}
\mathcal{D} : \pr^{*}\mathcal{G} \to  \hat\pr^{*}\hat{\mathcal{G}} \otimes \mathcal{I}_{\rho_{g,\hat g}}\text{.}
\end{equation*}
\end{definition}

\begin{remark}
\label{re:H}
We shall explore some consequences of the isomorphism $\mathcal{D}$ in a geometric correspondence. For this, we will denote by $F,\hat F \in \Omega^2(X)$ the curvatures of the connections $\omega$ and $\hat\omega$, respectively. 
\begin{enumerate}[(a)]

\item 
Since the curvatures of isomorphic bundle gerbes with connection coincide, we have
\begin{equation}
\label{eq:preequi:H}
\pr^{*}H - \hat\pr^{*}\hat H = \mathrm{d}\rho_{g,\hat g} \text{,} 
\end{equation}
which is a condition in the context of T-duality with H-flux, see \cite[Eq. 1.12]{Bouwknegt} and \cref{def:TdualitywithHflux}.
From \cref{eq:preequi:H} and the definition of $\rho_{g,\hat g}$ one can deduce the equivariance rule
\begin{equation}
\label{eq:equi:H}
R^{*}H=e^{*}H+e^{*}p^{*}\hat F \,\dot\wedge\, h^{*}\theta
\end{equation}
on $E \times \T^{n}$, where $R$ is the principal action, $e$ the projection to $E$, and  $h$ the projection to $\T^{n}$. 
Similarly, on the dual  side we obtain 
\begin{equation}
\label{eq:equi:hatH}
R^{*}\hat H=\hat e^{*}\hat H+\hat e^{*}\hat p^{*}F \wedge h^{*}\theta
\end{equation}
on $\hat E \times \T^{n}$. In particular, these formulas show that $H$ and $\hat H$ are $\T^{n}$-invariant. 

\item
\label{re:K}
We consider the 3-forms
\begin{equation*}
\tilde K := \omega\dot \wedge\, p^{*}\hat F - H \in \Omega^3(E) \quand \hat K:=\hat \omega\dot \wedge\,\hat p^{*} F - \hat H \in \Omega^3(\hat E) 
\end{equation*}
Using \cref{eq:equi:H,eq:equi:hatH} one can show that $R^{*}\tilde K=e^{*}\tilde K$ and $R^{*}\hat K=\hat e^{*}\hat K$, so that these forms descend to $X$. 
In fact, $\tilde K$ and $\hat K$ both descend to the \emph{same} 3-form $K\in \Omega^3(X)$, i.e., $p^{*}K=\tilde K$ and $\hat p^{*}K=\hat K$. To see this, it suffices to note that the pullbacks of $\tilde K$ and $\hat K$ to  the correspondence space coincide, which again can be checked using \cref{eq:preequi:H,def:rhogg}.
Summarizing, every geometric correspondence determines a 3-form $K\in \Omega^3(X)$ 
such that 
\begin{equation*}
p^{*}K=\omega\dot\wedge\, p^{*}\hat F-H
\quand
\hat p^{*}K=\hat\omega\dot\wedge\,\hat p^{*} F - \hat H\text{.}
\end{equation*}
Note that $\mathrm{d}K = F \dot\wedge\hat F$.

\end{enumerate}
\end{remark}

\begin{remark}
\label{rem:inversecorr}
Geometric correspondence is a symmetric relation on the set $\GTB(X)$.
If $\mathcal{D}$ is a correspondence from $(E,g,\mathcal{G})$ to $(\hat E, \hat g, \hat{\mathcal{G}})$, then we construct a correspondence from $(\hat E, \hat g, \hat{\mathcal{G}})$ to $(E,g,\mathcal{G})$  as follows. Let $s: \hat E \times_X  E \to  E \times_X \hat E$ denote the swap map. Then, we consider
\begin{equation*}
\alxydim{@C=5em}{\hat\pr^{*}\hat{\mathcal{G}} = \hat\pr^{*}\hat{\mathcal{G}} \otimes \mathcal{I}_{s^{*}\rho_{g,\hat g}} \otimes \mathcal{I}_{-s^{*}\rho_{g,\hat g}} \ar[r]^-{s^{*}\mathcal{D}^{-1} \otimes \id} & \pr^{*}\mathcal{G} \otimes \mathcal{I}_{-s^{*}\rho_{g,\hat g}}\text{.}}
\end{equation*}
Since $-s^{*}\rho_{g,\hat g}=\rho_{\hat g,g}$, this is again a geometric correspondence. 

\end{remark}

\begin{definition}
\label{def:tcorr:equiv}
Two geometric correspondences over $X$, $((E,g,\mathcal{G}),(\hat E,\hat g,\hat{\mathcal{G}}),\mathcal{D})$ and $((E',g',\mathcal{G}'),(\hat E',\hat g',\hat{\mathcal{G}'}),\mathcal{D}')$, are considered to be equivalent, if there exist isometric bundle isomorphisms $f:E \to E'$ and $\hat f: \hat E \to \hat E'$, connection-preserving bundle gerbe isomorphisms $\mathcal{A}:\mathcal{G} \to  f^{*}\mathcal{G}'$ and $\hat{\mathcal{A}}:\hat{\mathcal{G}} \to \hat f^{*}\hat{\mathcal{G}}'$, and a connection-preserving 2-isomorphism 
\begin{equation*}
\alxydim{@C=5em}{\pr^{*}\mathcal{G} \ar[r]^-{\mathcal{D}} \ar[d]_{\pr^{*}\mathcal{A}} & \hat\pr^{*}\hat{\mathcal{G}}\otimes \mathcal{I}_{\rho_{g,\hat g}} \ar[d]^{\hat\pr^{*}\hat{\mathcal{A}}\otimes \id} \ar@{=>}[ddl]|*+{\xi} \\ \pr^{*} f^{*}\mathcal{G}' \ar@{=}[d] & \hat\pr^{*}\hat f^{*}\hat{\mathcal{G}}' \otimes \mathcal{I}_{\rho_{g,\hat g}} \ar@{=}[d] \\ F^{*}\pr'^{*}\mathcal{G}' \ar[r]_-{F^{*}\mathcal{D}'} &  F^{*}\hat\pr'^{*}\hat{\mathcal{G}}' \otimes F^{*}\mathcal{I}_{\rho_{g',\hat g'}}} 
\end{equation*}
where $F:= f \times \hat f:E \times_X \hat E \to E' \times_X \hat E'$. The set of equivalence classes of geometric correspondences over $X$ is denoted by $\GC(X)$. 
\end{definition}

\begin{remark}
In above definition we have implicitly used that $F^{*}\rho_{g',\hat g'}=\rho_{g_1,\hat g_1}$, which follows from the fact that $ f$ and $\hat f$ are connection-preserving, which in turn follows from the assumption that $ f$ and $\hat f$ are isometric (\cref{re:isometries}).
\end{remark}

\begin{remark}
\label{re:actiongerbe}
Let $\mathcal{H}$ be a bundle gerbe with connection over $X$. Then, we may send
\begin{equation*}
((E,g,\mathcal{G}),(\hat E,\hat g,\hat{\mathcal{G}}),\mathcal{D}) \mapsto ((E,g,\mathcal{G}\otimes p^{*}\mathcal{H}),(\hat E,\hat g,\hat{\mathcal{G}} \otimes \hat p^{*}\mathcal{H}),\mathcal{D} \otimes \id_{\mathcal{H}})\text{.}
\end{equation*} 
This gives a well-defined action of the group of isomorphism classes of bundle gerbes with connection  on the set of equivalence classes of geometric correspondences,
\begin{equation*}
\mathrm{Grb}^{\nabla}(X) \times\GC(X) \to \GC(X)\text{.}
\end{equation*}
\end{remark}

\begin{remark}
It is straightforward to see that equivalent geometric correspondences determine the same 3-form $K$. The action of \cref{re:actiongerbe} shifts this 3-form by $\mathrm{curv}(\mathcal{H})$. 
\end{remark}

\begin{definition}
\label{def:gtdc}
A geometric correspondence $\mathcal{D}$ between two geometric T-backgrounds
$(E,g,\mathcal{G})$ and $(\hat E, \hat g, \hat{\mathcal{G}})$ over $X$ is called \emph{geometric T-duality correspondence} if the following conditions hold:
\begin{enumerate}[(T1),leftmargin=2.5em]

\item 
\label{def:gtdc:1}
The Riemannian metrics $g'$ and $\hat g'$ on $X$ determined by the metrics $g$ and $\hat g$, respectively, under \cref{th:kaluzaklein} coincide, i.e., $g'=\hat g'$. 

\item
\label{def:gtdc:2}
The families of inner products $h$ and $\hat h$ on $\R^{n}$ determined by the metrics $g$ and $\hat g$, respectively, under \cref{th:kaluzaklein}, satisfy $h^{-1}=\hat h$ under their identification with $(n\times n)$-matrices.  

\item
\label{def:gtdc:3}
Every point $x\in X$ has an open neighborhood $U \subset M$ such that the following structures exist:
\begin{enumerate}[(a)]

\item 
\label{def:gtdc:3a}
Trivializations $\varphi : U \times \T^{n} \to E|_U$ and $\hat \varphi: U \times \T^{n} \to \hat E|_U$ of principal $\T^{n}$-bundles over $U$.

\item
\label{def:gtdc:3b}
Two 2-forms $B,\hat B \in \Omega^2(U \times \T^{n})$ together with connection-preserving isomorphisms $\mathcal{T}: \varphi^{*}\mathcal{G} \to \mathcal{I}_B$  and $\hat{\mathcal{T}}:\hat\varphi^{*}\hat{\mathcal{G}}\to \mathcal{I}_{\hat B}$ over $U \times \T^{n}$.

\item
\label{def:gtdc:3c}
Consider $U \times \T^{2n}$ with projection maps $\pr, \hat \pr$ to $U \times \T^{n}$. Further, consider the map $\Phi: U \times \T^{2n} \to E \times_X \hat E$ defined by $\Phi(x,a,\hat a) := (\varphi(x,a),\hat \varphi(x,\hat a))$. 
Let $P$ denote the principal $\T$-bundle with connection over $U \times \T^{2n}$ that corresponds to the isomorphism
\begin{align*}
&\alxydim{@C=3.5em}{\mathcal{I}_{\pr^{*}B}=\pr^{*}\mathcal{I}_{B} \ar[r]^-{\pr^{*}\mathcal{T}^{-1}} &  \pr^{*}\varphi^{*}\mathcal{G} = \Phi^{*}\pr^{*}\mathcal{G}}
\\[-1em]
&\hspace{10em}\alxydim{@R=3em}{&&\ar[dll]^-{\Phi^{*}\mathcal{D}}\\ &}
\\[-1em]
&\hspace{6em}\alxydim{@C=4em}{\Phi^{*}(\hat\pr^{*}\hat{\mathcal{G}} \otimes \mathcal{I}_{\rho})= \hat \pr^{*}\hat\varphi^{*}\hat{\mathcal{G}} \otimes \mathcal{I}_{\Phi^{*}\rho} \ar[r]^-{\hat \pr^{*}\hat{\mathcal{T}} \otimes \id} &  \hat\pr^{*}\mathcal{I}_{\hat B} \otimes \mathcal{I}_{\Phi^{*}\rho} = \mathcal{I}_{\hat \pr^{*}\hat B + \Phi^{*}\rho}}
\end{align*}
under the equivalence of \cref{lem:gerbehombundle}.
We require a connection-preserving isomorphism
\begin{equation*}
P \cong \pr_{\T^{2n}}^{*}\poi\text{,}
\end{equation*}
where $\mathcal{P}$ is the $n$-fold Poincaré bundle with its canonical connection.

\end{enumerate}

\end{enumerate}
The set of equivalence classes of geometric T-duality correspondences over $X$ (with the equivalence relation just as in \cref{def:tcorr:equiv}) is denoted by $\GTC(X)$. 
\end{definition}

\begin{remark}
If $\mathcal{D}$ is a geometric T-duality correspondence, then the inverse correspondence $s^{*}\mathcal{D}^{-1}$ of \cref{rem:inversecorr} is also a geometric T-duality correspondence. 
Indeed, conditions \cref{def:gtdc:1*} and \cref{def:gtdc:2*} are obviously symmetric, and in \cref{def:gtdc:3*} it is straightforward to show that swapping $s$ and inversion $\mathcal{D}^{-1}$  have \emph{both} the effect of dualizing the bundle $P$.
Thus, geometric T-duality is a symmetric relation on the set $\GTC(X)$.
\end{remark}

\begin{remark}
\label{re:action}
It is straightforward to see that the action of \cref{re:actiongerbe} restricts to an action of $\hat \h^3(X)$ on $\GTC(X)$. The properties of this action are best studied in the context of differential T-duality and carried out in differential cohomology, see \cref{prop:actionfreeandtransitive}. The result of \cref{prop:actionfreeandtransitive} is the following.  
\end{remark}

\begin{proposition}
\label{lem:action}  
Let
\begin{equation}
\label{eq:actionfibring}
\GTC(X) \to \mathrm{Bun}^{\nabla}_{\T^{n}}(X)
 \times 
\mathrm{Bun}^{\nabla}_{\T^{n}}(X)
\end{equation}
be the projection to the isomorphism classes of the principal $\T^{n}$-bundles $E$ and $\hat E$ and their Kaluza-Klein connections $\omega$ and $\hat\omega$ induced by the metric $g$ and $\hat g$, respectively. We denote by $(F,\hat F) \in \Omega^2(X) \times \Omega^2(X)$ the well-defined pair of curvature forms.  Consider the subgroup 
\begin{equation*}
\mathcal{F}_{F,\hat F}:= \{ \mathcal{I}_{\hat y F+y\hat F} \sep y,\hat y\in \R \}\subset \mathrm{Grb}^{\nabla}(X)\text{.}
\end{equation*}
Then, the quotient $\mathrm{Grb}^{\nabla}(X)/\mathcal{F}_{F,\hat F}$ acts free and transitively in the fibre of \cref{eq:actionfibring} over an element with curvature pair $(F,\hat F)$. 
\end{proposition}

\begin{remark}
The assignments $X \mapsto \GTB(X)$ and $X \mapsto \GTC(X)$ are  presheaves on the category of smooth manifolds. In fact, it is straightforward and only omitted for brevity to enhance the sets $\GTB(X)$ and $\GTC(X)$ to bicategories, which then form \emph{sheaves} of bicategories on the site of smooth manifolds.  
\end{remark}

\subsection{Relation to Buscher rules}

\label{sec:relationtobuscherrules}

We will now make a deeper analysis of condition  \cref{def:gtdc:3*}\unskip\cref{def:gtdc:3c*}, and in particular show that the Buscher rules are satisfied over $U$. 

\begin{proposition}
\label{lem:Bs}
Let $((E,g,\mathcal{G}),(\hat E,\hat g,\hat{\mathcal{G}}),\mathcal{D})$ be a geometric T-duality correspondence. Consider an open set $U \subset X$ together with the structure listed in \cref{def:gtdc:3*}\unskip\cref{def:gtdc:3b*}. Then, the 
 Buscher pairs $(g,B)$ and $(\hat g,\hat B)$ satisfy the Buscher rules. 

\end{proposition}

\begin{proof}
Applying \cref{lem:gerbehombundle} to the bundle gerbe isomorphism in \cref{def:gtdc:3*}\unskip\cref{def:gtdc:3c*} yields
\begin{equation*}
\hat\pr^{*}\hat B  - \pr^{*}B = \pr_{\T^{2n}}^{*}\Omega-\Phi^{*}\rho_{g,\hat g}\text{.}
\end{equation*}
In addition to \cref{def:gtdc:1*,def:gtdc:2*},   \cref{lem:buscherandpoincare,lem:br:equiv} show  that $(g,B)$ and $(\hat g,\hat B)$ satisfy the Buscher rules. 
\end{proof}

Conversely, geometric T-duality locally does not pose any more conditions than the Buscher rules. To see this, we observe that any Buscher pair $(g,B)$ extends to a geometric T-duality background, with $E^{s,n}:= \R^{s} \times \T^{n}$, the given metric $g$, and the trivial bundle gerbe $\mathcal{I}_B$.
If Buscher pairs $(g,B)$ and $(\hat g,\hat B)$ satisfy the Buscher rules, then we have $\hat\pr^{*}\hat B -\pr^{*}B+\rho_{g,\hat g}=\pr_{\T^{2n}}^{*}\Omega$ by \cref{lem:buscherandpoincare}. Thus, $\pr_{\T^{2n}}^{*}\poi$ corresponds under \cref{lem:gerbehombundle} to a connection-preserving isomorphism
\begin{equation}
\label{eq:corresiso}
\mathcal{D}:\pr^{*}\mathcal{I}_B \to\hat\pr^{*}\mathcal{I}_{\hat B} \otimes \mathcal{I}_{\rho_{g,\hat g}}
\end{equation}
over the correspondence space 
$E^{s,n} \times_{\R^{s}} E^{s,n}$. 
  
\begin{proposition}
\label{prop:gtdclocally}
Suppose $(g,B)$ and $(\hat g,\hat B)$ are Buscher pairs and satisfy the Buscher rules. Then, the connection-preserving isomorphism \cref{eq:corresiso} establishes a geometric T-duality correspondence between $(E^{s,n},g,\mathcal{I}_B)$ and $(E^{s,n},\hat g,\mathcal{I}_{\hat B})$. 
\end{proposition}

\begin{proof}
Conditions \cref{def:gtdc:1*,def:gtdc:2*} of \cref{def:gtdc} are \cref{lem:buscherandpoincare} (a) and (b). That condition \cref{def:gtdc:3*}
is  satisfied can be seen using the identity trivializations $\varphi, \hat\varphi$ and $\mathcal{T},\hat{\mathcal{T}}$. 
\end{proof}

\subsection{Relation to topological T-duality}

\label{sec:topTduality}

We shall first recall the definition of topological T-duality following \cite{Bunke2005a,Bunke2006a,Mathai2006}.

\begin{definition}
\label{def:ttd}
A \emph{topological T-background} over $X$ is a principal $\T^{n}$-bundle $E$ together with a bundle gerbe $\mathcal{G}$ over $E$. Two topological T-backgrounds $(E,\mathcal{G})$ and $(E',\mathcal{G}')$ over $X$ are \emph{equivalent} if there exists a bundle isomorphism $\varphi: E \to E'$ and a bundle gerbe isomorphism $\mathcal{G} \cong \varphi^{*}\mathcal{G}'$. Equivalence classes of topological T-backgrounds over $X$ form a set $\TTB (X)$.
\end{definition}

Every geometric T-background $(E,g,\mathcal{G})$ induces a topological T-background $(E,\mathcal{G})$ by forgetting the metric and forgetting the gerbe connection. Conversely, if $(E,\mathcal{G})$ is a topological T-background, one can choose any $\T^{n}$-invariant metric on $E$ and use the fact that every bundle gerbe admits a  connection \cite{murray}, to upgrade it to a geometric T-background.
Thus, we have a surjective map
\begin{equation*}
\GTB(X) \to \TTB(X)\text{.}
\end{equation*}

\begin{definition}
A \emph{topological correspondence} between  T-backgrounds $(E,\mathcal{G})$ and $(\hat E,\hat{\mathcal{G}})$ is an isomorphism $\mathcal{D}: \pr^{*}\mathcal{G} \to \hat\pr^{*}\hat{\mathcal{G}}$ over $E \times_X E$. A topological correspondence $\mathcal{D}$ is called topological T-duality correspondence, if each point $x\in X$ has an open neighborhood $U \subset X$ such that the following structures exist:
\begin{enumerate}[(a)]

\item 
Trivializations $\varphi : U \times \T^{n} \to E|_U$ and $\hat \varphi: U \times \T^{n} \to \hat E|_U$ of principal $\T^{n}$-bundles over $U$.

\item
Bundle gerbe isomorphisms $\mathcal{T}: \varphi^{*}\mathcal{G} \to \mathcal{I}$  and $\hat{\mathcal{T}}:\hat\varphi^{*}\hat{\mathcal{G}}\to \mathcal{I}$ over $U \times \T^{n}$.

\item
The bundle gerbe isomorphism
\begin{equation*}
\alxydim{@C=4em}{\mathcal{I} \ar[r]^-{\pr^{*}\mathcal{T}^{-1}} &  \pr^{*}\varphi^{*}\mathcal{G} = \Phi^{*}\pr^{*}\mathcal{G} \ar[r]^-{\Phi^{*}\mathcal{D}}  & \Phi^{*}(\hat\pr^{*}\hat{\mathcal{G}})=\hat \pr^{*}\hat\varphi^{*}\hat{\mathcal{G}}  \ar[r]^-{\hat \pr^{*}\hat{\mathcal{T}}} &  \mathcal{I} }
\end{equation*}
over $U \times \T^{2n}$ corresponds under the equivalence of \cref{lem:gerbehombundle} to $\pr_{\T^{2n}}^{*}\poi$.
\end{enumerate}

\end{definition}

\begin{proposition}
\label{prop:geototop}
If $\mathcal{D}$  is a geometric T-duality correspondence between two geometric T-backgrounds
$(E,g,\mathcal{G})$ and $(\hat E, \hat g, \hat{\mathcal{G}})$ over $X$, then $\mathcal{D}$ is a topological T-duality correspondence between  the topological T-backgrounds $(E,\mathcal{G})$ and $(\hat E, \hat{\mathcal{G}})$. 
\end{proposition}

\begin{proof}
Discarding all metrics and connections from \cref{def:gc,def:gtdc} results precisely in \cref{def:ttd}. 
\end{proof}

\begin{corollary}
If two geometric T-backgrounds
$(E,g,\mathcal{G})$ and $(\hat E, \hat g, \hat{\mathcal{G}})$ over $X$ are in geometric T-duality correspondence, then the homomorphism
\begin{equation*}
\alxydim{}{K^{\mathcal{G}}(E) \ar[r]^-{\pr^{*}} & K^{\pr^{*}\mathcal{G}}(E \times_X \hat E) \ar[r]^-{\mathcal{D}} &  K^{\hat\pr^{*}\hat{\mathcal{G}}}(E \times_X \hat E) \ar[r]^-{\hat\pr_{*}} & K^{\hat{\mathcal{G}}}(\hat E) }
\end{equation*}
of twisted K-theory groups is an isomorphism.
\end{corollary}

There is also an interesting converse question. Suppose two topological T-backgrounds are in topological T-duality correspondence. Can one lift them to geometric T-backgrounds that are in geometric T-duality correspondence? 

\begin{proposition}
\label{prop:upgradefromtoptogeo}
Every topological T-duality correspondence can be lifted to a geometric T-duality correspondence. In more detail, 
suppose $(E,\mathcal{G})$ and $(\hat E,\hat{\mathcal{G}})$ are topological T-backgrounds, and suppose $\mathcal{D}$ is a topological T-duality correspondence. Then, there exist $\T^{n}$-equivariant metrics $g$ and $\hat g$ on $E$ and $\hat E$,  connections on $\mathcal{G}$ and $\hat{\mathcal{G}}$ and a connection on $\mathcal{D}$ such that $\mathcal{D}$ is a geometric T-duality correspondence between $(E,g,\mathcal{G})$ and $(\hat E,\hat g,\hat{\mathcal{G}})$. 
\end{proposition}

\begin{proof}
Combines \cref{prop:upgradefromtoptodiff,prop:upgradefromdifftogeo}, to be proved later using the local formalism. 
\end{proof}

\begin{remark}
It is straightforward to see that \cref{prop:geototop} 
induces a map
\begin{equation*}
\GTC(X) \to \TTC(X)\text{.}
\end{equation*}
\cref{prop:upgradefromtoptogeo} implies that this map is surjective. 
\end{remark}

\begin{remark}
A purely topological version of the action of \cref{re:action},
\begin{equation*}
\h^3(X) \times \TTC(X) \to \TTC(X)
\end{equation*}
exists, and it obviously acts in the fibres of the map $\TTC(X) \to \mathrm{Bun}_{\T^{n}}(X) \times \mathrm{Bun}_{\T^{n}}(X)$. Bunke-Rumpf-Schick have investigated a similar action in \cite[\S 7.2]{Bunke2006a}.
\end{remark}

\subsection{Relation to T-duality with H-flux}

\label{sec:BEM}

In this section we show that geometric T-duality implies T-duality with H-flux in the sense developed by Bouwknegt-Evslin-Mathai in \cite{Bouwknegt} and Bouwknegt-Hannabuss-Mathai in \cite{Bouwknegt2004}. In these papers, T-duality is not considered as a relation between T-backgrounds, but rather as a transformation that takes a T-background to another. 
A description of T-duality with H-flux as a relation on a class of suitable backgrounds has been given by Gualtieri-Cavalcanti  in \cite{Cavalcanti} based on \cite{Bouwknegt,Bouwknegt2004}, and we will use this here. 

\begin{definition}
A \emph{T-background with H-flux} over $X$ is a principal $\T^{n}$-bundle $E$ over $X$ together with a closed 3-form $H\in \Omega^3(E)$ with integral periods. 
\end{definition}

Every geometric T-background $(E,g,\mathcal{G})$ induces one with H-flux where the metric $g$ is forgotten and  $H$ is the curvature of $\mathcal{G}$.
Conversely, every T-background with H-flux $(E,H)$ can be upgraded to a geometric T-background by choosing some $\T^{n}$-invariant metric and some bundle gerbe with connection of curvature $H$.

For the following definition, we consider the correspondence space $E \times_X \hat E$, and at each point $(e,\hat e)$, projecting to some $x\in X$, the subspaces $V_e,\hat V_{\hat e} \subset T_{e,\hat e}(E \times_X \hat E) = T_eE \times_{T_xX} T_{\hat e}\hat E$ obtained as the image of the maps
\begin{align*}
&i: \R^{n} \stackrel\cong \to V_eE \incl T_eE \times_{T_xX} T_{\hat e}\hat E,\quad v \mapsto (T_1R_e(v),0)  && V_e := \mathrm{Im}(i)
\\
&\hat i: \R^{n} \stackrel\cong \to \hat V_{\hat e}E \incl T_eE \times_{T_xX} T_{\hat e}\hat E,\quad w \mapsto (0,T_1R_{\hat e}(w)) && \hat V_{\hat e} := \mathrm{Im}(\hat i)\text{.} 
\end{align*}

\begin{definition} 
\label{def:TdualitywithHflux}
A \emph{T-duality correspondence with H-flux} consists of two T-backgrounds with H-flux $(E,H)$ and $(\hat E,\hat H)$ and a $\mathbb{T}^{2n}$-invariant 2-form $F\in \Omega^2(E \times_X \hat E)$ such that
\begin{enumerate}[(a)]

\item 
$\pr^{*}H-\hat\pr^{*}\hat H = \mathrm{d}F$.

\item
The restriction of $F_{e,\hat e}$ to $V_e \times \hat V_{\hat e}$ is non-degenerate, for all $(e,\hat e)\in E \times_X \hat E$. 

\end{enumerate}
\end{definition}

Now we are in position to show that geometric T-duality reduces to T-duality with H-flux. This is a result of Kunath's PhD thesis \cite[Thm. 5.10]{Kunath2021}.

\begin{proposition}
\label{prop:geotoH}
Suppose $(E,g,\mathcal{G})$ and $(\hat E,\hat g,\hat{\mathcal{G}})$ are in geometric T-duality correspondence $\mathcal{D}$. Then, $F:= \rho_{g,\hat g}$ defined in \cref{def:rhogg} is a T-duality correspondence with H-flux.  
\end{proposition}

\begin{proof}
As remarked in \cref{sec:basicdefinitions}, the 2-form $\rho_{g,\hat g}$ is $\T^{2n}$-invariant. 
The first condition is proved in \cref{re:H}. For the second condition, we have
\begin{align*}
F_{e,\hat e}((T_1R_e(v),0),(0,T_1R_{\hat e}(w))) &= \hat\omega_{\hat e}(0)\cdot \omega_e(0)-\hat\omega_{\hat e}(T_1R_{\hat e}(w))\cdot \omega_e(T_1R_e(v))
= -w\cdot v\text{;}
\end{align*}
i.e., we obtain (minus) the standard scalar product of $\R^{n}$, which is non-degenerate. 
\end{proof}

\begin{remark}
For a general base manifold $X$, one cannot expect that every given T-duality correspondence with 
H-flux can be upgraded to a geometric (or only topological) T-duality correspondence. Indeed, a topological T-duality correspondence implies the triviality of the class $c_1(E) \cup c_1(\hat E) \in \h^4(X,\Z)$, while a T-duality correspondence with H-flux only implies the triviality of that class in de Rham cohomology.   
\end{remark}

\setsecnumdepth{2}

\section{Local perspective to geometric T-duality}

\label{sec:localformalism}

We may see               condition \cref{def:gtdc:3*} of \cref{def:gtdc} as enforcing  a geometric T-duality correspondence to be \emph{locally trivial}. Just as for locally trivial fibre bundles, one may then extract \quot{local data}, or \quot{gluing data}. It is instructive to first do this in an ad hoc manner, which is the content of  \cref{sec:exlocdata}. In \cref{sec:localdata} we organize local data in a more systematic way, establishing the table in \cref{fig:localdata} of  \cref{sec:introduction}. \cref{sec:reconstruction,sec:welldefinednessofreconstruction,sec:localtoglobal}
are devoted to a full proof of a bijection between the set $\GTC(X)$ of equivalence classes of geometric T-duality correspondences and a set $\LDgeo(X)$ of equivalence classes of local data. In \cref{re:TDcocycles} we reduce the discussion of local data to \emph{topological} T-duality, and show that this reduction becomes the non-abelian cohomology with values in the T-duality 2-group.

\subsection{Extraction of local data}

\label{sec:exlocdata}

We suppose that we have a geometric T-duality correspondence $\mathcal{D}$ as in \cref{def:gtdc}, between geometric T-backgrounds $(E,g,\mathcal{G})$ and $(\hat E,\hat g,\hat{\mathcal{G}})$ over $X$. 
We assume then that $X$ is covered by open sets $U_i$ over which condition \cref{def:gtdc:3*} holds, and that corresponding bundle trivializations $\varphi_i,\hat\varphi_i$, bundle gerbe trivializations $\mathcal{T}_i,\hat{\mathcal{T}}_i$ and 2-isomorphisms $\xi_i$ are chosen for all $U_i$, where $\xi_i$ are the connection-preserving 2-isomorphisms 
\begin{equation}
\label{eq:2isomorphismsxii}
\xi_i: \hat{\mathcal{T}}_i\circ \Phi_i^{*}\mathcal{D} \circ \mathcal{T}_i^{-1} \Rightarrow \pr_{\T^{2n}}^{*}\poi\text{.}
\end{equation}

Let  $a_{ij}:U_i \cap U_j \to \T^{n}$ be the transition functions of $E$, which are determined by the trivializations $\varphi_i$ and $\varphi_j$, i.e., $\varphi_j(x,a)\cdot a_{ij}(x)=\varphi_i(x,a)$.
It will soon become necessary to choose and fix lifts of these transition functions along $\R^{n} \to \T^{n}$, which is always possible after eventually passing to a refinement of the open cover. The former cocycle condition then reveals \quot{winding numbers} $m_{ijk} \in \Z^{n}$ such that 
\begin{align}
\label{ex:cc:bundle}
a_{ij} + a_{jk} +m_{ijk} &= a_{ik}
\end{align}
and these integers $m_{ijk}$ themselves satisfy the usual \v Cech cocycle condition. We will also denote by $a_{ij}$ the corresponding map
\begin{equation*}
(U_i \cap U_j) \times \T^{n} \to (U_i \cap U_j) \times \T^{n} :(x,a)\mapsto (x,a+a_{ij}(x)) 
\end{equation*}
that multiplies by $a_{ij}(x)$; note that this map
satisfies $\varphi_i = \varphi_j \circ a_{ij}$.
Next, we  consider the composite
\begin{equation*}
\alxydim{@=3em}{\mathcal{I}_{B_i} \ar[r]^-{\mathcal{T}_i^{-1}} & \varphi_i^{*}\mathcal{G}=  a_{ij}^{*}\varphi_j^{*}\mathcal{G} \ar[r]^-{ a_{ij}^{*}\mathcal{T}_j} &  \mathcal{I}_{ a_{ij}^{*}B_j}}
\end{equation*}
of bundle gerbe isomorphisms over $(U_i \cap U_j)\times \T^{n}$,
which corresponds by \cref{lem:gerbehombundle} to a principal $\T$-bundle $L_{ij}$ over $(U_i \cap U_j) \times \T^{n}$ with connection of curvature $ a_{ij}^{*}B_j - B_i$. 

The same works on the dual side, resulting in transition functions $\hat a_{ij}:U_i \cap U_j \to \R^{n}$, winding numbers $\hat m_{ijk} \in \Z^{n}$ satisfying 
\begin{align}
\label{ex:cc:bundlehat}
\hat a_{ij} + \hat a_{jk} +\hat m_{ijk} &= \hat a_{ik}\text{,}
\end{align}
and principal $\T$-bundles $\hat L_{ij}$ over $(U_i \cap U_j) \times \T^{n}$ with connection of curvature $\hat a_{ij}^{*}\hat B_j - \hat B_i$.

Before we proceed, we remark that the local trivializations $\varphi_i,\hat \varphi_i$  also define local $\T^{n}$-invariant metrics $g_i := \varphi_i^{*}g$ and $\hat g_i:=\hat\varphi_i^{*}\hat g$ on $U_i \times \T^{n}$. Due to the $\T^{n}$-invariance of $g$ and $\hat g$, we have 
\begin{equation}
\label{ex:cc:metric}
 a_{ij}^{*}g_j=g_i
\quand
\hat  a_{ij}^{*}\hat g_j=\hat g_i\text{.}
\end{equation}
As seen in \cref{lem:Bs}, the pairs $(g_i,B_i)$ and $(\hat g_i,\hat B_i)$ satisfy the Buscher rules.

\begin{lemma}
\label{lem:extract}
The principal $\T$-bundles $L_{ij}$ and $\hat L_{ij}$ are trivializable. Thus, there exist $A_{ij},\hat A_{ij} \in \Omega^1((U_i \cap U_j) \times \T^{n})$ and connection-preserving isomorphisms
\begin{equation*}
\lambda_{ij}: L_{ij} \to \trivlin_{A_{ij}}
\quand
\hat\lambda_{ij}: \hat L_{ij} \to  \trivlin_{\hat A_{ij}}
\end{equation*}
over $(U_i \cap U_j) \times \T^{n}$. In particular,
we have\begin{equation}
\label{ex:cc:gerbe1}
 a_{ij}^{*}B_j - B_i= \mathrm{d} A_{ij}
\quand
\hat  a_{ij}^{*}\hat B_j - \hat B_i = \mathrm{d}\hat  A_{ij}\text{.}
\end{equation}
\end{lemma}

\begin{proof}
We assume that all non-empty double intersections $U_i \cap U_j$ are contractible; this can again be achieved by passing to a refinement. 
Then, the first Chern classes of $L_{ij}$ and $\hat L_{ij}$ must be pullbacks from $\T^{n}$. We have $\h^2(\T^{n},\Z) \cong \mathfrak{so}(n,\Z)$, the group of skew-symmetric integral $(n\times n)$-matrices, and this isomorphism can be realized explicitly using the Poincaré bundle $\poi$ over $\T^2$: we send a matrix $D\in \mathfrak{so}(n,\Z)$ to the principal $\T$-bundle 
\begin{equation*}
\poi_D := \bigotimes_{1\leq \alpha<\beta\leq n} \pr_{\alpha\beta}^{*}\poi^{D_{\alpha\beta}}\text{,}
\end{equation*}
see \cite[\S B]{Nikolause}.
Thus, there exist unique matrices  $D_{ij},\hat D_{ij}\in \mathfrak{so}(n,\Z)$ and (non-unique) bundle isomorphisms $L_{ij}\cong \pr_{\T^{n}}^{*}\mathcal{P}_{D_{ij}}$ and $\hat L_{ij}\cong \pr_{\T^{n}}^{*}\mathcal{P}_{\hat D_{ij}}$. 
Taking  connections into account, there exist 1-forms $A_{ij},\hat A_{ij} \in \Omega^1((U_i \cap U_j) \times \T^{n})$ and connection-preserving isomorphisms
\begin{equation*}
\lambda_{ij}: L_{ij} \to \pr_{\T^{n}}^{*}\poi_{D_{ij}} \otimes \trivlin_{A_{ij}}
\quand
\hat\lambda_{ij}: \hat L_{ij} \to \pr_{\T^{n}}^{*}\poi_{\hat D_{ij}} \otimes \trivlin_{\hat A_{ij}}
\end{equation*}
over $(U_i \cap U_j) \times \T^{n}$.

We show next that $D_{ij}=\hat D_{ij}=0$, implying the claim of the lemma. This will be a consequence of the geometric T-duality correspondence, and so we need to work over $(U_i\cap U_j) \times\ \T^{2n}$. We consider the following maps:
\begin{align*}
\pr &: (U_i \cap U_j) \times \T^{2n} \to (U_i \cap U_j) \times \T^{n}:(x,a,\hat a) \mapsto (x,a)
\\  
\hat\pr &: (U_i \cap U_j) \times \T^{2n} \to (U_i \cap U_j) \times \T^{n}:(x,a,\hat a) \mapsto (x,\hat a)
\\
\tilde  a_{ij} &: (U_i \cap U_j) \times \T^{2n} \to (U_i \cap U_j) \times \T^{2n}: (x,a,\hat a)\mapsto (x,a+a_{ij}(x),\hat a + \hat a_{ij}(x))
\\
\Phi_i &: U_i \times \T^{2n} \to E \times_X \hat E: (x,a,\hat a) \mapsto (\varphi_i(x,a),\hat\varphi_i(x,\hat a))\text{,}
\end{align*}
and construct with them the following diagram of bundle gerbes with connections and connection-preserving isomorphisms over $(U_i \cap U_j) \times \T^{2n}$:
\begin{align}
\label{eq:diagrametaij}
\alxydim{@C=3.3em@R=3em@M=0.3em}{&&&\mathcal{I}_{\pr^{*} a_{ij}^{*}B_j} \ar@{=}[d]^{}="h" \\\mathcal{I}_{\pr^{*}B_i} \ar@/^2pc/[rrru]^{\pr^{*}\pr_{\T^{n}}^{*}\poi_{D_{ij}}\otimes \trivlin_{\pr^{*} A_{ij}}}="1" \ar[d]_{\pr^{*}_{\T^{2n}}\poi}="b" \ar[r]^-{\pr^{*}\mathcal{T}_i^{-1}} & \pr^{*}\varphi_i^{*}\mathcal{G}   \ar[d]^{\Phi_i^{*}\mathcal{D}}="a" \ar@{=}[r]^{}="g" & \pr^{*} a_{ij}^{*}\varphi_j^{*}\mathcal{G} \ar@{<=}"h" \ar[d]_{\tilde  a_{ij}^{*}\Phi_j^{*}\mathcal{D}}="f" \ar@/^/[ru]^-{\pr^{*} a_{ij}^{*}\mathcal{T}_j} & \mathcal{I}_{\pr^{*} a_{ij}^{*}B_j} \ar[l]^-{\pr^{*} a_{ij}^{*}\mathcal{T}_j^{-1}}  \ar[d]^{\tilde  a_{ij}^{*}\pr^{*}_{\T^{2n}}\poi}="e"
\\
\mathcal{I}_{\hat\pr^{*}\hat B_i + \Phi_j^{*}\rho} \ar@{=}[d]_-{}="c"   & \hat\pr^{*}\hat \varphi_i^{*}\hat{\mathcal{G}} \otimes \mathcal{I}_{\Phi_i^{*}\rho} \ar@{=>}"c"  \ar[l]_-{\hat\pr^{*}\hat{\mathcal{T}}_i} \ar@{=}[r]^{}="d" & \hat\pr^{*}\hat a_{ij}^{*}\hat\varphi_j^{*}\hat{\mathcal{G}} \otimes \mathcal{I}_{\tilde  a_{ij}^{*}\Phi_j^{*}\rho}  \ar[r]_-{\hat\pr^{*}\hat  a_{ij}^{*}\hat{\mathcal{T}}_j} & \mathcal{I}_{\hat\pr^{*}\hat a_{ij}^{*}\hat B_j + \tilde  a_{ij}^{*}\Phi_j^{*}\rho} \\ \mathcal{I}_{\hat\pr^{*}\hat B_i + \Phi_j^{*}\rho} \ar@/_/[ru]_-{\hat\pr^{*}\hat{\mathcal{T}}_i^{-1}}   \ar@/_2pc/[rrru]_-{\hat\pr^{*}\pr_{\T^{n}}^{*}\poi_{\hat D_{ij}} \otimes \trivlin_{\hat\pr^{*}\hat A_{ij}}}="2" \ar@{=>}"a";"b"_{\xi_{i}} \ar@{<=}"2";"d"_-{\hat\pr^{*}\hat\lambda_{ij}} \ar@{=>}"e";"f"^{\tilde a_{ij}^{*}\xi_j^{-1}} \ar@{=>}"1";"g"_{\pr^{*}\lambda_{ij}^{-1}} }
\end{align}
The unlabelled double arrows are the canonical unit and counit 2-isomorphisms of the adjunction between a 1-isomorphism and its inverse.   
The rectangular subdiagram in the middle commutes on the nose. The outer shape of the diagram is, via \cref{lem:gerbehombundle}, a connection-preserving bundle isomorphism
\begin{equation}
\label{eq:cpi}
\eta_{ij}: \tilde  a_{ij}^{*}\pr^{*}_{\T^{2n}}\poi \otimes \pr^{*}\pr_{\T^{n}}^{*}\poi_{D_{ij}}\otimes \trivlin_{\pr^{*} A_{ij}} \cong \hat\pr^{*}\pr_{\T^{n}}^{*}\poi_{\hat D_{ij}} \otimes \trivlin_{\hat\pr^{*}\hat A_{ij}} \otimes \pr^{*}_{\T^{2n}}\poi
\end{equation}
over $(U_i\cap U_j)\times \T^{2n}$.

We shall forget the connections (and thus all trivial bundles) for a moment. Due to the equivariance of the Poincaré bundle discussed in \cref{sec:poincare}, the lifts $a_{ij}$ and $\hat a_{ij}$ determine an isomorphism $\tilde  a_{ij}^{*}\pr_{\T^{2n}}\poi \cong \pr_{\T^{2n}}^{*}\poi$.
Using this in \cref{eq:cpi}, we are in the situation that all bundles are pulled back along the projection $(U_i \cap U_j) \times \T^{2n} \to\T^{2n}$. Hence, these bundles must already have been isomorphic before pullback; and we conclude that there exists a bundle isomorphism
\begin{equation*}
\poi \otimes \pr^{*}\poi_{D_{ij}} \cong \hat \pr^{*}\poi_{\hat D_{ij}} \otimes \poi\text{.}
\end{equation*}
over $\T^{2n}$. Hence, there also exists a bundle isomorphism
\begin{equation*}
\pr^{*}\poi_{D_{ij}} \cong \hat \pr^{*}\poi_{\hat D_{ij}}\text{,}
\end{equation*}
and this shows that both bundles separately are trivializable. This implies $D_{ij}=\hat D_{ij}=0$. 
\end{proof}

\begin{remark}
The principal $\T$-bundle $L_{ij}$ and $\hat L_{ij}$ can be regarded as part of the gluing data for the bundle gerbes $\mathcal{G}$ and $\hat{\mathcal{G}}$, respectively. Their triviality in case of geometric (or only topological) T-duality shows that T-backgrounds that can be part of a  T-duality correspondence are of a special kind. More precisely, it means exactly that the Dixmier-Douady classes of $\mathcal{G}$ and $\hat{\mathcal{G}}$ are in the second step of the filtration of $\h^3(E,\Z)$  that comes from the Serre spectral sequence, see \cite{Bunke2006a} and \cite[\S 2.1]{Nikolause}.
\end{remark}

Next we will spend some time on finding trivializations $\lambda_{ij}$ and $\hat\lambda_{ij}$ with  particular  covariant derivatives $A_{ij}$ and $\hat A_{ij}$. 
We start with arbitrary choices as they exist by \cref{lem:extract} and will then perform three revisions of the isomorphisms $\lambda$ and $\hat\lambda$, and accordingly shift the 1-forms $A_{ij}$ and $\hat A_{ij}$, finally arriving at \cref{ex:cc:br2}.

We will only discuss $A_{ij}$, the treatment of $\hat A_{ij}$ is analogous.  
We remark that due to \cref{lem:buscherrulesBequiv}, \cref{eq:equi:B}, the 2-form $ a_{ij}^{*}B_j - B_i$ is $\T^{n}$-invariant; moreover, we have
\begin{equation}
\label{eq:equivcurvL}
( a_{ij}^{*}B_j - B_i)_{1+2}= (a_{ij}^{*}B_j - B_i)_{1}-\hat a_{ij}^{*}\theta\dot \wedge \theta_2
\end{equation}
over $(U_i \cap U_j) \times \T^{2n}$.
Here, we use the notation introduced in \cref{sec:poincare}: an index $(..)_{\alpha}$ means a pullback from the $\alpha$-th $\T^{n}$-factor, and the index $(..)_{1+2}$ means  a pullback along the addition of two $\T^{n}$-factors.  
\cref{ex:cc:gerbe1,eq:equivcurvL} imply
\begin{equation*}
\mathrm{d}((A_{ij})_2-(A_{ij})_1)=-\mathrm{d}(\hat a_{ij}\theta_{2-1})\text{.}
\end{equation*}
This shows that we have a closed 1-form 
\begin{align}
\label{eq:defalpha}
\alpha_{ij} &:=(A_{ij})_2-(A_{ij})_1+\hat a_{ij}\theta_{2-1} \in \Omega^1_{cl}((U_i \cap U_j) \times \T^{2n} )\text{.} 
\end{align}
Since the de Rham cohomology class of $\alpha_{ij}$ can only have contributions from the torus, and these contributions must be linear combinations of the generators $[\theta] \in \h^1(S^1,\R)$, there exists a smooth map $\beta_{ij}:(U_i \cap U_j) \times \T^{2n} \to \R$ and vectors $p_{ij},q_{ij}\in \R^{n}$ such that
\begin{equation}
\label{eq:sdfhdsfsd}
\alpha_{ij} = \mathrm{d}\beta_{ij} + p_{ij}\theta_1 + q_{ij}\theta_2\text{.}
\end{equation} 
Moreover, since the definition of $\alpha_{ij}$ is skew-symmetric with respect to the exchange of the two $\T^{n}$-factors; we have $q_{ij}=-p_{ij}$. 
We may now shift the isomorphism $\lambda_{ij}$ by the smooth map 
\begin{equation*}
(U_i \cap U_j) \times \T^{n} \to \T:(x,a)\mapsto -p_{ij}a\text{.}
\end{equation*}
Its derivative is $-p_{ij}\theta$; thus, $A_{ij}$ becomes replaced by $A_{ij}+p_{ij}\theta$, and \cref{eq:sdfhdsfsd} is replaced by just
\begin{align}
\label{eq:sdfhdsfs}
\alpha_{ij}=  \mathrm{d}\beta_{ij} \text{.}
\end{align}
In particular,  we have shown that $\lambda_{ij}$ can be chosen such that $\alpha_{ij}$ is trivial in de Rham cohomology. 
The left hand side is still skew-symmetric, and so we have
$\mathrm{d}(\beta_{ij} + s^{*}\beta_{ij})=0$,
where $s$ is the map that swaps the $\T^{n}$ factors. 
This means that $c_{ij}:=\beta_{ij}(x,a,b)  + \beta_{ij}(x,b,a)$ is a constant function. Shifting $\beta_{ij}$ by $-\frac{1}{2}c_{ij}$, we can achieve that $c_{ij}=0$, i.e., achieve that $\beta_{ij}$ is skew-symmetric in $a$ and $b$.  

Over $(U_i \cap U_j) \times \T^{3n}$ one can deduce from \cref{eq:defalpha} the cocycle condition
\begin{equation*}
(\mathrm{d}\beta_{ij})_{1,3} = (\mathrm{d}\beta_{ij})_{1,2} + (\mathrm{d}\beta_{ij})_{2,3}\text{.}
\end{equation*}
This shows that there exists a constant $c_{ij}\in \R$ such that 
\begin{equation*}
\beta_{ij}(x,a,c)= \beta_{ij}(x,b,c)+\beta_{ij}(x,a,b) + c_{ij}
\end{equation*} 
for all $a,b,c\in \T^{n}$. 
Putting $a=b=c$ shows that $c_{ij}=0$. Putting $b=0$ implies that
\begin{equation*}
\beta_{ij}(x,a,c)= \beta_{ij}(x,0,c)+\beta_{ij}(x,a,0)\text{.}
\end{equation*} 
Thus, we may define $\tilde\beta_{ij}: (U_i \cap U_j) \times \T^{n} \to \R$ by $\tilde\beta_{ij}(x,a):= \beta_{ij}(x,a,0)$ and obtain, using the skew-symmetry of $\beta_{ij}$,
\begin{equation*}
\beta_{ij}(x,a,b) = \tilde\beta_{ij}(x,a) - \tilde\beta_{ij}(x,b)\text{.}
\end{equation*}
We are now in position to  make a second revision of the choice of the isomorphism $\lambda_{ij}$, and shift it by the smooth map $(U_i \cap U_j) \times \T^{n} \to \T:(x,a)\mapsto -\tilde \beta_{ij}(x,a)$. This shifts $A_{ij}$ by $\mathrm{d}\tilde \beta_{ij}$.
Then, \cref{eq:sdfhdsfs} is replaced by $\alpha_{ij}=0$, and \cref{eq:defalpha} results in
\begin{align}
\label{eq:ex:Aequiv}
(A_{ij})_2-(A_{ij})_1&=-\hat a_{ij}\theta_{2-1}\text{.}
\end{align} 
On the dual side, we obtain analogously
\begin{align}
\label{eq:ex:Ahatequiv}
(\hat A_{ij})_{2}- (\hat A_{ij})_1&=-a_{ij}\theta_{2-1}\text{.}
\end{align}
 
Next we have to bring $A_{ij}$ and $\hat A_{ij}$ together, and consider for this purpose the connection-preserving isomorphism $\eta_{ij}$ of \cref{eq:cpi}.
By  \cref{lem:extract}, it simplifies to a connection-preserving isomorphism
\begin{equation}
\label{eq:etanew}
\eta_{ij} : \tilde  a_{ij}^{*}\pr^{*}_{\T^{2n}}\poi \otimes \trivlin_{\pr^{*} A_{ij}} \to  \trivlin_{\hat\pr^{*}\hat A_{ij}} \otimes \pr^{*}_{\T^{2n}}\poi\text{.}
\end{equation}
As a result of the fixed lifts $a_{ij}$ and $\hat a_{ij}$, we obtain canonically a connection-preserving isomorphism
\begin{equation}
\label{eq:poieq}
R_{ij} :  \tilde  a_{ij}^{*}\pr_{\T^{2n}}^{*}\poi  \to \pr_{\T^{2n}}^{*}\poi\otimes \trivlin_{\psi_{ij}}
\end{equation}
over $(U_i \cap U_j) \times \T^{2n}$,
see \cref{eq:isorl}.
Here, $\psi_{ij} \in \Omega^1((U_i \cap U_j) \times \T^{2n})$ with
\begin{equation}
\label{eq:defpsi}
\psi_{ij} :=- a_{ij}\mathrm{d}\hat a_{ij} - a_{ij}\hat \pr^{*}\theta + \hat a_{ij}\pr^{*}\theta\text{,}
\end{equation}
Under the isomorphism $R_{ij}$ of \cref{eq:poieq} we obtain from \cref{eq:etanew} a connection-preserving bundle isomorphism \begin{equation*}
\trivlin_{\psi_{ij}}  \otimes \trivlin_{\pr^{*} A_{ij}} \cong  \trivlin_{\hat\pr^{*}\hat A_{ij}} \text{,}
\end{equation*}
which in turn corresponds via the bijection \cref{eq:trivialbundleisos} to a smooth map $h_{ij}:(U_i \cap U_j) \times \T^{2n} \to \T$ such that
\begin{equation}
\label{eq:dloghij}
\hat\pr^{*}\hat A_{ij} =  \pr^{*} A_{ij}+\psi _{ij} + h_{ij}^{*}\theta\text{.}
\end{equation}

\begin{lemma}
The maps $h_{ij}$ are independent of the $\T^{2n}$-factor.
\end{lemma}

\begin{proof}
Considering \cref{eq:dloghij} over $(U_i \cap U_j) \times \T^{4n}$ twice, 
\begin{align*}
(\hat A_{ij})_4 &=  (A_{ij})_2+ (\psi _{ij})_{2,4} + (h_{ij}^{*}\theta)_{2,4}
\\
(\hat A_{ij})_3 &=  (A_{ij})_1+ (\psi _{ij})_{1,3} + (h_{ij}^{*}\theta)_{1,3}\text{,}
\end{align*}
taking their difference, and using
 \cref{eq:ex:Aequiv,eq:ex:Ahatequiv,eq:defpsi} yields
\begin{equation*}
(h_{ij}^{*}\theta)_{2,4}  = (h_{ij}^{*}\theta)_{1,3}\text{.}
\end{equation*}
This implies that 
\begin{equation*}
z_{ij} :=  h_{ij}(x,b,\hat b)^{-1}\cdot h_{ij}(x,a,\hat a)\in \T
\end{equation*}
is a constant. Putting $a=b$ and $\hat a=\hat b$ shows that $z_{ij}=0$. We obtain $h_{ij}(x,b,\hat b)= h_{ij}(x,a,\hat a)$.
This shows the claim.
\end{proof} 

We now make one last revision of the choice of the isomorphism $\hat\lambda_{ij}$, and shift it by $h_{ij}$. This changes $\hat A_{ij}$ by $h_{ij}^{*}\theta$, and hence turns \cref{eq:dloghij} into
\begin{equation}
\label{ex:cc:br2}
\hat\pr^{*}\hat A_{ij} =  \pr^{*} A_{ij}+\psi _{ij}\text{.}
\end{equation}
Note that \cref{eq:ex:Aequiv,eq:ex:Ahatequiv} continue to hold, as a change by a 1-form that does not depend on $\T^{n}$ cancels itself on both sides.

The definition of the principal $\T$-bundle $L_{ij}$ induces  a canonical connection-preserving bundle isomorphism
\begin{equation*}
L_{ij} \otimes  a_{ij}^{*}L_{jk} \cong L_{ik} 
\end{equation*} 
over $(U_i \cap U_j \cap U_k) \times \T^{n}$. Under the trivialization $\lambda_{ij}$, it corresponds to a smooth map
\begin{equation*}
c_{ijk}: (U_i \cap U_j \cap U_k) \times \T^{n} \to \T
\end{equation*}
such that 
\begin{equation}
\label{ex:cc:gerbe2}
A_{ik}=A_{ij} +  a_{ij}^{*}A_{jk} + c_{ijk}^{*}\theta.
\end{equation}
Further, by going to a quadruple intersection, it is straightforward to see that we obtain a cocycle condition
\begin{equation}
\label{ex:cc:gerbe3}
 a_{ij}^{*}c_{jkl}\cdot c_{ijl} = c_{ijk}\cdot c_{ikl}\text{.}
\end{equation}
The same holds on the dual side, leading to a smooth map $\hat c_{ijk}$ satisfying
\begin{equation}
\label{ex:cc:gerbehat2}
\hat A_{ik}=\hat A_{ij} + \hat  a_{ij}^{*}\hat A_{jk} + \hat c_{ijk}^{*}\theta
\end{equation}
and the cocycle condition
\begin{equation}
\label{ex:cc:gerbehat3}
 a_{ij}^{*}\hat c_{jkl}\cdot \hat c_{ijl} = \hat c_{ijk}\cdot \hat c_{ikl}\text{.}
\end{equation}

\begin{lemma}
\label{eq:complicatedcc}
The following equation of maps $(U_i \cap U_j\cap U_k) \times \T^{2n} \to \T$ holds: 
\begin{equation*}
\hat\pr^{*} \hat c_{ijk} =  \pr^{*}c_{ijk}\cdot f_{ijk}\text{,}
\end{equation*}
where $f_{ijk}$ is defined by the expression 
\begin{equation*}
f_{ijk}: (U_i \cap U_j\cap U_k) \times \T^{2n} \to \T: (x,a,\hat a) \mapsto  \hat m_{ijk}a-m_{ijk}(\hat a +\hat a_{ik}(x))-a_{jk}(x)\hat a_{ij}(x)\text{.}
\end{equation*}
\end{lemma}

\begin{proof}
We put the diagrams \cref{eq:diagrametaij} for $ij$ and $jk$, respectively, next to each other. In the middle, two occurrences of $\tilde a_{jk}^{*}\xi_j$ cancel, and we obtain the following equality of connection-preserving 2-isomorphisms:
\begin{multline}
\label{eq:bigdia}
\alxydim{@C=4em}{\mathcal{I}_{\pr^{*}B_k} \ar@/^4pc/[rrr]^>>>>>>>>>>>>>>>>>>>>>>>>{\trivlin_{\pr^{*}A_{ik}}}="1" \ar@{=>}"1";[r]_>>>>>>>>>>>*+{\pr^{*}c_{ijk}^{-1}} \ar[r]^{ \trivlin_{\pr^{*} A_{jk}}} \ar[d]_{\pr^{*}_{\T^{2n}}\poi}  & \mathcal{I}_{\pr^{*} a_{jk}^{*}B_j}\ar@{=>}[dl]|*+{\eta_{jk}}\ar[d]|{\tilde  a_{jk}^{*}\pr^{*}_{\T^{2n}}\poi} \ar[r]^-{\trivlin_{\pr^{*}  a_{jk}^{*}A_{ij}}} & \mathcal{I}_{\pr^{*} a_{jk}^{*} a_{ij}^{*}B_i \ar@{=}[r] \ar@{=>}[dl]|*+{\tilde  a_{jk}^{*}\eta_{ij}} \ar[d]|{\tilde  a_{jk}^{*}\tilde  a _{ij}^{*}\pr_{\T^{2n}}^{*}\poi}} & \mathcal{I}_{\pr^{*}  a_{ik}^{*}B_i} \ar[d]^{\tilde  a_{ik}^{*}\pr_{\T^{2n}}^{*}\poi}
\\
\mathcal{I}_{\hat\pr^{*}\hat B_k + \Phi_k^{*}\rho} \ar@/_4pc/[rrr]_<<<<<<<<<<<<<<<<<<{\trivlin_{\hat\pr^{*}\hat A_{ik}}}="2" \ar[r]_{ \trivlin_{\hat\pr^{*}\hat A_{jk}}}& \mathcal{I}_{\hat\pr^{*}\hat a_{jk}^{*}\hat B_j + \tilde  a_{jk}^{*}\Phi_j^{*}\rho} \ar[r]_-{\trivlin_{\hat\pr^{*} \hat  a_{jk}^{*}\hat A_{ij}}} & \mathcal{I}_{\hat\pr^{*}\hat a_{jk}^{*}\hat  a_{ij}^{*}\hat B_i + \tilde  a_{jk}^{*}\tilde  a_{ij}^{*}\Phi_i^{*}\rho} \ar@{=>}"2"^*+{\hat\pr^{*}\hat c_{ijk}} \ar@{=}[r] & \mathcal{I}_{\hat\pr^{*}\hat  a_{ik}^{*}\hat B_i + \tilde  a_{ik}^{*}\Phi_i^{*}\rho}}
\\=
\alxydim{@C=4em}{\mathcal{I}_{\pr^{*}B_k} \ar[d] \ar[r]^{\trivlin_{\pr^{*}A_{ik}}} & \mathcal{I}_{\pr^{*} a_{ik}^{*}B_i} \ar@{=>}[dl]|*+{\eta_{ik}} \ar[d]^{\tilde a_{ik}^{*}\pr_{\T^{2n}}^{*}\poi}
\\
\mathcal{I}_{\hat\pr^{*}\hat B_k + \Phi_k^{*}\rho} \ar[r]_{\trivlin_{\hat\pr^{*}\hat A_{ik}}} & \mathcal{I}_{\hat\pr^{*}\hat  a_{ik}^{*}\hat B_i + \tilde  a_{ik}^{*}\Phi_i^{*}\rho}}
\end{multline}
Our choice of isomorphisms $L_{ij}\cong \trivlin_{A_{ij}}$ and $\hat L_{ij}\cong \trivlin_{\hat A_{ij}}$ is such that we have an equality
\begin{equation*}
\alxydim{@=4em}{\mathcal{I} \ar[r]^{\trivlin_{\pr^{*} A_{ij}}}\ar[d]_{\pr^{*}_{\T^{2n}}\poi} & \mathcal{I}\ar@{=>}[dl]|*+{\eta_{ij}} \ar[d]^{\tilde  a_{ij}^{*}\pr^{*}_{\T^{2n}}\poi}  \\ \mathcal{I} \ar[r]_{\trivlin_{\pr^{*} \hat A_{ij}}} & \mathcal{I}}= 
 \alxydim{@R=4em@C=8em}{\mathcal{I} \ar[r]^{\trivlin_{\pr^{*} A_{ij}}}\ar[d]_{\pr^{*}_{\T^{2n}}\poi} & \mathcal{I}\ar@{=}[dl] \ar@/^6pc/[d]^<<<<<<<<<<<<<{\tilde  a_{ij}^{*}\pr^{*}_{\T^{2n}}\poi}="1" \ar[d]|>>>>>{\pr^{*}_{\T^{2n}}\poi \otimes \trivlin_{\psi_{ij}}}="2"  \\ \mathcal{I} \ar[r]_{\trivlin_{\pr^{*} \hat A_{ij}}} & \mathcal{I} \ar@{=>}"1";"2"|<<<<<<<<<*+{R_{ij}}}
\end{equation*}
Substituting this in \cref{eq:bigdia} we collect on the left hand side an isomorphism $R_{jk} \circ \tilde a_{jk}^{*}R_{ij}$ and on the right hand side an isomorphism $R_{ik}$. We compute the relation between these two isomorphisms: 
\begin{align*}
R_{jk} \circ \tilde a_{jk}^{*}R_{ij}&= R_{a_{jk},\hat a_{jk}} \circ \tilde  a_{jk}^{*}R_{a_{ij},\hat a_{ij}} \\&\eqcref{eq:poieq:1} R_{a_{ij}+a_{jk},\hat a_{ij}+\hat a_{jk}}\cdot (a_{jk}\hat a_{ij})^{-1}\\&= R_{a_{ik}-m_{ijk},\hat a_{ik}-\hat m_{ijk}}\cdot (a_{jk}\hat a_{ij})^{-1}
\\&\eqcref{eq:poieq:2} R_{a_{ik},\hat a_{ik}}\cdot f_{-m_{ijk},-\hat m_{ijk}}\cdot (-m_{ijk}\hat a_{ik})\cdot (a_{jk}\hat a_{ij})^{-1}\\&= R_{ik}\cdot f_{ijk}\text{,}
\end{align*}
with $f_{ijk}$ as defined above. 
\end{proof}

We will see in the following sections that the differential forms and functions collected so far, and the conditions derived for them, are sufficient.

\subsection{Geometric T-duality cocycles}

\label{sec:localdata}

In this section we organize the local data extracted in the previous section. For this purpose, we fix the following definition.
A \emph{geometric T-duality cocycle} with respect to an open cover $\{U_i\}$ of $X$  consists of the following data:
\begin{enumerate}

\item 
Riemannian, $\T^{n}$-invariant metrics $g_i$ and $\hat g_i$ on $U_i \times \T^{n}$,

\item
2-forms $B_i,\hat B_i \in \Omega^2(U_i \times \T^{n})$,

\item
1-forms $A_{ij},\hat A_{ij} \in \Omega^1((U_i \cap U_j) \times \T^{n})$,

\item
smooth maps $a_{ij},\hat a_{ij}:U_i \cap U_j \to \R^{n}$,

\item
$m_{ijk},\hat m_{ijk}\in \Z^{n}$, and

\item
smooth maps $c_{ijk},\hat c_{ijk}:(U_i \cap U_j \cap U_k) \times \T^{n} \to \T$. 

\end{enumerate}
This local data is subject to the following conditions \cref{cc:bundle,cc:bundlehat,cc:metric,cc:metrichat,cc:gerbe,cc:gerbehat,cc:br1,cc:br2,cc:br3}. 
\begin{enumerate}[(LD1),leftmargin=3.5em]

\item
\label{cc:bundle}
The pair $(a_{ij},m_{ijk})$ is  local data for a principal $\T^{n}$-bundle $E$ over $X$, i.e.,
\begin{align*}
a_{ik} &= m_{ijk}+a_{ij}+a_{jk}
\\
m_{jkl}+m_{ijl} &= m_{ikl}+m_{ijk}\text{.}
\end{align*}
We remark that the second line follows from the first; it is only listed for convenience.

\item
\label{cc:bundlehat}
The pair $(\hat a_{ij},\hat m_{ijk})$ is local data for a principal $\T^{n}$-bundle $\hat E$ over $X$, i.e.,
\begin{align*}
\hat a_{ik} &= \hat m_{ijk}+\hat a_{ij}+\hat a_{jk}
\\
\hat m_{jkl}+\hat m_{ijl} &= \hat m_{ikl}+\hat m_{ijk}\text{.}
\end{align*}

\item
\label{cc:metric}
The metrics $g_i$ yield a metric on $E$, i.e.,
\begin{equation*}
 a_{ij}^{*}g_j = g_i\text{.}
\end{equation*} 

\item
\label{cc:metrichat}
The metrics $\hat g_i$ yield a metric on $\hat E$, i.e.,
\begin{equation*}
\hat  a_{ij}^{*}\hat g_j = \hat g_i\text{.}
\end{equation*} 

\item 
\label{cc:gerbe}
The triple $(B_i,A_{ij},c_{ijk})$ is local data for a bundle gerbe with connection over $E$, i.e.,
\begin{align*}
 a_{ij}^{*}B_j &= B_i+ \mathrm{d} A_{ij}
\\
A_{ik}&=A_{ij} +  a_{ij}^{*}A_{jk} + c_{ijk}^{*}\theta
\\
 a_{ij}^{*}c_{jkl}\cdot c_{ijl} &= c_{ijk}\cdot c_{ikl}
\end{align*}

\item 
\label{cc:gerbehat}
The triple $(\hat B_i,\hat A_{ij},\hat c_{ijk})$ is local data for a bundle gerbe with connection over $\hat E$, i.e.,
\begin{align*}
\hat  a_{ij}^{*}\hat B_j &= \hat B_i+ \mathrm{d} \hat A_{ij}
\\
\hat A_{ik}&=\hat A_{ij} + \hat  a_{ij}^{*}\hat A_{jk} + \hat c_{ijk}^{*}\theta
\\
\hat  a_{ij}^{*}\hat c_{jkl}\cdot \hat c_{ijl} &= \hat c_{ijk}\cdot \hat c_{ikl}
\end{align*}

\item
\label{cc:br1}
The pairs $(g_i,B_i)$ and $(\hat g_i,\hat B_i)$ satisfy the  Buscher rules.

\item
\label{cc:br2}
The \emph{second order Buscher rules} are satisfied:
\begin{align*}
\hat \pr^{*}\hat A_{ij} &= \pr^{*}A_{ij} -a_{ij}\hat\pr^{*}\theta+  \hat a_{ij}\pr^{*}\theta-a_{ij}\hat a_{ij}^{*}\theta\text{.}
\end{align*}

\item
\label{cc:br3}
The \emph{third order Buscher rules} are satisfied:
\begin{align*}
\hat c_{ijk}(x,\hat a)+m_{ijk}(\hat a_{ik}(x)+\hat a)= c_{ijk}(x,a)  +\hat m_{ijk}a-a_{ij}(x)\hat a_{jk}(x)\text{.}
\end{align*}

\end{enumerate}

The data of a geometric T-duality cocycle
are highly redundant; some of these redundancies are described in the following. A minimized version will be obtain in the context of topological T-duality (\cref{re:TDcocycles}) and differential T-duality (\cref{sec:diffTdualitycocycles}).  Despite of its lavish data content, a geometric T-duality cocycle clearly reflects the situation of a geometric T-duality correspondence, with data from both sides separated from each other, subject to the Buscher rules \cref{cc:br1,cc:br2,cc:br3} relating them.  

\begin{remark}
\label{re:localconnections}
Let $\omega_i,\hat\omega_i \in \Omega^1(U_i \times \T^{n},\R^{n})$ be the connections on the trivial bundle $U_i \times \T^{n}$ that are induced by the metrics $g_i$ and $\hat g_i$, respectively, under \cref{th:kaluzaklein}. We remark that by \cref{cc:metric,cc:metrichat} the bundle isomorphisms $ a_{ij}$ and $\hat a_{ij}$ are isometries, and hence connection-preserving by \cref{re:isometries}. Thus, by bijection \cref{eq:trivialbundleisos}, the connections transform under the transition functions as
\begin{equation*}
\omega_i =\omega_j+a_{ij}^{*}\theta
\quand
\hat\omega_i =\hat \omega_j+\hat a_{ij}^{*}\theta\text{.}
\end{equation*}
The connections in turn correspond to 1-forms $A_i,\hat A_i\in \Omega^1(U_i,\R^{n})$, via $\omega_i=(A_i)_1+\theta_2    $ and $\hat\omega_i= (\hat A_i)_1+\theta_2$, which then transform as
\begin{equation*}
A_i = A_j + a_{ij}^{*}\theta
\quand
\hat A_i = \hat A_j + \hat a_{ij}^{*}\theta\text{.}
\end{equation*}
By
\cref{cc:br1} and  \cref{lem:buscherandpoincare}, the equivariance rules of \cref{lem:buscherrulesBequiv} apply to $B_i$ and $\hat B_i$, i.e.
\begin{align}
\label{eq:equivBfromLD}
(B_i)_{1,2+3} &= (B_i)_{1,2}+ (\hat A_i)_1\wedge \theta_3
\\
\label{eq:equivhatBfromLD}
(\hat B_i)_{1,2+3} &=  (\hat B_i)_{1,2}+ (A_i)_1 \wedge \theta_3
\end{align}
over $U_i \times \T^{n} \times \T^{n}$. In particular, $\hat B_i$ and $B_i$ are $\T^{n}$-invariant. 
We may further consider the 3-forms $K_i := \omega_i \dot\wedge \hat F - \mathrm{d} B_i$ and $\hat K_i := \hat \omega_i \dot\wedge F-\mathrm{d}\hat B_i$ on $U_i \times \T^{n}$, where $F$ and $\hat F$ are the globally defined curvatures of the connections $\omega$ and $\hat\omega$, respectively. Using \cref{cc:br1} one can show that $K_i=\hat K_i$ and that they are the pullback of a globally defined 3-form $K\in\Omega^3(X)$ along $U_i \times \T^{n} \to X$. 
\end{remark}

\begin{remark}
\label{re:cocycles:1forms}
Similarly as proved in \cref{sec:exlocdata},
\cref{cc:br2} implies
\begin{align*}
(A_{ij})_2-(A_{ij})_1&=-\hat a_{ij}\theta_{2-1}\text{.}
\\
(\hat A_{ij})_{2}- (\hat A_{ij})_1&=-a_{ij}\theta_{2-1}\text{.}
\end{align*}
In particular, $A_{ij}$ and $\hat A_{ij}$ are $\T^{n}$-invariant. 
\end{remark}

\begin{remark}
\label{re:cequiv}
We notice that in \cref{cc:br3} the right hand side is independent of $\hat a$, and the left hand side is independent of $a$. In other words, the right hand side is constant in $a$, and the left hand side is constant in $\hat a$, and these two constants are equal. Explicitly, if we define
\begin{equation*}
t_{ijk}:U_i \cap U_j \cap U_k \to \T
\end{equation*}
to be this constant,
then we get
\begin{equation}
\label{eq:tijkcijkhatcijk}
 -\hat c_{ijk}(x,\hat a)-m_{ijk}(\hat a_{ik}(x)+\hat a) =t_{ijk}(x)= -c_{ijk}(x,a)-\hat m_{ijk}a +a_{ij}(x)\hat a_{jk}(x) 
\end{equation}
for all $a,\hat a\in \T^{n}$. We deduce from this the equivariance rules
\begin{align}
\label{eq:equivc}
c_{ijk}(x,a+a') &=  c_{ijk}(x,a)-\hat m_{ijk}a'
\\
\label{eq:equivchat}
\hat c_{ijk}(x,\hat a+\hat a') &=\hat c_{ijk}(x,\hat a)-m_{ijk}\hat a'
\end{align}
\end{remark}

\begin{remark}
\label{re:cc:overdetermined}
The Buscher rules \cref{cc:br1,cc:br2,cc:br3} determine $\hat g_i$, $\hat B_i$, $\hat A_{ij}$, and $\hat c_{ijk}$ uniquely. If $\hat g_i$, $\hat B_i$, $\hat A_{ij}$, and $\hat c_{ijk}$ exist and satisfy \cref{cc:br1,cc:br2,cc:br3}, one can in fact show that $\hat g_i$ is a $\T^{n}$-invariant Riemannian metric satisfying \cref{cc:metrichat}, and that $(\hat B_i,\hat A_{ij},\hat c_{ijk})$ satisfy \cref{cc:gerbehat}. The same holds upon exchanging quantities with hats and without. In other words, either \cref{cc:metric,cc:gerbe}, or \cref{cc:metrichat,cc:gerbehat} can be omitted in the above list of conditions. Since there is no way to decide which ones should be omitted, we kept both. 
\end{remark}

We will next describe the conditions under which two geometric T-duality cocycles are considered to be equivalent. We suppose that we have two cocycles 
\begin{align*}
(g_i,\hat g_i,B_i,\hat B_i,A_{ij},\hat A_{ij},a_{ij},\hat a_{ij},m_{ijk},\hat m_{ijk},c_{ijk},\hat c_{ijk})
\\
(g_i',\hat g_i',B_i',\hat B_i',A_{ij}',\hat A_{ij}',a_{ij}',\hat a_{ij}',m_{ijk}',\hat m_{ijk}',c_{ijk}',\hat c_{ijk}')
\end{align*}
with respect to the same open cover $\{U_i\}$. These are considered to be equivalent, if there exist:
\begin{enumerate}

\item
1-forms $C_i,\hat C_i \in \Omega^1(U_i \times \T^{n})\text{,}$

\item
smooth maps $p_i,\hat p_i: U_i \to \R^{n}$, 

\item 
numbers $z_{ij},\hat z_{ij}\in \Z^{n}$,
and

\item
smooth maps $d_{ij},\hat d_{ij}:(U_i \cap U_j) \times \T^{n} \to \T$,

\end{enumerate}
such that the following conditions \cref{cce:bundles,cce:bundleshat,cce:metrics,cce:metricshat,cce:gerbe,cce:gerbehat,cce:buscher1,cce:buscher2} are satisfied. Abusing notation, we consider in the following the functions $p_i,\hat p_i$ eventually as maps $p_i,\hat p_i: U_i \times \T^{n} \to U_i \times \T^{n}$ given  by $(x,a)\mapsto  (x,a+p_i(x))$ and $(x,\hat a)\mapsto (x,\hat a+\hat p_i(x))$, respectively. 
\begin{enumerate}[(LD-E1),leftmargin=4em]

\item 
\label{cce:bundles}
The bundles $E$ and $E'$ corresponding to $(a_{ij},m_{ijk})$ and $(a'_{ij},m'_{ijk})$ are isomorphic:
\begin{align*}
a_{ij}' + p_i &= z_{ij}+p_j+a_{ij}
\\
m_{ijk}'+z_{ij}+z_{jk} &= z_{ik} +m_{ijk}
\end{align*}
We remark  that the second line follows from the first and \cref{cc:bundle}; it is only listed for convenience. 
\item
\label{cce:bundleshat}
The bundles $\hat E$ and $\hat E'$ corresponding to $(\hat a_{ij},\hat m_{ijk})$ and $(\hat a'_{ij},\hat m'_{ijk})$ are isomorphic:
\begin{align*}
\hat a_{ij}' + \hat p_i &= \hat z_{ij}+\hat p_j+\hat a_{ij}
\\
\hat m_{ijk}'+\hat z_{ij}+\hat z_{jk} &= \hat z_{ik}+\hat m_{ijk}
\end{align*}

\item
\label{cce:metrics}
Under the bundle isomorphism of \cref{cce:bundles}, the metrics $g$ and $g'$ corresponding to $g_i$ and $g_i'$ are identified:
\begin{equation*}
p_i^{*}g_i' = g_i
\end{equation*}

\item
\label{cce:metricshat}
Under the bundle isomorphism of \cref{cce:bundleshat}, the metrics $\hat g$ and $\hat g'$ corresponding to $\hat g$ and $\hat g'$ are identified:
\begin{equation*}
\hat p_i^{*}\hat g_i' = \hat g_i
\end{equation*} 

\item
\label{cce:gerbe}
The pair $(C_i,d_{ij})$ is a connection-preserving 1-isomorphism between the bundle gerbes corresponding to $(B_i,A_{ij},c_{ijk})$ and $(B_i',A_{ij}',c_{ijk}')$:
\begin{align*}
p_i^{*}B_i' &= B_i + \mathrm{d}C_i
\\
p_i^{*}A_{ij}' &= A_{ij} - C_i + a_{ij}^{*} C_j + d_{ij}^{*}\theta
\\
p_i^{*}c_{ijk}' &= c_{ijk} +d_{ik}-  d_{ij} -  a_{ij}^{*}d_{jk}
\end{align*}

\item
\label{cce:gerbehat}
The pair $(\hat C_i,\hat d_{ij})$ is a connection-preserving 1-isomorphism between the bundle gerbes corresponding to $(\hat B_i,\hat A_{ij},\hat c_{ijk})$ and $(\hat B_i',\hat A_{ij}',\hat c_{ijk}')$:
\begin{align*}
\hat p_i^{*}\hat B_i' &= \hat B_i + \mathrm{d}\hat C_i
\\
\hat p_i^{*}\hat A_{ij}' &= \hat A_{ij} - \hat C_i +\hat a_{ij}^{*} \hat C_j + \hat d_{ij}^{*}\theta
\\
\hat p_i^{*}\hat c_{ijk}' &= \hat c_{ijk} +\hat d_{ik}-  \hat d_{ij} - \hat  a_{ij}^{*}\hat d_{jk}
\end{align*}

\item
\label{cce:buscher1}
The following equality of 1-forms on $U_i \times \T^{2n}$ holds:
\begin{equation*}
\hat\pr^{*}\hat C_i - \pr^{*}C_i =-p_i\hat\pr^{*}\theta -p_i\mathrm{d}\hat p_i +\hat p_i\pr^{*}\theta  
\end{equation*}

\item
\label{cce:buscher2}
The following equality holds for all $(x,a,\hat a)\in (U_i \cap U_j) \times \T^{2n}$:
\begin{multline*}
d_{ij} (x,a)+\hat z_{ij}a-z_{ij}(\hat p_i(x)+ \hat a_{ij}'(x))+  \hat a_{ij}'(x)p_i(x)  = \hat d_{ij}(x,\hat a)+z_{ij}\hat a+\hat p_j(x)a_{ij}(x)\text{.} 
\end{multline*}

\end{enumerate}

\begin{remark}
\label{re:connection1formstransformunderequivalence}
Let $\omega_i,\omega_i' \in \Omega^1(U_i \times \T^{n},\R^{n})$ be the connections on the trivial bundle $U_i \times \T^{n}$ that are induced by the metrics $g_i$ and $g_i'$, respectively, under \cref{th:kaluzaklein}. We remark that the bundle isomorphism $p_{i}$ is an isometry, and hence connection-preserving by \cref{re:isometries}. Thus, the connections transform under the functions $p_{i}$ as
$\omega_i = \omega'_i + p_i^{*}\theta$.
The connections in turn correspond to 1-forms $A_i,A'_i\in \Omega^1(U_i,\R^{n})$, via $\omega_i=(A_i)_1+\theta_2$ and $\omega'_i=(A'_i)_1+\theta_2$, which then, according to \cref{eq:trivialbundleisos}, transform as
\begin{equation}
\label{eq:connection1formstransformunderequivalence}
A_i = A_i' + p_{i}^{*}\theta\text{.}
\end{equation}
Analogous formulas hold on the dual side, i.e.,
\begin{equation}
\label{eq:dualconnection1formstransformunderequivalence}
\hat A_i = \hat A_i' + \hat p_{i}^{*}\theta\text{.}
\end{equation}
\end{remark}

\begin{remark}
\label{re:equivC}
From \cref{cce:buscher1} one can derive the following equivariance rules over $U_i \times \T^{2n}$:
\begin{align*}
(C_i)_2 - (C_i)_1 &= \hat p_i\theta_{1-2}
\\
(\hat C_i)_2 - (\hat C_i)_1 &= p_i\theta_{1-2}
\end{align*}
\end{remark}

\begin{remark}
\label{re:dequiv}
We notice that in \cref{cce:buscher2} the left hand side is independent of $\hat a$, and the right hand side is independent of $a$. In other words, the right hand side is constant in $\hat a$, and the left hand side is constant in $a$, and these two constants are equal. If we define
\begin{equation*}
e_{ij}: U_i \cap U_j \to \T
\end{equation*}
to be this constant,
then we get, for all $a,\hat a\in \T^{n}$, the equality
\begin{equation*}
-d_{ij} (x,a)-\hat z_{ij}a+z_{ij}(\hat p_i(x)+ \hat a_{ij}'(x))-  \hat a_{ij}'(x)p_i(x)=e_{ij} (x)=-\hat d_{ij}(x,\hat a)-z_{ij}\hat a-\hat p_j(x)a_{ij}(x)\text{.}
\end{equation*}
From this, we can deduce the following equivariance properties:
\begin{align}
\label{eq:equivdij}
d_{ij}(x,a+a')&=d_{ij} (x,a)-\hat z_{ij}a'
\\
\label{eq:equivhatdij}
\hat d_{ij}(x,\hat a+\hat a') &= \hat d_{ij}(x,\hat a)-z_{ij}\hat a'
\end{align}
\end{remark}

The set of equivalence classes of geometric T-duality cocycles with respect to an open cover $\{U_i\}$ is denoted by $\LDgeo (\{U_i\})$.
A refinement $\{V_j\} \to \{U_i\}$ of open covers evidently induces a restriction map $\LDgeo (\{U_i\}) \to \LDgeo (\{V_j\})$, turning $\LDgeo$ into a direct system w.r.t. to refinements.  

\begin{definition}
The direct limit of $\LDgeo (\{U_i\})$ over refinements of open covers is denoted by $\LDgeo(X)$.
\end{definition}

With this precise definition of local data at hand, we will prove in the following two sections that $\LDgeo(X)$ indeed classifies geometric T-duality correspondences over $X$.

\subsection{Reconstruction of a geometric T-duality correspondence}

\label{sec:reconstruction}

In the following we describe a  procedure that constructs from a geometric T-duality cocycle
\begin{equation*}
(g_i,\hat g_i,B_i,\hat B_i,A_{ij},\hat A_{ij},a_{ij},\hat a_{ij},m_{ijk},\hat m_{ijk},c_{ijk},\hat c_{ijk})
\end{equation*}
a geometric T-duality correspondence in the sense of \cref{def:gtdc}.  First of all, 
the maps $a_{ij}$ and $\hat a_{ij}$ become (after exponentiation) $\T^{n}$-valued transition functions, and we let $E$ and $\hat E$ be the corresponding principal $\T^{n}$-bundles. Note that these come with canonical trivializations $\varphi_i$ and $\hat\varphi_i$ over $U_i$, which induce the given transition functions.
Due to \cref{cc:metric,cc:metrichat}, the locally defined metrics $g_i$ and $\hat g_i$ yield metrics on $E$ and $\hat E$, respectively, which are Riemannian and $\T^{n}$-invariant. 

Next we construct the bundle gerbe $\mathcal{G}$ over $E$. We define the surjective submersion $\pi:Y \to E$ by putting
\begin{equation*}
Y := \coprod_{i\in I} U_i \times \T^{n}
\end{equation*}
and
$\pi|_{U_i \times \T^{n}} := \varphi_i$. Over $Y$ we consider the 2-form $B$ defined by $B|_{U_i \times \T^{n}} := B_i$.
The fibre products over $E$ can be identified in the following way:
\begin{equation}
\label{eq:fibreproducts}
Y^{[k]} \cong \coprod_{(i_1,...,i_k) \in I^{k}} Y_{i_1,...,i_k}
\quith
Y_{i_1,...,i_k} := (U_{i_1}\cap ...\cap U_{i_k} )\times \T^{n}\text{,} 
\end{equation}
where the projection maps $\pr_j: Y^{[k]} \to Y$ become, under this identification, 
\begin{equation}
\label{projections}
\pr_j|_{Y_{i_1,...,i_k}}(x,a) = (i_j,x,a+a_{i_1i_j}(x))\text{.}
\end{equation}
We remark that the more general projections $\pr_{j_1,...,j_l}:Y^{[k]} \to Y^{[l]}$ can then be described using \cref{projections} in each component of the range separately.  

On $Y^{[2]}$ we define  the 1-form $A$ by $A|_{Y_{ij}} := A_{ij}$; then, the first line of \cref{cc:gerbe} implies $\pr_2^{*}B-\pr_1^{*}B=\mathrm{d}A$. 
We may interpret $A$ as  a connection on the trivial principal $\T$-bundle $L$ over $Y^{[2]}$, so that $\mathrm{d}A$ is its curvature. Finally, we define an isomorphism
\begin{equation*}
\mu: \pr_{12}^{*}L \otimes \pr_{23}^{*}L \to \pr_{13}^{*}L
\end{equation*}
over $Y^{[3]}$ as multiplication by the smooth map 
$-c: Y \to \T$, i.e., $-c|_{Y_{ijk}}:= -c_{ijk}$.
The second line of \cref{cc:gerbe} implies that $\mu$ is connection-preserving, and the third line implies that it satisfies the cocycle condition. 
This finishes the construction of the bundle gerbe $\mathcal{G}$.

Note that the pullback $\varphi_i^{*}\mathcal{G}$ comes with a canonical trivialization $\mathcal{T}_i: \varphi_i^{*}\mathcal{G} \to \mathcal{I}_{B_i}$ induced by the  section
\begin{equation*}
\alxydim{}{
& Y \ar[d]^{\pi}
\\
U_i \times \T^{n} \ar@{=}[ur] \ar[r]_-{\varphi_i} & E\text{.}}
\end{equation*}

On the dual side, the construction of $\hat{\mathcal{G}}$ is completely analogous, using \cref{cc:gerbehat}. In particular, we use the same manifold $Y$, but with the projection $\hat \pi:Y \to \hat E$ defined by $\hat\pi|_{Y_i}:=\hat\varphi_i$. 
In particular, $\hat\varphi_i^{*}\hat{\mathcal{G}}$ comes with a canonical trivialization $\hat{\mathcal{T}}_i: \hat\varphi_i^{*}\hat{\mathcal{G}} \to \mathcal{I}_{\hat B_i}$. 

It remains to construct the connection-preserving isomorphism $\mathcal{D}$ on correspondence space. 
We may  consider the commutative diagram 
\begin{equation*}
\alxydim{}{& Z \ar[d]^{\zeta} \ar[dl]_{\pr'} \ar[dr]^{\hat\pr'} \\ Y \ar[d]_{\pi} & E \times_X \hat E \ar[dr]^{\hat \pr} \ar[dl]_{\pr} & Y \ar[d]^{\hat \pi} \\ E \ar[dr]_{p} && \hat E \ar[dl]^{\hat p} \\ & X}
\end{equation*}
where 
\begin{equation*}
Z := \coprod_{i\in I} Z_i
\quith
Z_i:= U_i \times \R^{2n}\text{,}
\end{equation*}
and the maps are defined by $\zeta(i,x,a,\hat a) := (\varphi_i(x,a),\hat\varphi_i(x,\hat a))$ as well as $\pr'(i,x,a,\hat a):=(i,x,a)$ and $\hat\pr'(i,x,a,\hat a) := (i,x,\hat a)$. 
The fibre products of $\zeta:Z \to E \times_X \hat E$ can be identified as
\begin{equation*}
Z^{[k]} \cong \coprod_{i_1,...,i_k} Z_{i_1,...,i_k}
\quith Z_{i_1,...,i_k}:= U_{i_1}\cap ... \cap U_{i_k}  \times \R^{2n} \times \underbrace{\Z^{2n}\times ... \times \Z^{2n}}_{(k-1)\text{ times}}
\end{equation*}
under a diffeomorphism
\begin{equation*}
((i_1,x,a_1,\hat a_1),...,(i_k,x,a_k,\hat a_k)) \mapsto (i_1,...,i_k,x,a_1,\hat a_1,m_2,\hat m_2,...,m_k,\hat m_k)\text{,}
\end{equation*}
where the integers are defined by $a_{p}=a_{1}+a_{i_1i_p}(x)+m_{p}$ for $2\leq p \leq k$, and similarly for the $\hat m_p$. 

The bundle gerbes $\mathcal{G}$ and $\hat{\mathcal{G}}$ pull back to correspondence space and become bundle gerbes  with surjective submersion $\zeta$. 
Thus, we can construct the isomorphism $\mathcal{D}$ working over $Z$. 
For this, we need to find a smooth map  $z:Z^{[2]} \to \T$ 
and a 1-form $\omega\in\Omega^1(Z)$ such that 
\begin{align}
\label{eq:glfc:1}
\hat \pr'^{*}\hat B + \zeta^{*}\rho_{g,\hat g} &= \pr'^{*}B+ \mathrm{d}\omega &&\text{ over }Z=U_i \times \R^{2n} 
\\
\label{eq:glfc:2}
(\hat\pr'^{[2]})^{*}\hat A+\pr_1^{*}\omega &= (\pr'^{[2]})^{*}A+\pr_2^{*}\omega+z^{*}\theta&&\text{ over }Z^{[2]} 
\\
\label{eq:glfc:3}
(\hat\pr'^{[3]})^{*}\hat c +\pr_{12}^{*} z+\pr_{23}^{*}z &= \pr_{13}^{*}z + (\pr'^{[3]})^{*}c&&\text{ over }Z^{[3]}
\end{align}
hold. 
We define $\omega_i \in \Omega^1(U_i \times \R^{2n})$ by  
\begin{equation}
\label{eq:defomega}
\omega_i=- a\mathrm{d}\hat a
\end{equation}
where $(a,\hat a)$ are the coordinates of $\R^{2n}$, and define $\omega$ by $\omega|_{Z_i} := \omega_i$. Moreover, we define $z_{ij}:Z_{ij} \to \T$ by
\begin{equation}
\label{eq:xiij}
z_{ij}(x,a,\hat a,m_2,\hat m_2):=m_2\hat a+ \hat a_{ij}(x) m_2+  \hat a_{ij}(x)a\text{,}
\end{equation}
 and define $z$ by $z|_{Z_{ij}} := z_{ij}$.

\begin{lemma}
\label{lem:constrcorres}
Our definitions \cref{eq:defomega,eq:xiij} satisfy \cref{eq:glfc:1,eq:glfc:2,eq:glfc:3}. 
\end{lemma}

\begin{proof}
\Cref{eq:glfc:1} follows from \cref{cc:br1}, as $\mathrm{d}\omega = \Omega$. For the remaining equations, it is now important to understand the various projections $Z^{[k]} \to Z^{[l]}$, under above identifications. We have:
\begin{align*}
\pr_{1}(i,j,x,a,\hat a) &=(i,x,a,\hat a)
\\
\pr_{2}(i,j,x,a,\hat a) &=(j,x,a+a_{ij}(x)+m_2,\hat a+\hat a_{ij}(x)+\hat m_2)
\end{align*}
From this, we can calculate $\pr_1^{*}\omega$ and $\pr_2^{*}\omega$; together with \cref{cc:br2} this gives \cref{eq:glfc:2}.  
Finally, we have:
\begin{align*}
\pr_{12}(i,j,k,x,a, \hat a,m_2,\hat m_2,m_3,\hat m_3) &=(i,j,x,a,\hat a,m_2,\hat m_2)
\\
\pr_{23}(i,j,k,x,a, \hat a,m_2,\hat m_2,m_3,\hat m_3) &= (j,k,x,a_{ij}(x)+a+m_2,\hat a_{ij}(x)+\hat a+\hat m_2,
\\&\qquad\qquad\qquad\qquad m_{ijk}+m_3-m_2,\hat m_{ijk}+\hat m_3-\hat m_2) 
\\
\pr_{13}(i,j,k,x,a, \hat a,m_2,\hat m_2,m_3,\hat m_3) &= (i,k,x,a,\hat a,m_3,\hat m_3)  
\end{align*}
Then, a direct calculation shows that 
\begin{equation*}
(\pr_{12}^{*}z + \pr_{23}^{*}z - \pr_{13}^{*}z)|_{Z_{ijk}} = m_{ijk}(\hat a_{ik}(x)+\hat a)+  \hat a_{jk}(x)a_{ij}(x)- \hat m_{ijk}a\text{.}
\end{equation*}
This is, via \cref{cc:br3}, the claimed equality \cref{eq:glfc:3}.
\end{proof}

So far we have provided the structure of a geometric T-duality correspondence. It remains to prove the axioms. Conditions \cref{def:gtdc:1*,def:gtdc:2*} of \cref{def:gtdc} follow from \cref{cc:br1} via \cref{lem:br:equiv,lem:buscherandpoincare}. For \cref{def:gtdc:3*}, consider one of the open sets $U_i$, over which we have the trivializations $\varphi_i$ and $\hat \varphi_i$, and the trivializations $\mathcal{T}_i: \varphi_i^{*}\mathcal{G} \to \mathcal{I}_{B_i}$ and $\hat{\mathcal{T}}_i:\hat\varphi_i^{*}\hat{\mathcal{G}} \to \mathcal{I}_{\hat B_i}$ mentioned above.

\begin{lemma}
The principal $\T$-bundle with connection over $U_i \times \T^{2n}$ that corresponds to the connection-preserving bundle gerbe isomorphism
\begin{align*}
&\alxydim{@C=3.5em}{\mathcal{I}_{\pr^{*}B_i}=\pr^{*}\mathcal{I}_{B_i} \ar[r]^-{\pr^{*}\mathcal{T}^{-1}_i} &  \pr^{*}\varphi_i^{*}\mathcal{G} = \Phi_i^{*}\pr^{*}\mathcal{G}}
\\[-1em]
&\hspace{10em}\alxydim{@R=3em}{&&\ar[dll]^-{\Phi_i^{*}\mathcal{D}}\\ &}
\\[-1em]
&\hspace{6em}\alxydim{@C=4em}{\Phi_i^{*}(\hat\pr^{*}\hat{\mathcal{G}} \otimes \mathcal{I}_{\rho_{g,\hat g}})= \hat \pr^{*}\hat\varphi_i^{*}\hat{\mathcal{G}} \otimes \mathcal{I}_{\Phi_i^{*}\rho_{g,\hat g}} \ar[r]^-{\hat \pr^{*}\hat{\mathcal{T}_i} \otimes \id} &  \hat\pr^{*}\mathcal{I}_{\hat B_i} \otimes \mathcal{I}_{\Phi_i^{*}\rho_{g,\hat g}} = \mathcal{I}_{\hat \pr^{*}\hat B_i + \Phi^{*}\rho_{g,\hat g}}}
\end{align*}
is given w.r.t. the covering $Z_i \to U_i \times \T^{2n}$ by the connection 1-form $\omega_i\in\Omega^1(Z_i)$ and the transition function $z_{ii}: Z_i^{[2]} \to \T$.
\end{lemma}

\begin{proof}
All bundle gerbes and bundle gerbe isomorphisms that appear in the composition above just involve trivial principal $\T$-bundles. The composition has to be computed over a common refinement of all involved surjective submersions; here, $Z_i \to U_i \times \T^{2n}$ is sufficient. The trivializations contribute, since we work over a single open set $U_i$, the trivial functions $c_{iii}=1$ and $\hat c_{iii}=1$. It remains the contribution of $\Phi_i^{*}\mathcal{D}$, which is $z_{ii}$. For the connections, it is similar: the trivializations contribute $A_{ii}=0$ and $\hat A_{ii}=0$, and $\Phi_i^{*}\mathcal{D}$ contributes $\omega_i$.   
\end{proof}

It remains to notice that  $z_{ii}(x,a,\hat a,m,\hat m)=\hat am$. This function, as well as the 1-form $\omega_i$, are obviously pulled back along the following map of coverings:
\begin{equation*}
\alxydim{}{Z_i \ar[r] \ar[d] & \R^{2n} \ar[d] \\ U_i \times \T^{2n} \ar[r] & \T^{2n}}
\end{equation*}
Comparing with \cref{eq:transitionfunctionleftsection,eq:covderpoi}, we see that $z_{ii}$ and $\omega_i$ are the local data of the Poincaré bundle and its connection, w.r.t. the section $\chi_l:\R^{2n} \to \T^{2n}$.
This shows that \cref{def:gtdc:3*} is satisfied.  
 
\begin{remark}
\label{re:reconK}
Under reconstruction, the 3-forms $K\in\Omega^3(X)$ from \cref{re:K,re:localconnections} coincide. 
\end{remark} 
 
\subsection{Well-definedness of reconstruction under equivalence}

\label{sec:welldefinednessofreconstruction}

In this section we show that the reconstruction of a geometric T-duality correspondence from a geometric T-duality cocycle described in \cref{sec:reconstruction} is compatible with equivalences between correspondences (\cref{def:tcorr:equiv}) and  cocycles (\cref{sec:localdata}). For this purpose, we consider two geometric T-duality cocycles
\begin{align*}
(g_i,\hat g_i,B_i,\hat B_i,A_{ij},\hat A_{ij},a_{ij},\hat a_{ij},m_{ijk},\hat m_{ijk},c_{ijk},\hat c_{ijk})
\\
(g_i',\hat g_i',B_i',\hat B_i',A_{ij}',\hat A_{ij}',a_{ij}',\hat a_{ij}',m_{ijk}',\hat m_{ijk}',c_{ijk}',\hat c_{ijk}')
\end{align*}
and an equivalence between them provided by a tuple $(C_i,\hat C_i,p_i,\hat p_i,z_{ij},\hat z_{ij},d_{ij},\hat d_{ij})$. Moreover, we let $((E,g,\mathcal{G}),(\hat E,\hat g,\hat{\mathcal{G}}),\mathcal{D})$ and $((E',g',\mathcal{G}'),(\hat E',\hat g',\hat{\mathcal{G}'}),\mathcal{D}')$ be the geometric T-duality correspondences reconstructed from the two cocycles.   

The functions $p_i$ and $\hat p_i$ define bundle isomorphisms $p: E \to E'$ and $\hat p: \hat E \to \hat E'$ due to \cref{cce:bundles,cce:bundleshat}. 
It is straightforward to see using \cref{cce:metrics,cce:metricshat} that $p$ and $\hat p$ are isometric. 
Concerning the bundle gerbe $\mathcal{G}$ and $\mathcal{G}'$, we have a commutative diagram
\begin{equation*}
\alxydim{}{Y \ar[d]_{\pi} \ar[r]^{p'} & Y \ar[d]^{\pi'} \\ E \ar[r]_{p} & E'}
\end{equation*}
with $p'(i,x,a) := (i,x,a+p_i(x))$, i.e., $p'|_{Y_{i}}=p_i$. Thus, we may construct a bundle gerbe isomorphism $\mathcal{A}:\mathcal{G} \to p^{*}\mathcal{G}'$ using the common refinement
\begin{equation*}
\alxydim{@C=1em}{& Y \ar@/^1pc/[drrr]^{p'} \ar@{=}[ld] \ar[rd]^{(\pi,p')} \\ \quad Y\quad\  \ar[d]_{\pi} &&  E \ttimes p{\pi'} Y \ar[rr]^-{\pr_Y}  \ar[d]^{\pr_{E}} && Y \ar[d]^{\pi'} \\  E \ar@{=}[rr] && E \ar[rr]_{p} && E'}
\end{equation*}
of their surjective submersions.
We define the 1-form $C \in \Omega^1(Y)$ by setting $C|_{U_i \times \T^{n}} := C_i$, and consider the trivial bundle $Q:=\trivlin_{C}$ over $Y$. 
Then, the first ingredient of the isomorphism $\mathcal{A}$ is the equation $p'^{*}B' = B + \mathrm{curv}(Q)$, which follows immediately from the first equation in \cref{cce:gerbe}. The next part is to provide a connection-preserving bundle isomorphism
\begin{equation*}
\alpha: L \otimes \pr_2^{*}Q \to \pr_1^{*}Q \otimes (p'^{[2]})^{*}L'
\end{equation*}  
over $Y^{[2]}$. Since all bundles are trivial ($L=\trivlin_A$ and $L'=\trivlin_{A'}$), this is the same as a smooth map $d:Y^{[2]} \to \T$ such that
\begin{equation*}
\pr_1^{*}C+(p'^{[2]})^{*}A' = A+\pr_2^{*}C + d^{*}\theta\text{.}
\end{equation*}
Over $Y_{ij}=(U_i \cap U_j) \times \T^{n}$, this is a smooth map $d_{ij}: (U_i \cap U_j) \times \T^{n} \to \T$ such that
\begin{equation*}
C_i + p_i^{*}A'_{ij} = A_{ij} +  a_{ij}^{*} C_j + d_{ij}^{*}\theta\text{;}
\end{equation*}
thus, we can take the given data $d_{ij}$ according to the second equation in \cref{cce:gerbe}. 
Finally, we have to show that the diagram
\begin{equation*}
\alxydim{@C=8em}{\pr_{12}^{*}L \otimes \pr_{23}^{*}L \otimes \pr_3^{*}Q \ar[r]^-{\mu \otimes \id} \ar[d]_{\id \otimes \pr_{23}^{*}\alpha} & \pr_{13}^{*}L \otimes \pr_3^{*}Q \ar[dd]^{\pr_{13}^{*}\alpha} \\ \pr_{12}^{*}L \otimes \pr_2^{*}Q \otimes \pr_{23}^{*}(p'^{[2]})^{*}L' \ar[d]_{\pr_{12}^{*}\alpha \otimes \id} \\ \pr_1^{*}Q \otimes\pr_{12}^{*}(p'^{[2]})^{*}L' \otimes \pr_{23}^{*}(p'^{[2]})^{*}L' \ar[r]_-{\id \otimes (p'^{[3]})^{*}\mu'} & \pr_1^{*}Q \otimes \pr_{13}^{*}(p'^{[2]})^{*}L' }
\end{equation*}
of bundle isomorphisms over $Y^{[3]}$ is commutative. Restricting to $Y_{ijk}$, this means that
\begin{equation*}
d_{ik} + c_{ijk} = \rho_i^{*}c_{ijk}' + d_{ij} +  a_{ij}^{*}d_{jk}\text{,} 
\end{equation*}
which is the third equation in \cref{cce:gerbe}.
The dual side works precisely in an analogous way, using \cref{cce:gerbehat}.

It remains to produce the connection-preserving 2-isomorphism $\xi$ of \cref{def:tcorr:equiv}.
We consider a commutative diagram
\begin{equation*}
\alxydim{}{Z \ar[r]^{P'} \ar[d]_{\zeta} & Z \ar[d]^{\zeta'} \\ E \times_X \hat E \ar[r]_-{P} &  E' \times_X \hat E'} 
\end{equation*}
where $P:=p\times p'$, and $P': Z \to Z$ is defined by 
\begin{equation*}
P'|_{Z_i} :=\tilde p_i(x,i,a,\hat a) := (x,i,a+p_i(x),\hat a+\hat p_i(x))\text{.}
\end{equation*}
The 2-isomorphism $\xi$ is given by a function $w: Z \to \T$ satisfying:
\begin{align}
\label{eq:dfsf:1}
\pr^{*}C + P'^{*}\omega' &= \omega + \hat\pr^{*}\hat C +w^{*}\theta
\\
\label{eq:dfsf:2}
\pr^{*}d +(P'^{[2]})^{*}z'  &= z +\hat\pr^{*} \hat d + \pr_1^{*}w-\pr_2^{*}w
\end{align}
Here, $\omega,\omega'$ are the 1-forms \cref{eq:defomega} from the reconstruction of $\mathcal{D}$ and $\mathcal{D}'$, respectively, and $z,z'$ are the corresponding $\T$-valued functions \cref{eq:xiij}.

\begin{lemma}
The function $w(i,x,a,\hat a) := -\hat p_i(x)a$ satisfies \cref{eq:dfsf:1,eq:dfsf:2}.
\end{lemma}

\begin{proof}
We set $w_i := w|_{Z_i}$. 
Employing definitions, we find
\begin{align*}
\tilde p_i^{*}\omega_i' - \omega_i &=-p_i\mathrm{d}\hat a-a\mathrm{d}\hat p_i -p_i\mathrm{d}\hat p_i 
\\
w_i^{*}\theta &= -\hat p_i \mathrm{d} a - a\mathrm{d}\hat p_i\text{,}
\end{align*}
under which \cref{eq:dfsf:1} becomes \cref{cce:buscher1}.
In order to treat \cref{eq:dfsf:2} we need to compute the induced map $P'^{[2]}: Z^{[2]} \to Z^{[2]}$, resulting in
\begin{equation*}
(i,j,x,a,\hat a,m_2,\hat m_2)\mapsto  (i,j,x,a+p_i(x),\hat a+\hat p_i(x),m_2-z_{ij},\hat m_2-\hat z_{ij})\text{.}
\end{equation*}
Using this, \cref{eq:dfsf:2} becomes equivalent to
\begin{multline*}
d_{ij} (x,a)+ z'_{ij}(x,a+p_i(x),\hat a+\hat p_i(x),m_2-z_{ij},\hat m_2-\hat z_{ij})  \\=  z_{ij}(x,a,\hat a,m_2,\hat m_2) + \hat d_{ij}(x,\hat a) +  w_i(x,a,\hat a)- w_j(x,a+a_{ij}(x)+m_2,\hat a+\hat a_{ij}(x)+\hat m_2)\text{.}
\end{multline*}
Inserting the definitions of $z_{ij}$ and $w_i$, and once using \cref{cce:bundleshat}, one can see that the latter equation is equivalent to \cref{cce:buscher2}, hence satisfied.   
\end{proof}

Summarizing the work of \cref{sec:reconstruction,sec:welldefinednessofreconstruction}, we   have constructed a well-defined map
$\LDgeo(\{U_i\}) \to \GTC(X)$.
It is straightforward to see that this map is invariant under refinements of open covers, and hence induces a map
\begin{equation}
\label{eq:rec}
\LDgeo(X) \to \GTC(X)\text{.}
\end{equation}
In the next section we show that it is a bijection.

\subsection{Local-to-global equivalence}

\label{sec:localtoglobal}

In this section we prove the following result. 

\begin{proposition}
\label{th:localtoglobalequivalence}
The map \cref{eq:rec} is a bijection, 
\begin{equation*}
\LDgeo(X) \cong \GTC(X)\text{.}
\end{equation*}
\end{proposition}

We begin with showing surjectivity. 
Given a geometric T-duality correspondence $((E,g,\mathcal{G}),(\hat E,\hat g,\hat{\mathcal{G}}),\mathcal{D})$, we extract local data as explained in \cref{sec:exlocdata}, using  trivializations $\varphi_i,\hat\varphi_i$ of the $\T^{n}$-bundles, trivializations $\mathcal{T}_i,\hat{\mathcal{T}}_i$ of the bundle gerbes, and 2-isomorphisms $\xi_i$ as in \cref{eq:2isomorphismsxii}.  Let 
\begin{equation*}
(g_i,\hat g_i,B_i,\hat B_i,A_{ij},\hat A_{ij},a_{ij},\hat a_{ij},m_{ijk},\hat m_{ijk},c_{ijk},\hat c_{ijk})
\end{equation*}
be the local data obtained by these choices. 
Under reconstruction we obtain a new geometric T-duality correspondence $((E',g',\mathcal{G}'),(\hat E',\hat g',\hat{\mathcal{G}'}),\mathcal{D}')$. 

Obviously, bundle isomorphisms $\psi: E' \to E$ and $\hat\psi: \hat E' \to \hat E$ are given by $\psi([i,x,a]) := \varphi_i(x,a)$ and $\hat \psi([i,x,a]) := \hat\varphi_i(x,a)$. 
Concerning the bundle gerbes, our extraction procedure exhibits $\mathcal{G}$ as canonically isomorphic to a bundle gerbe defined as follows:
\begin{enumerate}

\item 
Its surjective submersion is $\pi: Y \to E$, where $\displaystyle Y=\coprod_{i\in I} U_i \times \T^{n}$ and $\pi(i,x,a) := \varphi_i(x,a)$.

\item
Its curving $B\in \Omega^2(Y)$ defined by $B|_{U_i \times \T^{n}} := B_i$.

\item
Its 
principal $\T$-bundle is $L_{ij}$, which is in turn isomorphic to $\trivlin_{A_{ij}}$ under the isomorphisms $\lambda_{ij}$, see \cref{lem:extract}.

\item
Its bundle gerbe product $\trivlin_{A_{ij}} \otimes  a_{ij}^{*}\trivlin_{A_{jk}} \to \trivlin_{A_{ik}}$ is induced by the map $c_{ijk}$, see \cref{ex:cc:gerbe2}.

\end{enumerate}
Since we have a commutative diagram
\begin{equation*}
\alxydim{}{Y \ar@{=}[r] \ar[d]_{\pi'} & Y \ar[d]^{\pi} \\ E' \ar[r]_{\psi} & E}
\end{equation*}
pulling back along the diffeomorphism $\psi$ leaves this structure as it is, yielding a bundle gerbe with connection $\mathcal{G}'$ over $E'$. We observe that $\mathcal{G}'$ is precisely the bundle gerbe reconstructed from the data $(B_i,A_{ij},c_{ijk})$. This way, we obtain a connection-preserving isomorphism $\mathcal{A}: \mathcal{G}' \to \psi^{*}\mathcal{G}$. Analogously, we  treat the dual side, and obtain another connection-preserving isomorphism $\hat{\mathcal{A}}:\hat{\mathcal{G}}' \to \hat\psi^{*}\hat{\mathcal{G}}$.

It remains to treat the correspondence isomorphism $\mathcal{D}$. 
\cref{eq:2isomorphismsxii} says that it becomes -- under the  isomorphisms $\mathcal{A}$ and $\hat{\mathcal{A}}$ -- 2-isomorphic  to an isomorphism $\mathcal{D}'$ defined over $Z:= \coprod U_i  \times \T^{2n}$, with bundle the Poincaré bundle $\pr_{\T^{2n}}^{*}\poi$,  and over $Z_{ij}$ the bundle isomorphism $\eta_{ij}$ of \cref{eq:etanew}. In more detail, this is a connection-preserving line bundle isomorphism
\begin{equation*}
\eta_{ij}: \pr^{*}\trivlin_{A_{ij}} \otimes  \tilde a_{ij}^{*}\pr_{\T^{2n}}^{*}\poi \to \pr_{\T^{2n}}^{*}\poi \otimes \hat\pr^{*}\trivlin_{\hat A_{ij}}
\end{equation*}
which was composed of the isomorphism
\begin{equation*}
R_{ij}: \tilde a_{ij}\pr_{\T^{2n}}^{*}\poi \to \pr_{\T^{2n}}^{*}\poi \otimes \psi_{ij}
\end{equation*}
from \cref{eq:poieq} and the equality $\hat \pr^{*}\hat A_{ij} = \pr^{*} A_{ij}+\psi _{ij}$ from \cref{ex:cc:br2}.
In order to compare $\mathcal{D}'$ with the reconstructed correspondence isomorphism, we change the covering of $\mathcal{D}'$ along
\begin{equation*}
Z':=\coprod_{i \in I} U_i  \times \R^{2n} \to Z\text{.}
\end{equation*} 
One can then trivialize the Poincaré bundle using the section $\chi_l: \R^{2n} \to \poi$, see \cref{sec:poincare}. This results into a 2-isomorphic 1-morphism $\mathcal{D}''$. As the covariant derivative of $\chi_l$ is the 1-form $\omega := -a\mathrm{d}\hat a$ on $\R^{2n}$, the principal $\T$-bundle of $\mathcal{D}''$ is $\trivlin_{\omega}$. Its isomorphism is the composite
\begin{equation*}
\alxydim{@C=3em}{\pr^{*}\trivlin_{A_{ij}} \otimes \pr_2^{*}\trivlin_{\omega} \ar[r]^-{\id \otimes \chi_l} &  \pr^{*}\trivlin_{A_{ij}}\otimes \tilde a_{ij}^{*}\pr_{\T^{2n}}^{*}\poi \ar[d]^-{\id \otimes R_{ij}} \\& \pr_{\T^{2n}}^{*}\poi \otimes \hat\pr^{*}\trivlin_{\hat A_{ij}} \ar[r]_-{\chi_l^{-1} \otimes \id} & \pr_1^{*}\trivlin_{\omega} \otimes \hat\pr^{*}\trivlin_{\hat A_{ij}}\text{,}  }
\end{equation*} 
where the projections $\pr_1,\pr_2: Z'^{[2]} \to Z'$ are as  in the proof of \cref{lem:constrcorres}. 
Using the formulas \cref{eq:chil,eq:Rexplicit} we can calculate this isomorphism explicitly:
\begin{equation*}
\alxydim{@C=-3em}{(a+a_{ij}(x)+m,\hat a+\hat a_{ij}(x)+\hat m,0) &\sim(a+a_{ij}(x),\hat a+\hat a_{ij}(x),-m\hat a-m\hat a_{ij}(x)) \ar[d]^{R_{ij}|_{a,\hat a}}\\ & (a,\hat a,-m\hat a-m\hat a_{ij}(x)-a\hat a_{ij}(x))  \mapsto (a,\hat a) \cdot -z_{ij}(x,a,\hat a,m,\hat m)\mquad}
\end{equation*}
where $z_{ij}$ was defined in \cref{eq:xiij}. This shows that the bundle isomorphism of $\mathcal{D}''$ is multiplication with $z_{ij}$. Hence, $\mathcal{D}''$ is precisely the reconstructed isomorphism, proving the surjectivity in \cref{th:localtoglobalequivalence}.

It remains to prove injectivity of reconstruction. For this purpose, we look at two geometric T-duality cocycles,
\begin{align*}
(g_i,\hat g_i,B_i,\hat B_i,A_{ij},\hat A_{ij},a_{ij},\hat a_{ij},m_{ijk},\hat m_{ijk},c_{ijk},\hat c_{ijk})
\\
(g_i',\hat g_i',B_i',\hat B_i',A_{ij}',\hat A_{ij}',a_{ij}',\hat a_{ij}',m_{ijk}',\hat m_{ijk}',c_{ijk}',\hat c_{ijk}')
\end{align*}
consider the corresponding reconstructed geometric T-duality correspondences $((E,g,\mathcal{G}),(\hat E,\hat g,\hat{\mathcal{G}}),\mathcal{D})$ and $((E',g',\mathcal{G}'),(\hat E',\hat g',\hat{\mathcal{G}'}),\mathcal{D}')$, and assume that these are equivalent in the sense of \cref{def:tcorr:equiv}.  Thus, there exist isometric bundle isomorphisms $ p:E \to E'$ and $\hat p: \hat E \to \hat E'$, connection-preserving bundle gerbe isomorphisms $\mathcal{A}:\mathcal{G} \to  p^{*}\mathcal{G}'$ and $\hat{\mathcal{A}}:\hat{\mathcal{G}} \to \hat p^{*}\hat{\mathcal{G}}'$, and a connection-preserving 2-isomorphism 
\begin{equation*}
\alxydim{@C=5em}{\pr^{*}\mathcal{G} \ar[r]^-{\mathcal{D}} \ar[d]_{\pr^{*}\mathcal{A}} & \hat\pr^{*}\hat{\mathcal{G}}\otimes \mathcal{I}_{\rho_{g_1,\hat g_1}} \ar[d]^{\hat\pr^{*}\hat{\mathcal{A}}\otimes \id} \ar@{=>}[ddl]|*+{\xi} \\ \pr^{*} p^{*}\mathcal{G}' \ar@{=}[d] & \hat\pr^{*}\hat{\mathcal{G}}' \otimes \mathcal{I}_{\rho_{g_1,\hat g_1}} \ar@{=}[d] \\  P^{*}\pr^{*}\mathcal{G}' \ar[r]_-{ P^{*}\mathcal{D}'} &   P^{*}\hat\pr^{*}\hat{\mathcal{G}}' \otimes  P^{*}\mathcal{I}_{\rho_{g_2,\hat g_2}}} 
\end{equation*}
where $ P:= p \times \hat p:E \times_X \hat E \to E' \times_X \hat E'$. It is straightforward to see that the isomorphisms $ p$ and $\hat p$ induce smooth maps $p_i,\hat p_i:U_i \to \R^{n}$ and $z_{ij},\hat z_{ij} \in \Z^{n}$ satisfying \cref{cce:bundles,cce:bundleshat,cce:metrics,cce:metricshat}.
Note that the surjective submersions of all 4 bundle gerbes have the same domain $Y = \coprod Y_i$, with $Y_i:=U_i \times \T^{n}$, and the bundle isomorphisms $ p$ and $\hat p$ lift to $Y^{[k]}$ as the component-wise defined maps 
\begin{equation*}
 p_{i_1}: (U_{i_1} \cap ... \cap U_{i_k}) \times \T^{n} \to (U_{i_1} \cap ... \cap U_{i_k}) \times \T^{n}:(x,a) \mapsto (x,a+p_{i_1}(x))\text{,}
\end{equation*}
and the analogous $\hat p_i$. 
We may thus assume that the isomorphisms $\mathcal{A}$ and $\hat{\mathcal{A}}$ consist of principal $\T$-bundles $Q$ and $\hat Q$ with connections over $Y$. Their restrictions to $Y_i$ will be denoted by $Q_i$ and $\hat Q_i$, respectively. The curvatures are  $\mathrm{curv}(Q_i) =  p_i^{*}B_i'-B_i$ and $\mathrm{curv}(\hat Q_i)=\hat  p_i^{*}\hat B_i'-B_i'$, and their  connection-preserving bundle isomorphisms over $Y^{[2]}\cong \coprod Y_{ij}$ are component-wise
\begin{equation*}
\chi_{ij}: \trivlin_{A_{ij}} \otimes \pr_2^{*}Q_{j} \to \pr_1^{*}Q_{i} \otimes \trivlin_{ p_i^{*}A'_{ij}}
\end{equation*} 
and an analogous  $\hat\chi_{ij}$.

As explained in the proof of \cref{lem:extract}, there exist bundle isomorphisms $Q_i\cong \pr_{\T^{n}}^{*}\poi_{F_i}$ and $\hat Q_i\cong \pr_{\T^{n}}^{*}\poi_{\hat F_i}$, where $F,\hat F \in \mathfrak{so}(n,\Z)$. The isomorphism $\chi_{ij}$ shows that $F_i=F_j$ and $\hat F_i =\hat F_j$, so that we can omit the indices. The 2-isomorphism $\xi$ induces over $(U_i \cap U_j) \times \T^{2n}$ an isomorphism
\begin{equation*}
\pr_{\T^{2n}}^{*}\poi \otimes \pr_{\T^{2n}}^{*}\hat \pr^{*}\poi_{\hat F} \cong \pr^{*}_{\T^{2n}}\pr^{*}\poi_{F} \otimes \pr_{\T^{2n}}^{*}\poi
\end{equation*}
which then implies $F=\hat F=0$. Thus, there exist 1-forms $C_i,\hat C_i\in \Omega^1(Y_i)$ and connection-preserving bundle isomorphisms $\kappa_i: Q_i \to \trivlin_{C_i}$ and $\hat \kappa_i : \hat Q_i\to \trivlin_{\hat C_i}$.
The isomorphisms $\chi_{ij}$ and $\hat\chi_{ij}$ then induce functions $d_{ij},\hat d_{ij}: (U_i \cap U_j) \times \T^{n} \to \T$ such that \cref{cce:gerbe,cce:gerbehat} are satisfied.

Note that we have $\mathrm{d}C_i=\mathrm{curv}(Q_i) =  p_i^{*}B_i'-B_i$. From this, \cref{eq:equivBfromLD,eq:equivhatBfromLD,eq:dualconnection1formstransformunderequivalence}
one can then derive
\begin{equation*}
(\mathrm{d}C_i)_{1,3} - (\mathrm{d}C_i)_{1,2} = -\mathrm{d}(\hat p_i\theta_{3-2})
\end{equation*}
over $U_i \times \T^{n} \times \T^{n}$. Now we proceed similar as in \cref{sec:exlocdata}.
We have a closed 1-form $\alpha_{i}\in \Omega^1_{cl}(U_i \times \T^{2n})$ defined by $\alpha_i := (C_i)_{1,3}- (C_i)_{1,2}+\hat p_i \theta_{3-2}$. Since the de Rham cohomology of $U_i \times \T^{2n}$ only has torus contributions, there exist a smooth map $\beta_i: U_i \times \T^{2n} \to \R$ and vectors $r_i,s_i\in \R^{n}$ such that
\begin{equation*}
\alpha_i = \mathrm{d}\beta_{i} + r_i \theta_2 + s_i\theta_3\text{.}
\end{equation*}
Moreover, since the definition of $\alpha_{i}$ is skew-symmetric with respect to the exchange of $a$ with $b$; this implies that $r_{i}=-s_{i}$. We may now shift the isomorphism $\kappa_i$ by the smooth map $U_i \times \T^{n} \to \T: (x,a)\mapsto r_i a$. This shifts $C_i$ by $r_i\theta$ and shows that
\begin{equation*}
(C_i)_{1,3}- (C_i)_{1,2}+\hat p_i\theta_{3-2}=\mathrm{d}\beta_{i}\text{.}
\end{equation*} 
Again, the left hand side is skew-symmetric, so that
\begin{equation*}
\mathrm{d}(\beta_{i}+s^{*}\beta_{i})=0,
\end{equation*}
where $s$ swaps the two $\T^{n}$-factors.
Thus, $c_i:=\beta_{i}+s^{*}\beta_{i}\in \R$ is a constant. Shifting $\beta_i$, we can achieve that this constant is zero, and that $\beta_i$ is skew-symmetric; moreover, 
defining $\tilde\beta_i: U_i \times \T^{n} \to \R$ by $\tilde\beta_i(x,a) := \beta_i(x,a,0)$, we obtain
\begin{equation*}
\beta_i(x,a,b) = \tilde \beta_i(x,a)-\tilde \beta_i(x,b)\text{.}
\end{equation*}
We may now shift $\kappa_i$ by the function $(x,a) \mapsto \tilde \beta_i(x,a)$, getting the formula
\begin{equation}
\label{eq:equivC}
(C_i)_{1,3} = (C_i)_{1,2}-\hat p_i\theta_{3-2}\text{.}
\end{equation}
On the dual side, we obtain analogously
\begin{equation}
\label{eq:equivhatC}
(\hat C_i)_{1,3} =(\hat C_i)_{1,2}-p_i\theta_{3-2}\text{.}
\end{equation}

We continue by looking at the local description of the 2-isomorphism $\xi$. We pullback to the space $Z=\coprod U_i \times \T^{2n}$, where, as $\mathcal{D}$ and $\mathcal{D}'$ are obtained by reconstruction, they consist of the trivial bundles with connections $\omega_{i},\omega'_i$ and of the bundle morphisms $z_{ij},z'_{ij}$ defined in \cref{eq:defomega,eq:xiij}.  
 Note that $\omega_i'=\omega_i$, whereas $z_{ij}$ and $z'_{ij}$ are different. Thus, the 2-isomorphism $\xi$ consists of smooth maps $w_{i}: U_i  \times \R^{2n}\to \T$ such that
\begin{equation}
\label{eq:eta1}
\tilde p_{i}^{*}\omega + \pr^{*}C_i=\pr_{\T^{2n}}^{*}\omega + \hat \pr^{*}\hat C_i + w_{i}^{*}\theta
\end{equation} 
and
\begin{equation}
\label{eq:eta2}
\tilde p_i^{*}z_{ij}'+\pr^{*} d_{ij} = \hat\pr^{*}\hat d_{ij}(x,\hat a)+z_{ij}-w_j + w_i
\end{equation}
also see \cref{eq:dfsf:1,eq:dfsf:2}.
We study now the dependence of $w_i$ on the first and the second $\R^{n}$-factor. 
From \cref{eq:eta1,eq:equivC} one can show that
\begin{equation*}
(w_i^{*}\theta)_{1,3,4} - (w_i^{*}\theta)_{1,2,4} = \mathrm{d}(\hat p_i (a-a'))
\end{equation*}
over $(x,a,a',\hat a)\in U_i \times \R^{n}\times \R^{n}\times \R^{n}$ holds. 
Similarly, \cref{eq:eta1,eq:equivhatC} imply that
\begin{equation*}
(w_i^{*}\theta)_{1,2,4} - (w_i^{*}\theta)_{1,2,3} =0\text{.}
\end{equation*}
In particular, defining $\tilde w_i: U_i \to \T$ by $\tilde w_i(x) := w_i(x,0,0)$, we have
\begin{equation*}
w_i^{*}\theta =\tilde w_i^{*}\theta-\mathrm{d}(a\hat p_i)
\end{equation*}
over $(x,a,\hat a)\in U_i \times \R^{2n}$.
Thus, there exists $z_i\in  \T$ such that
\begin{equation*}
w_i(x,a,\hat a) = \tilde w_i(x)- a\hat p_i(x) + z_i\text{.}
\end{equation*}
Putting $a=\hat a=0$ shows that $z_i=0$. We make a final revision of the isomorphism $\kappa_i$ by the function $\tilde w_i$. This changes $C_i$ to $C_i + \tilde w_i^{*}\theta$, and changes $w_i$ to just 
\begin{equation}
\label{eq:etafinal}
w_i(x,a,\hat a):= -a\hat p_i(x)\text{.}
\end{equation} 
Now, \cref{eq:eta1} becomes exactly \cref{cce:buscher1}. 
Finally, we consider  \cref{eq:eta2}. Using the definitions of $z_{ij}$ and $z'_{ij}$ from \cref{eq:xiij}, and using \cref{eq:etafinal}, \cref{eq:eta2} becomes \cref{cce:buscher2}; see the comments at the end of \cref{sec:welldefinednessofreconstruction}. This shows that the geometric T-duality cocycles we started with are equivalence, and completes the proof of injectivity of \cref{th:localtoglobalequivalence}.

\subsection{Local perspective to topological T-duality}
\label{re:TDcocycles}

In this section, we deduce from the local perspective to geometric T-duality obtained in \cref{sec:exlocdata,sec:localdata,sec:reconstruction,sec:welldefinednessofreconstruction,sec:localtoglobal} a corresponding local perspective to topological T-duality, and relate that to the non-abelian cohomology with values in the T-duality 2-group. 

We define a \emph{topological T-duality cocycle} as a geometric T-duality cocycle  with all metrics and differential forms stripped off. Thus, a topological T-duality cocycle is a tuple 
\begin{equation*}
(a_{ij},\hat a_{ij},m_{ijk},\hat m_{ijk},c_{ijk},\hat c_{ijk})
\end{equation*}
of data as in \cref{sec:localdata}, subject to conditions \cref{cc:bundle,cc:bundlehat}, only the last equations of \cref{cc:gerbe,cc:gerbehat}, and the third order Buscher rule \cref{cc:br3}. Two topological T-duality cocycles are considered to be equivalent if there
exist equivalence data $(z_{ij},\hat z_{ij},p_i,\hat p_i,d_{ij},\hat d_{ij})$  as in \cref{sec:localdata}, satisfying \cref{cce:bundles,cce:bundleshat}, the last equations of \cref{cce:gerbe,cce:gerbehat}, and \cref{cce:buscher2}. 
The direct limit of equivalence classes over refinement of open covers  will be denoted by $\LDtop(X)$.

Applying the reconstruction procedure of \cref{sec:reconstruction,sec:welldefinednessofreconstruction} to only the topological data  establishes a map
\begin{equation}
\label{eq:loctoptoTCtop}
\LDtop(X) \to \TTC(X)\text{.}
\end{equation}
In principle it could be argued similarly as in \cref{sec:localtoglobal} that this map is a bijection. However, we will prove this in a different way using the non-abelian differential cohomology $\h^1(X,\mathbb{TD})$ of the T-duality 2-group $\mathbb{TD}$, and a result of \cite{Nikolause}, see \cref{prop:localtop}.

The T-duality 2-group $\mathbb{TD}$ has been introduced in \cite[\S 3.2]{Nikolause}. Its definition and a general definition of non-abelian cohomology can be found there. Here we only recall the resulting definition of the set $\h^1(X,\mathbb{TD})$, see \cite[Rem. 3.7]{Nikolause}. An element in $\h^1(X,\mathbb{TD})$ is represented with respect to an open cover $\{U_i\}$ by a $\mathbb{TD}$-cocycle, a tuple $(a_{ij},\hat a_{ij},m_{ijk},\hat m_{ijk},t_{ijk})$, where the first four quantities are exactly as in geometric T-duality cocycles, and $t_{ijk}:U_i \cap U_j \cap U_k \to \T$ are smooth functions. The cocycle conditions are \cref{cc:bundle,cc:bundlehat}, and 
\begin{equation}
\label{eq:TDcoc}
t_{ikl}+t_{ijk}-m_{ijk}\hat a_{kl} = t_{ijl}+t_{jkl}\text{.}
\end{equation}
Two $\mathbb{TD}$-cocycles $(a_{ij},\hat a_{ij},m_{ijk},\hat m_{ijk},t_{ijk})$ and $(a'_{ij},\hat a'_{ij},m'_{ijk},\hat m'_{ijk},t'_{ijk})$ are equivalent if there exists a tuple $(z_{ij},\hat z_{ij},p_i,\hat p_i,\tilde e_{ij})$, with the first four quantities just as in the case of an equivalence between geometric T-duality cocycles, and smooth functions $\tilde e_{ij}:U_i \cap U_j \cap U_k \to \T$, satisfying \cref{cce:bundles,cce:bundleshat}, and 
\begin{equation}
\label{eq:TDcoceq}
t_{ijk}'+\tilde e_{ij}-\hat a'_{jk}z_{ij}+\tilde e_{jk} = \tilde e_{ik}+t_{ijk}-\hat p_km_{ijk}\text{.}
\end{equation} 
Then, $\h^1(X,\mathbb{TD})$ is a direct limit of equivalence classes of $\mathbb{TD}$-cocycles over refinement of open covers. We recall the following result. 

\begin{proposition} 
\label{prop:classTD}
{\normalfont\cite[Prop. 3.9]{Nikolause}}
There is a bijection
\begin{equation*}
\TTC(X) \cong \h^1(X,\mathbb{TD})\text{.}
\end{equation*} 
\end{proposition}

We will now describe a map
\begin{equation}
\label{eq:maploctop}
\LDtop(X) \to \h^1(X,\mathbb{TD}) 
\end{equation}
and prove that it is a bijection, see \cref{prop:ldequivtop}.  
Let $(a_{ij},\hat a_{ij},m_{ijk},\hat m_{ijk},c_{ijk},\hat c_{ijk})$ be a topological T-duality cocycle, representing an element in $\LDtop(X)$.
In \cref{re:cequiv} we have already defined the function
\begin{equation}
\label{cc:deft}
t_{ijk}(x) := -\hat c_{ijk}(x,0)-m_{ijk}\hat a_{ik}(x)\text{.} 
\end{equation}
A straightforward calculation using \cref{eq:equivchat} shows that $t_{ijk}$ indeed satisfies \cref{eq:TDcoc}.

\label{re:equivalenceofTDcocycles}

Given an equivalence between two topological T-duality cocycles 
\begin{align*}
(a_{ij},\hat a_{ij},m_{ijk},\hat m_{ijk},c_{ijk},\hat c_{ijk})
\\
(a_{ij}',\hat a_{ij}',m_{ijk}',\hat m_{ijk}',c_{ijk}',\hat c_{ijk}')\text{,}
\end{align*}
 established by a tuple $(z_{ij},\hat z_{ij},p_i,\hat p_i,d_{ij},\hat d_{ij})$, 
we consider a slight modification of the function $e_{ij}$ defined in \cref{re:dequiv}, namely,
we set\begin{align}
\nonumber \tilde e_{ij}(x) :\!&= e_{ij}(x)-z_{ij}(\hat a_{ij}(x)+\hat p_j(x))+\hat p_j(x)a_{ij}(x)
\\\label{eq:defetilde}&= - \hat d_{ij}(x,0) -z_{ij}(\hat a_{ij}(x)+\hat p_j(x))
\end{align}
Equivalently, using \cref{cce:buscher2}, we could write 
\begin{align}
\label{eq:tildee}
\tilde e_{ij}(x) &=-d_{ij} (x,0)-  \hat a_{ij}'(x)p_i(x)+\hat p_j(x)a_{ij}(x)\text{.}
\end{align}
One can then check  (using \cref{eq:defetilde} and the first equations of \cref{cce:bundles,cce:bundleshat}) that  $\tilde e_{ij}$ satisfies \cref{eq:TDcoceq}, i.e., it establishes an equivalence between the $\mathbb{TD}$-cocycles.
This completes the construction of the map \cref{eq:maploctop}.

\begin{lemma}
\label{prop:ldequivtop}
The map \cref{eq:maploctop} establishes a bijection,
\begin{equation*}
\LDtop(X) \cong \h^1(X,\mathbb{TD})\text{.}
\end{equation*}
\end{lemma}

\begin{proof}
We  construct an inverse map. 
If $(a_{ij},\hat a_{ij},m_{ijk},\hat m_{ijk},t_{ijk})$ is a $\mathbb{TD}$-cocycle, then we may restore $c_{ijk}$ and $\hat c_{ijk}$ from  formula \cref{eq:tijkcijkhatcijk} in \cref{re:cequiv}, namely
\begin{align}
\label{eq:defcfromt}
c_{ijk}(x,a)  &:= -t_{ijk}(x)-\hat m_{ijk}a+a_{ij}(x)\hat a_{jk}(x)
\\
\label{eq:defhatcfromt}
\hat c_{ijk}(x,\hat a) &:=-t_{ijk}(x)-m_{ijk}(\hat a_{ik}(x)+\hat a)\text{.} 
\end{align}
This satisfies obviously \cref{cc:br3}; and it is straightforward to show using \cref{eq:TDcoc}  that the last equations of \cref{cc:gerbe,cc:gerbehat} are satisfied. 
Hence, we obtain a topological T-duality cocycle. Moreover, this is strictly inverse to \cref{cc:deft}.

Next we consider an equivalence between $\mathbb{TD}$-cocycles $(a_{ij},\hat a_{ij},m_{ijk},\hat m_{ijk},t_{ijk})$ and $(a'_{ij},\hat a'_{ij},m'_{ijk},\hat m'_{ijk},t'_{ijk})$, established by a tuple $(p_i,\hat p_i,z_{ij},\hat z_{ij},\tilde e_{ij})$. We will then define
\begin{equation}
\label{eq:defefromtildee}
e_{ij}(x)  := \tilde e_{ij}(x) +z_{ij}(\hat a_{ij}(x)+\hat p_j(x))-\hat p_j(x)a_{ij}(x)
\end{equation}
and recover $d_{ij}$ and $\hat d_{ij}$ via \cref{re:dequiv}, namely,
\begin{align}
\label{eq:defdfrome}
d_{ij} (x,a) &:=-e_{ij} (x) -\hat z_{ij}a+z_{ij}(\hat p_i(x)+ \hat a_{ij}'(x))-  \hat a_{ij}'(x)p_i(x)
\\
\hat d_{ij}(x,\hat a) &:=-e_{ij} (x)-z_{ij}\hat a-\hat p_j(x)a_{ij}(x)-\hat p_j(x)p_j(x)\text{.}
\end{align}
This satisfies \cref{cce:buscher2} by definition, and the last equations of \cref{cce:gerbe,cce:gerbehat} follow from a straightforward computation.
\end{proof}

\begin{proposition}
\label{prop:localtop}
The maps from \cref{eq:maploctop,eq:loctoptoTCtop,prop:classTD} fit into a commutative diagram
\begin{equation*}
\alxydim{@C=0em@R=3em}{\LDtop(X) \ar[rr] \ar[dr] && \TTC(X)  \\ & \h^1(X,\mathbb{TD}) \ar[ur]}
\end{equation*}
in which all maps are bijections. 
\end{proposition}

\begin{proof}
The commutativity of the diagram needs to be checked using the definition of the map $\h^1(X,\mathbb{TD}) \to \TTC(X)$ from \cite{Nikolause}; this can be done in a straightforward way. Then, \cref{prop:classTD,prop:ldequivtop} show that all maps are bijections. 
\end{proof}

\section{Differential T-duality}

\label{sec:differentialTduality}

In this section we investigate the relation between geometric T-duality as discussed in \cref{sec:gtd} and a closed related notion of \quot{differential T-duality}. Differential T-duality can be seen as a reformulation of Kahle-Valentino's \quot{differential T-duality pairs}  \cite{Kahle}. It is an intermediate step between geometric and topological T-duality, in which just the metrics are replaced by their Kaluza-Klein connections. This intermediate step turns out to be useful for proving our main  \cref{th:main2}.

\subsection{Differential T-duality correspondences}

We first give a definition of differential T-duality  that fits into the setting of geometric and topological T-duality. This definition is very natural, but has not appeared anywhere else, as far as I know. The relation to  the work of Kahle-Valentino \cite{Kahle} will be described later in \cref{sec:kahle}.

\begin{definition}
A \emph{differential T-background} over $X$ is a triple $(E,\omega,\mathcal{G})$  consisting of a principal $\T^{n}$-bundle $E$ with connection $\omega$ over $X$ and  a bundle gerbe $\mathcal{G}$ with connection over $E$. Two differential T-backgrounds $(E,\omega,\mathcal{G})$ and $(E',\omega',\mathcal{G}')$ over $X$ are equivalent if there exists a connection-preserving bundle isomorphism $p: E \to E'$ and a connection-preserving bundle gerbe isomorphism $\mathcal{G} \cong p^{*}\mathcal{G}'$. The set of equivalence classes of differential T-backgrounds is denoted by $\DTB(X)$.
\end{definition}

Obviously, every geometric T-background $(E,g,\mathcal{G})$ induces a differential T-background $(E,\omega,\mathcal{G})$, where $\omega$ is the Kaluza-Klein connection of $g$. 
By \cref{th:kaluzaklein}, this establishes in fact a bijection
\begin{equation*}
\GTB(X) \cong \DTB(X) \times \mathrm{RieM}(X) \times C^{\infty}(X,\mathrm{PDS}(\R^{n}))  \text{,}
\end{equation*}
where $\mathrm{RieM}(X)$ is the set of all Riemannian metrics on $X$, and $\mathrm{PDS}(\R^{n})$ is the manifold of all positive-definite symmetric bilinear forms on $\R^{n}$. 
We see that \emph{differential} T-backgrounds are almost as good as \emph{geometric} T-backgrounds, up to independent global information.

Given two differential T-backgrounds $(E,\omega,\mathcal{G})$ and $(\hat E,\hat\omega,\hat{\mathcal{G}})$ over $X$, we consider again the correspondence space $E \times_X \hat E$ and the $\T^{2n}$-invariant 2-form
\begin{equation*}
\rho_{\omega,\hat\omega} := \hat\pr^{*}\hat\omega \dot\wedge \pr^{*}\omega\in \Omega^2(E \times_X \hat E)\text{.}
\end{equation*}

\begin{definition}
\label{def:dtc}
A \emph{differential T-duality correspondence} between two differential T-backgrounds $(E,\omega,\mathcal{G})$ and $(\hat E,\hat\omega,\hat{\mathcal{G}})$ is a connection-preserving isomorphism 
\begin{equation*}
\mathcal{D}: \pr^{*}\mathcal{G} \to \hat\pr^{*}\hat{\mathcal{G}} \otimes \mathcal{I}_{\rho_{\omega,\hat\omega}}
\end{equation*}
over $E \times_X \hat E$, such that every point $x\in X$ has an open neighborhood $U \subset X$  over which condition \cref{def:gtdc:3*} of \cref{def:gtdc} is satisfied.
\end{definition}

Here, it is understood that  the 2-form $\rho_{g,\hat g}$ that appears in \cref{def:gtdc:3*} is replaced by $\rho_{\omega,\hat\omega}$. 
We shall fix the following obvious observation.

\begin{proposition}
\label{prop:geotodiff}
Suppose $\mathcal{D}$ is a geometric T-duality correspondence between geometric T-backgrounds $(E,g,\mathcal{G})$ and $(\hat E,\hat g,\hat{\mathcal{G}})$. Then, $\mathcal{D}$ is  a differential T-duality correspondence between the induced differential T-backgrounds $(E,\omega,\mathcal{G})$ and $(\hat E,\hat\omega, \hat{\mathcal{G}})$. 
\end{proposition}

We also have the following converse result.

\begin{proposition}
\label{prop:upgradefromdifftogeo}
Suppose $\mathcal{D}$ is a differential T-duality correspondence between differential T-backgrounds  $(E,\omega,\mathcal{G})$ and $(\hat E,\hat \omega,\hat{\mathcal{G}})$. Then, there exist metrics $g$ on $E$ and $\hat g$ on $\hat E$ whose Kaluza-Klein connections are $\omega$ and $\hat\omega$, respectively, such that $\mathcal{D}$ is a geometric T-duality correspondence between $(E,g,\mathcal{G})$ and $(\hat E,\hat g,\hat {\mathcal{G}})$.
\end{proposition}

\begin{proof}
We choose a Riemannian metric $g'$ on $X$. Let $h: \R^{n} \times \R^{n} \to \R$ denote the standard inner product. We define $g$ to be the $\T^{n}$-invariant metric on $E$ corresponding to the triple $(\omega,g',h)$ under \cref{th:kaluzaklein}, and we define $\hat g$ to be the metric on $\hat E$ corresponding to $(\hat\omega,g',h)$. We have $\rho_{\omega,\hat\omega}=\rho_{g,\hat g}$, 
so that $\mathcal{D}$ has the correct structure of a geometric T-duality correspondence. Finally, we observe that it satisfies all three conditions,  \cref{def:gtdc:1*,def:gtdc:2*,def:gtdc:3*}.
\end{proof}

Equivalences between differential T-duality correspondences are defined analogous to \cref{def:tcorr:equiv}.
The set of equivalence classes of differential T-duality correspondences is denoted by $\DTC(X)$. \cref{prop:upgradefromdifftogeo} establishes a map
\begin{equation*}
\GTC(X) \to \DTC(X)\text{,}
\end{equation*}
and \cref{prop:upgradefromdifftogeo} shows that this map is surjective. In fact, there is a bijection
\begin{equation*}
\GTC(X) \cong \DTC(X) \times \mathrm{RieM}(X) \times C^{\infty}(X,\mathrm{PDS}(\R^{n}))\text{,}
\end{equation*}
under which a geometric T-duality correspondence $((E,g,\mathcal{G}),(\hat E,\hat g,\hat{\mathcal{G}}),\mathcal{D})$ corresponds to the triple $(((E,\omega,\mathcal{G}),(\hat E,\hat\omega,\hat{\mathcal{G}}),\mathcal{D}),g',h)$, where the metrics $g$ and $\hat g$ correspond under \cref{th:kaluzaklein} to the triples $(\omega,g',h)$ and $(\hat\omega,g',h^{-1})$, respectively.

The following result is more difficult to show, and its proof relies on the local formalism developed in \cref{sec:localformalism} and extended to differential T-duality below in  \cref{sec:diffTdualitycocycles}. 

\begin{proposition}
\label{prop:upgradefromtoptodiff}
Suppose $(E,\mathcal{G})$ and $(\hat E,\hat{\mathcal{G}})$ are topological T-backgrounds, and $\mathcal{D}$ is a topological T-duality correspondence between them.
Suppose further that $\omega$ and $\hat\omega$ are connections on $E$ and $\hat E$, respectively.  Then, there exist  connections on $\mathcal{G}$, $\hat{\mathcal{G}}$, and $\mathcal{D}$, such that $\mathcal{D}$ becomes a differential T-duality correspondence between $(E,\omega,\mathcal{G})$ and $(\hat E,\hat\omega,\hat{\mathcal{G}})$. 
\end{proposition}

\begin{proof}
\cref{prop:lift} in combination with \cref{prop:ldequivtop,prop:LDdifftohatHbij}. 
\end{proof}

The obvious composition of \cref{prop:upgradefromdifftogeo,prop:upgradefromtoptodiff}, about lifting topological T-duality correspondences to geometric ones, is stated as \cref{prop:upgradefromtoptogeo} in \cref{sec:topTduality}.
On the level of equivalence classes, it is clear  that the map $\GTC(X) \to \TTC(X)$ from \cref{sec:topTduality} factors as 
\begin{equation*}
\GTC(X) \to \DTC(X) \to \TTC(X)\text{,}
\end{equation*}
where both maps are surjective.

\subsection{Local perspective to differential T-duality}

\label{sec:diffTdualitycocycles}

In this section we develop a local description of differential T-duality. 
We modify the geometric T-duality cocycles considered in \cref{sec:localdata} by replacing the metrics $g_i$ and $\hat g_i$ by 1-forms $A_i,\hat A_i\in \Omega^1(U_i,\R^{n})$, and replacing  conditions \cref{cc:metric,cc:metrichat} by the following new conditions:
\begin{enumerate}[(LD1'),leftmargin=3.5em]
\setcounter{enumi}{2}

\item
\label{cc:metricprime}
$ A_j =  A_i - a_{ij}^{*}\theta$

\item
\label{cc:metrichatprime}
$\hat  A_j = \hat  A_i - \hat a_{ij}^{*}\theta$.

\end{enumerate}
Concerning equivalences between cocycles, we
keep the structure of an equivalence as it is, and replace conditions \cref{cce:metrics,cce:metricshat} by the new conditions:
\begin{enumerate}[(LD-E1'),leftmargin=5em]
\setcounter{enumi}{2}

\item
\label{cce:metricsprime}
$ A'_i =  A_i - p_{i}^{*}\theta$

\item
\label{cce:metricshatprime}
$\hat  A'_i = \hat  A_i - \hat p_{i}^{*}\theta$.

\end{enumerate}
The corresponding set of equivalence classes, and its direct limit over refinements of open covers will be denoted by $\LDdiff(X)$. Enforced by \cref{th:kaluzaklein}, and using \cref{re:localconnections,re:connection1formstransformunderequivalence}, there is a bijection
\begin{equation*}
\LDgeo(X) \cong \LDdiff(X) \times  \mathrm{RieM}(X) \times C^{\infty}(X,\mathrm{PDS}(\R^{n}))\text{,} 
\end{equation*}
obtained by replacing the metrics $g_i$ and $\hat g_i$ by the local connection 1-forms $A_i,\hat A_i$ of their Kaluza-Klein connections.

The reconstruction procedure described in \cref{sec:reconstruction,sec:welldefinednessofreconstruction}, together with the proof of \cref{th:localtoglobalequivalence}, goes through with obvious small modifications, so that we infer the following result. 

\begin{proposition}
\label{th:localtoglobalequivalencediff}
Reconstruction is a bijection, 
\begin{equation*}
\LDdiff(X) \cong \DTC(X)\text{.}
\end{equation*}
\end{proposition}

Next we set differential T-duality in relation to the \emph{differential} non-abelian cohomology of the T-duality 2-group $\mathbb{TD}$, whose investigation was started recently by Kim-Saemann \cite{Kim2022}. 
Differential non-abelian cohomology in general has been studied by Breen-Messing \cite{breen1} and further developed in \cite{schreiber5,schreiber2,schreiber2011}. A common phenomenon in higher gauge theory is the appearance of several versions of connection-data, which, in my review in \cite[\S 2.2]{Waldorf2016} are categorized into \emph{fake-flat}, \emph{regular}, and \emph{generalized}, with increasing generality. Thus, there are (at least) 3 versions of non-abelian differential cohomology with values in some Lie 2-group $\Gamma$,
related by maps\begin{equation*}
\hat\h^1(X,\Gamma)^{ff} \to \hat\h^1(X,\Gamma)^{reg} \to \hat\h^1(X,\Gamma)^{gen}
\end{equation*}
that commute with the projections to the (non-differential) non-abelian cohomology $\h^1(X,\Gamma)$.

 Additionally, Kim-Saemann have invented a formalism of \emph{adjusted} differential cohomology \cite{Kim2020,Kim2022}. It requires to equip the Lie 2-group $\Gamma$ with an additional structure, called an \emph{adjustment $\kappa$}. Together with an adjustment, there is another version of non-abelian differential cohomology denoted $\hat\h^1(X,\Gamma_{\kappa})$. It comes equipped with a map $\hat\h^1(X,\Gamma_{\kappa}) \to \hat\h^1(X,\Gamma)^{gen}$, and the choice $\kappa=0$ reproduces $\hat\h^1(X,\Gamma_0)=\hat\h^1(X,\Gamma)^{reg}$. Relevant for us will be the adjusted differential cohomology $\hat\h^1(X,\mathbb{TD}_{\kappa})$ of the Lie 2-group $\mathbb{TD}$.

In order to explain it on the basis of \cite[\S 2.2]{Waldorf2016} and \cite{Kim2022}, we need to express the Lie 2-group $\mathbb{TD}$ and its associated Lie 2-algebra as crossed modules (of Lie groups and Lie algebras, respectively). The crossed module of $\mathbb{TD}$ consists of the Lie group homomorphism
\begin{equation*}
\tau: H \to G,\qquad H:=\T \times \Z^{2n}, \qquad G:= \R^{2n},\qquad \tau(t,m,\hat m) := (m, \hat m)
\end{equation*}
and the action  $\alpha: G \times H \to H$ defined by $\alpha((a,\hat a), (t,m,\hat m)) := (t-\hat am,m,\hat m)$, see \cite[\S 3.2]{Nikolause}.
The corresponding crossed module of Lie algebras is trivial: it consists of the induced  Lie algebra homomorphism, $\tau_{*}=0$, and the induced action of the Lie algebra $\mathfrak{g}$ of $G$ on the Lie algebra $\mathfrak{h}$ of $H$, $\alpha_{*}=0$. Of relevance is further the differential of the action of a fixed element of $G$, $\alpha_{g}: H \to H$, which is here $(\alpha_{a,\hat a})_{*}=\id_{\R}$, and the differential of the map
\begin{equation*}
\tilde \alpha_h: G \to H: g \mapsto h^{-1}\alpha(g,h)\text{,}
\end{equation*}     
which is here $(\tilde \alpha_{t,m,\hat m})_{*}(a,\hat a) = -\hat am$. 

With these expressions at hand, we can recall the definition of $\hat\h^1(X,\mathbb{TD})^{gen}$ on the basis of \cite[\S 2.2]{Waldorf2016}. Thus, a generalized differential $\mathbb{TD}$-cocycle consists of a $\mathbb{TD}$-cocycle $(a_{ij},\hat a_{ij},m_{ijk},\hat m_{ijk},t_{ijk})$ as in \cref{re:TDcocycles}, and additionally of 1-forms $ A_i,\hat  A_i\in \Omega^1(U_i,\R^{n})$, a 2-form $R_i \in \Omega^2(U_i)$, and a 1-form $\varphi_{ij}\in\Omega^1(U_i \cap U_j)$ such that \cref{cc:metricprime,cc:metrichatprime} and
\begin{align}
\label{eq:diffcc:4}
\varphi_{ik} -\hat  A_k m_{ijk} &=  \varphi_{jk}+\varphi_{ij} -t_{ijk}^{*}\theta
\end{align}
are satisfied. 
Indeed, for an equivalence between generalized differential $\mathbb{TD}$-cocycles 
\begin{align*}
(A_i,\hat A_i,R_i,a_{ij},\hat a_{ij},\varphi_{ij},m_{ijk},\hat m_{ijk},t_{ijk})\\(A_i',\hat A_i',R_i',a_{ij}',\hat a_{ij}',\varphi_{ij}',m'_{ijk},\hat m'_{ijk},t'_{ijk})
\end{align*}
we require a tuple $(\phi_i,p_i,\hat p_i,z_{ij},\hat z_{ij},\tilde e_{ij})$, where  $\phi_i\in \Omega^1(U_i)$, and $(p_i,\hat p_i,z_{ij},\hat z_{ij},\tilde e_{ij})$ is, as in \cref{re:TDcocycles}, an equivalence between the $\mathbb{TD}$-cocycles $(a_{ij}',\hat a'_{ij},m'_{ijk},\hat m'_{ijk},t'_{ijk})$ and $(a_{ij},\hat a_{ij},m_{ijk},\hat m_{ijk},t_{ijk})$, i.e., it satisfies \cref{cce:bundles,cce:bundleshat}  and \cref{eq:TDcoceq}. 
Additionally, we require \cref{cce:metricsprime,cce:metricshatprime} and 
\begin{align}
\label{eq:equivalencedifferentialTD}
\varphi'_{ij} + \phi_i- z_{ij}\hat A_j'&=\phi_j+\varphi_{ij}-\tilde e_{ij}^{*}\theta\text{.}
\end{align}

We remark that the 2-form $R_i$ does not appear in any of the above conditions. This will be fixed by considering an adjustment $\kappa$ for $\mathbb{TD}$. In general, an adjustment is a map $\kappa: G \times \mathfrak{g} \to \mathfrak{h}$, and in case of $\mathbb{TD}$ Saemann-Kim \cite{Kim2022}  use 
\begin{equation*}
\kappa((a, \hat a),(b ,\hat b)) := a\hat b\text{.}
\end{equation*}
Then, an adjusted differential $\mathbb{TD}$-cocycle satisfies, in addition to the conditions listed above, the condition
\begin{equation}
\label{eq:diffcc:3}
R_j + \mathrm{d}\varphi_{ij} =R_i+a_{ij}\hat F\text{,}
\end{equation}
where $\hat F\in \Omega^2(X)$ is defined by $\hat F|_{U_i} = \mathrm{d}\hat A_i$. 
Moreover, for an equivalence between adjusted differential $\mathbb{TD}$-cocycles, we additionally require the condition
\begin{equation}
\label{eq:equivalencedifferentialTD2}
R_i'+\mathrm{d}\phi_i = R_i + p_i \hat F\text{.}
\end{equation}

\begin{remark}
\label{eq:adjusted3curvature}
The 3-curvature of an adjusted differential $\mathbb{TD}$-cocycle is, by definition, 
\begin{equation}
K:=\mathrm{d}R_i +  A_i \wedge \hat F \in \Omega^3(X)\text{.}
\end{equation}
\end{remark}

Having recalled the definition of the $\kappa$-adjusted differential cohomology of $\mathbb{TD}$, we are in position to  construct a map
\begin{equation}
\label{eq:lddifftoh1}
\LDdiff(X) \to \hat \h^1(X,\mathbb{TD}_{\kappa})\text{.}
\end{equation}
Given a differential T-duality cocycle $(A_i,\hat A_i,B_i,\hat B_i,A_{ij},\hat A_{ij},a_{ij},\hat a_{ij},m_{ijk},\hat m_{ijk},c_{ijk},\hat c_{ijk})$, we consider the underlying $\mathbb{TD}$-cocycle $(a_{ij},\hat a_{ij},m_{ijk},\hat m_{ijk},t_{ijk})$, where  $t_{ijk}$ was defined in \cref{re:cequiv}, namely,
\begin{equation*}
t_{ijk}(x) :=  -c_{ijk}(x,0)+a_{ij}(x)\hat a_{jk}(x)\text{.} 
\end{equation*}
This coincides with the expression given in \cref{cc:deft}, using \cref{cc:br3}.
We add the given 1-forms  $ A_i$ and $\hat A_i$, so that \cref{cc:metricprime,cc:metrichatprime} are satisfied as before. Let $\sigma: U_i \to U_i \times \T^{n}$ be the zero section, $\sigma(x):=(x,0)$. The 2-form $R_i$ is then defined by
\begin{equation}
\label{eq:defRi}
R_i := -\sigma^{*}B_i\text{,}
\end{equation}
and the 1-form $\varphi_{ij}$ is defined by
\begin{equation}
\label{eq:defvarphiij}
\varphi_{ij} := \sigma^{*}A_{ij}+a_{ij}\hat  A_j\text{.}
\end{equation}
It remains to check condition \cref{eq:diffcc:4} for generalized differential cocycles and the additional condition \cref{eq:diffcc:3} for adjusted differential cocycles. These are straightforward calculations; the first involving \cref{cc:gerbe,re:cocycles:1forms,cc:metrichatprime}, the second involving \cref{cc:gerbe,eq:equi:B}. 

Let us now suppose that we have an equivalence between two differential T-duality cocycles, established by a tuple $(C_i,\hat C_i,p_i,\hat p_i,z_{ij},\hat z_{ij},d_{ij},\hat d_{ij})$. We recall from \cref{re:equivalenceofTDcocycles} that the functions $p_i,\hat p_i:U_i \to \R^{2n}$ and $\tilde e_{ij}:U_i \cap U_j \to \T$, defined in \cref{eq:tildee} by
\begin{equation*}
\tilde e_{ij} :=-d_{ij} (x,0)-  \hat a_{ij}'(x)p_i(x)+\hat p_j(x)a_{ij}(x) 
\end{equation*}
establish an equivalence between the underlying two $\mathbb{TD}$-cocycles. Additionally,  conditions \cref{cce:metricsprime,cce:metricshatprime} remain valid.  
It remains to provide 1-forms $\phi_i \in \Omega^1(U_i)$ satisfying \cref{eq:equivalencedifferentialTD,eq:equivalencedifferentialTD2}.  
We set
\begin{equation}
\label{eq:defphii}
\phi_i := \sigma^{*}C_i + p_i \hat A_i'\text{.}
\end{equation}
Checking \cref{eq:equivalencedifferentialTD2} is straightforward using \cref{eq:equivBfromLD,cce:gerbe}.  
\cref{eq:equivalencedifferentialTD} is a bit more difficult to verify; one can first derive from \cref{cce:gerbe,re:cocycles:1forms,re:equivC} the formula
\begin{equation}
\label{eq:ds453}
\sigma^{*}A'_{ij}+ \sigma^{*}C_i= \sigma^{*}A_{ij} +\sigma^{*}C_j+\hat a_{ij}'\mathrm{d}p_i-\hat p_j\mathrm{d}a_{ij} + \sigma^{*}d_{ij}^{*}\theta\text{.}
\end{equation}
This formula together with \cref{cc:metrichatprime,cce:bundles}  proves \cref{eq:equivalencedifferentialTD}. This completes the construction of the map \cref{eq:lddifftoh1}. 

\begin{remark}
We recall from \cref{re:localconnections} that every geometric T-duality cocycle comes equipped with a globally defined 3-form $K\in \Omega^3(X)$, which  corresponds to the 3-form of a geometric T-duality correspondence, see \cref{re:K,re:reconK}.  Under the map $\LDgeo(X) \to \LDdiff(X)$, the same 3-form can be obtained from a differential T-duality cocycle, namely
\begin{equation*}
K|_{U_i} =  A_i \dot\wedge \hat F - \sigma^{*}\mathrm{d} B_i\text{.}
\end{equation*}
Under the map \cref{eq:lddifftoh1}, $\LDdiff(X) \to \hat\h^1(X,\mathbb{TD}_{\kappa})$, the 3-form $K$ is precisely the curvature of \cref{eq:adjusted3curvature}.
\end{remark}

\begin{lemma}
\label{prop:LDdifftohatHbij}
The map \cref{eq:lddifftoh1} is a bijection, 
\begin{equation*}
\LDdiff(X) \cong \hat\h^1(X,\mathbb{TD}_{\kappa})\text{.}
\end{equation*}
\end{lemma}

\begin{proof}
We suppose that we have an adjusted differential $\mathbb{TD}$-cocycle 
\begin{equation*}
( A_i,\hat A_i,R_i,\varphi_{ij} ,a_{ij},\hat a_{ij},m_{ijk},\hat m_{ijk},t_{ijk})\text{.}
\end{equation*} 
First, we reproduce, as in the proof of \cref{prop:ldequivtop}, the topological part of a differential T-duality cocycle, i.e., we define $c_{ijk}$ and $\hat c_{ijk}$ as in \cref{eq:defcfromt,eq:defhatcfromt}.
We further revert the assignments made in the definition of \cref{eq:lddifftoh1} using \cref{lem:buscherrulesBequiv}, and set
\begin{equation*}
B_i :=  -(R_i)_1 + (\hat A_i)_1\wedge \theta_2
\end{equation*}
on $U_i \times \T^{n}$.
Similarly, using \cref{re:cocycles:1forms}, we set
\begin{equation*}
A_{ij} := (\varphi_{ij})_1-a_{ij}(\hat  A_j)_1-\hat a_{ij}\theta_2\text{.}
\end{equation*}
One can then check using \cref{eq:diffcc:3,eq:diffcc:4} that the first and second lines of \cref{cc:gerbe} are satisfied (the third line is already checked in \cref{prop:ldequivtop}).
Finally, we define $\hat B_i$ and $\hat A_{ij}$ such that the Buscher rules \cref{cc:br1,cc:br2} are satisfied. 
As mentioned in \cref{re:cc:overdetermined}, it then follows automatically that \cref{cc:gerbehat} is satisfied. This shows the surjectivity of our map. 

For injectivity, we assume that two differential T-duality cocycles,
\begin{align*}
(A_i,\hat A_i,B_i,\hat B_i,A_{ij},\hat A_{ij},a_{ij},\hat a_{ij},m_{ijk},\hat m_{ijk},c_{ijk},\hat c_{ijk})
\\(A_i',\hat A_i',B_i',\hat B_i',A_{ij}',\hat A_{ij}',a_{ij}',\hat a_{ij}',m_{ijk}',\hat m_{ijk}',c_{ijk}',\hat c_{ijk}')
\end{align*} 
become equivalent after passing to $\hat\h^1(X,\mathbb{TD}_{\kappa})$.
That is, there exists a tuple $(\phi_i,p_i,\hat p_i,z_{ij},\hat z_{ij},\tilde e_{ij})$ satisfying \cref{cce:metricsprime,cce:metricshatprime,eq:equivalencedifferentialTD,eq:equivalencedifferentialTD2}, as well as the usual (non-differential) cocycle conditions \cref{cce:bundles,cce:bundleshat,eq:TDcoceq}.
We have seen in the proof of \cref{prop:ldequivtop} how to obtain $d_{ij}$ and $\hat d_{ij}$ such that the third lines of \cref{cce:gerbe,cce:gerbehat} and \cref{cce:buscher2} are satisfied.
It remains to provide 1-forms $C_i,\hat C_i \in \Omega^1(U_i \times \R^{n})$  such that the first two lines of \cref{cce:gerbe,cce:gerbehat}, and \cref{cce:buscher1} hold. We set
\begin{align*}
C_i &:= (\phi_i)_1 -p_i (\hat A_i')_1-\hat p_i \theta_2
\\
\hat C_i &:=  (\phi_i)_1 -p_i (\hat A_i)_1 -p_i\theta_2    
\end{align*}
on $U_i \times \T^{n}$.
The first line reverts \cref{eq:defphii}, and the second is chosen such that \cref{cce:buscher1} holds. 
The first line of \cref{cce:gerbe} con now be verified using \cref{eq:equivBfromLD,eq:defRi,eq:equivalencedifferentialTD2}, 
and the second line of \cref{cce:gerbe} can be verified using \cref{re:cocycles:1forms,eq:equivalencedifferentialTD,eq:defvarphiij}.
The two first lines of \cref{cce:gerbehat} can be checked analogously. This shows that the given differential T-duality cocycles are equivalent. 
\end{proof}

The identification of differential T-duality correspondences with the adjusted differential cohomology of $\mathbb{TD}$ has the advantage that the presentation with differential $\mathbb{TD}$-cocycles is less redundant than the one with differential T-duality cocycles: instead of two 2-forms $B_i$ and $\hat B_i$ there is only a single 2-form $R_i$, instead of $A_{ij}$ and $\hat A_{ij}$ there is only $\varphi_{ij}$, and instead of $c_{ijk}$ and $\hat c_{ijk}$ there is only $t_{ijk}$. Moreover, all data are defined on the open sets $U_i$ and intersections thereof, while the data of T-duality cocycles live on $U_i \times \T^{n}$ and their intersections. The following two results show that (adjusted) differential cohomology is very efficient for calculations. 
The first, \cref{prop:lift}, delivers the core ingredient to the proofs  of our main results \cref{th:main2,th:main:3}. 

\begin{proposition}
\label{prop:lift}
Every $\mathbb{TD}$-cocycle can be lifted to an adjusted differential $\mathbb{TD}$-cocycle, i.e., the map
\begin{equation*}
\hat\h^1(X,\mathbb{TD}_{\kappa}) \to \h^1(X,\mathbb{TD})
\end{equation*} 
is surjective. 
\end{proposition}

\begin{proof}
Given a $\mathbb{TD}$-cocycle $(a_{ij},\hat a_{ij},m_{ijk},\hat m_{ijk},t_{ijk})$, by the well-known existence of connections on principal bundles we find 1-forms $ A_i,\hat A_i\in\Omega^1(U_i,\R^{n})$ satisfying \cref{cc:metricprime,cc:metrichatprime}.  
We write \cref{eq:diffcc:4} as
\begin{equation*}
(\delta\varphi)_{ijk}  =   t_{ijk}^{*}\theta-\hat  A_k m_{ijk}\text{,}
\end{equation*}
where $\delta$ denotes the \v Cech coboundary operator. It is easy to check using \cref{eq:TDcoc}  that the right hand side is a \v Cech 2-cocycle; then, by the exactness of the \v Cech complex with values in the sheaf $\Omega^1$ it follows that $\varphi_{ij}$ exist such that \cref{eq:diffcc:4} is satisfied. 
Finally, we write \cref{eq:diffcc:3} as
\begin{equation*}
(\delta R)_{ij} =a_{ij}\hat F- \mathrm{d}\varphi_{ij}
\end{equation*}
and check again that the right hand side is a \v Cech 1-cocycle. 
This shows that $R_i$ exists such that \cref{eq:diffcc:3} is satisfied.
\end{proof}

Our second result concerns the action (\cref{re:actiongerbe,re:action}) of the group of isomorphism classes of bundle gerbes with connection, $\mathrm{Grb}^{\nabla}(X)$, on the set of equivalence classes of geometric T-duality correspondences, $\GTC(X)$. We recall that this action was induced by 
\begin{equation}
\label{eq:actionrecall}
(\mathcal{H} \;,\; ((E,g,\mathcal{G}),(\hat E,\hat g,\hat{\mathcal{G}}),\mathcal{D})) \mapsto ((E,g,\mathcal{G} \otimes p^{*}\mathcal{H}),(\hat E,\hat g,\hat{\mathcal{G}} \otimes \hat p^{*}\mathcal{H}),\mathcal{D} \otimes \id)).
\end{equation}
Since the action does not concern the metrics, there is a corresponding action on differential T-duality correspondences, which, under the bijections of \cref{th:localtoglobalequivalence,prop:LDdifftohatHbij}, becomes an action
\begin{equation}
\label{eq:actiondiffco}
\hat\h^3(X) \times \hat\h^1(X,\mathbb{TD}_{\kappa}) \to \hat\h^1(X,\mathbb{TD}_{\kappa})\text{.}
\end{equation}
It is straightforward to obtain a formula for \cref{eq:actiondiffco}:
a Deligne 2-cocycle acts on an adjusted differential $\mathbb{TD}$-cocycle by
\begin{multline*}
(B_i,A_{ij},c_{ijk}) \cdot ( A_i,\hat A_i,R_i,\varphi_{ij} ,a_{ij},\hat a_{ij},m_{ijk},\hat m_{ijk},t_{ijk}) \\:=( A_i,\hat A_i,R_i+B_i,\varphi_{ij}+A_{ij} ,a_{ij},\hat a_{ij},m_{ijk},\hat m_{ijk},t_{ijk}+c_{ijk})\text{.}
\end{multline*}

Next we consider the projection
\begin{equation*}
\GTC(X) \to \mathrm{Bun}^{\nabla}_{\T^{n}}(X) \times \mathrm{Bun}^{\nabla}_{\T^{n}}(X)
\end{equation*}
from a geometric T-duality correspondence to the two principal $\T^{n}$-bundles $E$ and $\hat E$, which can be equipped with the Kaluza-Klein connections $\omega,\hat\omega$ induced from the metrics $g$ and $\hat g$, respectively. This projection is obviously invariant under the action \cref{eq:actionrecall}. The same projection exists for differential T-duality correspondences, and then in adjusted differential cohomology,
\begin{equation*}
\hat\h^1(X,\mathbb{TD}_{\kappa}) \to \mathrm{Bun}^{\nabla}_{\T^{n}}(X) \times \mathrm{Bun}^{\nabla}_{\T^{n}}(X)\text{.}
\end{equation*} 
There, it is induced by the formula
\begin{equation*}
( A_i,\hat A_i,R_i,\varphi_{ij} ,a_{ij},\hat a_{ij},m_{ijk},\hat m_{ijk},t_{ijk}) \mapsto ((A_i,a_{ij}),(\hat A_i,\hat a_{ij}))\text{.}
\end{equation*}
Summarizing, we have a commutative diagram
\begin{equation*}
\alxydim{}{\mathrm{Grb}^{\nabla}(X) \times \GTC(X) \ar[d] \ar[r] & \GTC(X) \ar[d] \ar[r] & \mathrm{Bun}^{\nabla}_{\T^{n}}(X) \times \mathrm{Bun}^{\nabla}_{\T^{n}}(X) \ar@{=}[d] \\ \mathrm{Grb}^{\nabla}(X) \times \DTC(X) \ar[d] \ar[r] & \DTC(X) \ar[r] \ar[d] & \mathrm{Bun}^{\nabla}_{\T^{n}}(X) \times \mathrm{Bun}^{\nabla}_{\T^{n}}(X) \ar@{=}[d] \\ \hat\h^3(X) \times \hat\h^1(X,\mathbb{TD}_{\kappa}) \ar[r] & \hat\h^1(X,\mathbb{TD}_{\kappa}) \ar[r] & \mathrm{Bun}^{\nabla}_{\T^{n}}(X) \times \mathrm{Bun}^{\nabla}_{\T^{n}}(X)\text{.}}
\end{equation*}
 
Finally, we note that a pair $((E,\omega),(\hat E,\hat\omega))$ of isomorphism classes of bundles with connection has a well-defined pair $(F,\hat F) \in \Omega^2(X) \times \Omega^2(X)$ of curvatures. We consider the subgroup 
\begin{equation*}
\mathcal{F}_{F,\hat F} := \{ \mathcal{I}_{y \hat F + \hat y F} \sep y,\hat y\in \R \} \subset \mathrm{Grb}^{\nabla}(X)\text{.}
\end{equation*}
This is a non-trivial subgroup, as $\mathcal{I}_{B} \cong \mathcal{I}_{C}$ holds if and only if $C-B$ is a closed 2-form with integral periods. Now, $F$ and $\hat F$ are closed 2-forms with integral periods, but allowing real multiplies spoils integrality.

\begin{proposition}
\label{prop:actionfreeandtransitive}
The action of \cref{eq:actiondiffco},
\begin{equation*}
\hat\h^3(X) \times \hat\h^1(X,\mathbb{TD}_{\kappa}) \to \hat\h^1(X,\mathbb{TD}_{\kappa})\text{,}
\end{equation*}
ha      s the following properties: 
\begin{enumerate}[(i)]

\item 
It acts transitively in the fibres of the projection $\hat\h^1(X,\mathbb{TD}_{\kappa}) \to \mathrm{Bun}^{\nabla}_{\T^{n}}(X) \times \mathrm{Bun}^{\nabla}_{\T^{n}}(X)$. 

\item
The stabilizer of each element in the fibre over  $(\xi,\hat \xi)\in \mathrm{Bun}^{\nabla}_{\T^{n}}(X) \times \mathrm{Bun}^{\nabla}_{\T^{n}}(X)$ with curvature pair $(F,\hat F)$ is the subgroup $\mathcal{F}_{F,\hat F}$. 

\end{enumerate}
In particular, the quotient $\hat\h^3(X)/\mathcal{F}_{F,\hat F}$  acts freely and transitively on the fibre over $(\xi,\hat \xi)$.
\end{proposition}

\begin{proof}
We first show that $\mathcal{F}_{F,\hat F}$ stabilizes. For this, we have to provide an equivalence
\begin{equation*}
( A_i,\hat A_i,R_i,\varphi_{ij} ,a_{ij},\hat a_{ij},m_{ijk},\hat m_{ijk},t_{ijk}) \sim ( A_i,\hat A_i,R_i+y\hat F + \hat y F,\varphi_{ij},a_{ij},\hat a_{ij},m_{ijk},\hat m_{ijk},t_{ijk}) 
\end{equation*}
of adjusted differential cocycles. 
We set $p_i := y$ and $\hat p_i:=- \hat y$, as well as $z_{ij}=\hat z_{ij}=0$. Moreover, we put $\phi_i :=-  \hat y A_i$, and $\tilde e_{ij} :=-\hat y a_{ij}$. Now, \cref{cce:metricsprime,cce:metricshatprime} hold since $p_i$ and $\hat p_i$ are constant. \cref{eq:equivalencedifferentialTD,eq:equivalencedifferentialTD2} follow directly from the definitions. Finally, \cref{eq:TDcoceq} becomes
\begin{equation*}
\tilde e_{ij}+\tilde e_{jk} = \tilde e_{ik}+\hat ym_{ijk}
\end{equation*} 
and thus follows from \cref{cc:bundle}. 

Next we show that no other group elements stabilize. For this, we suppose that a Deligne 2-cocycle $(B_i,A_i,c_{ijk})$ acts trivially, i.e., that we have an equivalence between adjusted differential $\mathbb{TD}$-cocycles
\begin{multline*}
( A_i,\hat A_i,R_i,\varphi_{ij} ,a_{ij},\hat a_{ij},m_{ijk},\hat m_{ijk},t_{ijk}) \sim ( A_i,\hat A_i,R_i+B_i,\varphi_{ij}+A_{ij} ,a_{ij},\hat a_{ij},m_{ijk},\hat m_{ijk},t_{ijk}+c_{ijk})\text{.} 
\end{multline*}
Let $(\phi_i,p_i,\hat p_i,z_{ij},\hat z_{ij},\tilde e_{ij})$ be a tuple expressing this equivalence. 
We start by looking at \cref{cce:metricsprime,cce:metricshatprime}, which here result in $\mathrm{d}p_i=\mathrm{d}\hat p_i=0$; in other words, these functions are constant. We further have $z_{ij} = p_i-p_j$. This means that $[z_{ij}] \in \h^1(X,\Z^{n})$ goes to zero under the map to $\h^1(X,\R^{n})$. But this map is injective, as the relevant part of the long exact sequence is
\begin{equation*}
\hdots \to \R^{n} \to \T^n \to \h^1(X,\Z^{n}) \to \h^1(X,\R^{n}) \to \hdots
\end{equation*} 
and the second arrow is surjective. Thus, there exist $z_i\in \Z^{n}$ such that $z_{ij}=z_i-z_j$. Observe that $p_i-z_i=p_j-z_j$, i.e., there is a      real number $y\in \R$ such that $y=p_i-z_i$. Analogously, we treat $\hat z_{ij}$, getting $\hat y\in \R$ such that $\hat y = \hat p_i - \hat z_i$. 
We consider now
\begin{equation*}
f_{ij}(x):=\tilde e_{ij}(x) + \hat y a_{ij} +\hat a_{ij}z_i\text{;}
\end{equation*}
then, one can show using \cref{eq:TDcoceq} that $f_{ij}$ trivializes $c_{ijk}$, i.e., 
\begin{equation*}
f_{ik}-f_{ij}-f_{jk}=c_{ijk}\text{.}
\end{equation*}  
Next we define $H_i \in \Omega^1(U_i)$ by $H_i := -\phi_i+z_i \hat A_i+\hat y A_i$.
Then we compute, using \cref{eq:equivalencedifferentialTD} and the fact that $\hat p_j$ is constant
\begin{equation*}
A_{ij}=H_i - H_j - f_{ij}^{*}\theta\text{.}
\end{equation*}
Finally, we get from \cref{eq:equivalencedifferentialTD2} that
\begin{equation*}
B_i=\mathrm{d}H_i+y\hat F-\hat y F\text{.} 
\end{equation*}
Summarizing, the last three equations show that there exist $y,\hat y\in \R$ such that $(B_i,A_{ij},c_{ijk}) \sim (y\hat F-\hat y F,0,0)$, i.e., $(B_i,A_{ij},c_{ijk})\in \mathcal{F}_{F,\hat F}$.

It remains to prove the transitivity statement. For this, we suppose that we have two differential cocycles
\begin{align*}
(A_i,\hat A_i,R_i,\phi_{ij},a_{ij},\hat a_{ij},m_{ijk},\hat m_{ijk},t_{ijk})
\\
(A'_i,\hat A'_i,R'_i,\phi'_{ij},a'_{ij},\hat a'_{ij},m'_{ijk},\hat m'_{ijk},t'_{ijk})
\end{align*}
and given equivalences $(p_i,z_{ij})$ and $(\hat p_i,\hat z_{ij})$ between the cocycles of the projected principal $\T^{n}$-bundles with connection, $(A_i,a_{ij},m_{ijk})$ and $(A'_i,a'_{ij},m_{ijk}')$, and $(\hat A_i,\hat a_{ij},\hat m_{ijk})$ and $(\hat A'_i,\hat a'_{ij},\hat m_{ijk}')$, respectively. We have to find a Deligne 2-cocycle $(B_i,A_{ij},c_{ijk})$ such that
\begin{multline}
\label{eq:actionproofeq}
( A_i,\hat A_i,R_i+B_i,\varphi_{ij}+A_{ij} ,a_{ij},\hat a_{ij},m_{ijk},\hat m_{ijk},t_{ijk}+c_{ijk})\\
\sim(A'_i,\hat A'_i,R'_i,\varphi'_{ij},a'_{ij},\hat a'_{ij},m'_{ijk},\hat m'_{ijk},t'_{ijk})\text{.}
\end{multline}
This is achieved by the definitions
\begin{align*}
B_i &:= R_i'-R_i-p_i\hat F
\\
A_{ij} &:=-\varphi_{ij}+\varphi'_{ij}-z_{ij}\hat A_j
\\
c_{ijk} &:= t'_{ijk}-t_{ijk}+\hat p_km_{ijk}'-z_{ij}\hat a_{jk}
\end{align*}
It is indeed straightforward to check using \cref{cce:bundles,eq:diffcc:4,eq:TDcoc,cce:bundleshat} that $(B_i,A_{ij},c_{ijk})$ is a Deligne 2-cocycle. 
In order to establish the equivalence \cref{eq:actionproofeq}, we set $\phi_i := 0$. \cref{eq:equivalencedifferentialTD2} is then obviously satisfied.
The next part of the equivalence is \cref{eq:equivalencedifferentialTD}, which here reads
\begin{align*}
 z_{ij}\hat A_j- z_{ij}\hat A_j'&=-\tilde e_{ij}^{*}\theta\text{.}
\end{align*}
This is satisfied by putting $\tilde e_{ij} := -z_{ij}\hat p_j$. The last equivalence conditions is now \cref{eq:TDcoceq}, which follows immediately from \cref{cce:bundleshat,cce:bundles}. 
\end{proof}

\subsection{Kahle-Valentino's T-duality pairs}

\label{sec:kahle}

In this section we discuss the relation between \emph{differential T-duality correspondences} as introduced in \cref{def:dtc} and \emph{differential T-duality} pairs considered by Kahle-Valentino \cite{Kahle}.

The setting of Kahle-Valentino \cite{Kahle} is  different as it does not explicitly involve string backgrounds. Their discussion is also limited to the case of torus dimension $n=1$. At the basis of their formalism is a groupoid version of differential cohomology, of which below we recall a slightly simplified version.  We consider  \emph{differential cohomology groupoids} $\mathcal{H}^p(X)$, so that the set of isomorphism classes of objects of $\mathcal{H}^{n}(X)$ is the ordinary differential cohomology group $\hat\h^n(X)$. Differential cohomology groupoids are supposed to be equipped with \emph{cup product functors}
\begin{equation*}
\cup: \mathcal{H}^p(X) \times \mathcal{H}^q(X) \to \mathcal{H}^{p+q}(X)\text{.}
\end{equation*} 
Moreover, 
they come equipped with a functor 
\begin{equation*}
\mathcal{I}: \Omega^{p-1}(X)_{dis} \to \mathcal{H}^{p}(X)\text{,}
\end{equation*}
where the left hand side denotes groupoid whose objects are  all $(p-1)$-forms on $X$, and which has only identity morphisms. 
A \emph{geometric trivialization} of an object $\xi\in \mathcal{H}^p(X)$ is a differential form $K \in \Omega^{p-1}(X)$ and an isomorphism $\tau: \xi \to \mathcal{I}_{K}$ in $\mathcal{H}^{p}(X)$.
The set $\hat\h^{p-1}(X)$ acts on the set of all geometric trivializations of $\xi$, where $[\eta]\in \hat\h^2(X)$ sends $\tau$ to $\tau+\eta$, and $K$ gets shifted by the \quot{curvature} of $\eta$. This action is free and transitive.

A concrete realization of these groupoids can be obtained using Deligne cocycles w.r.t. a fixed open cover with all finite non-empty intersections contractible, see \cite[\S A.2]{Kahle}. The objects of $\mathcal{H}^p(X)$ are Deligne $(p-1)$-cocycles $\xi$, and the morphisms $\xi_1 \to \xi_2$ are equivalence classes $[\eta]$ of $(p-2)$-cochains $\eta$ satisfying $\xi_2=\xi_1+\mathrm{D}\eta$, where $\mathrm{D}$ denotes the Deligne differential, and $\eta_1\sim\eta_2$ if there exists a $(p-3)$-cochain $\beta$ with $\eta_2=\eta_1+\mathrm{D}\beta$. Composition of morphisms is just addition. The cup product on the level of objects is the usual cup product in Deligne cohomology, as recalled below. The functor $\mathcal{I}$ is the usual inclusion $\varphi \mapsto (\varphi,0,.,,,0)$ of a globally defined differential form as a \quot{topologically trivial} Deligne cocycle. For $p=2$, the groupoid $\mathcal{H}^2(X)$ is equivalent to the groupoid of principal $\T$-bundles with connections, and connection-preserving bundle isomorphisms. Under this equivalence, a geometric trivialization is  a (not necessarily flat) section. The free and transitive action by $\hat\h^1(X)=C^{\infty}(X,\T)$ is the action of smooth $\T$-valued functions on sections.

\begin{definition}
\label{def:diffTdualitypair}
A \emph{differential T-duality pair} consists of two objects $\xi,\hat \xi \in \mathcal{H}^2(X)$ and a geometric trivialization $\tau: \xi \cup \hat \xi \to \mathcal{I}_K$.
\end{definition}

Kahle-Valentino claim in \cite[\S 2.5]{Kahle} that differential T-duality pairs induce topological T-duality correspondences. We want to sharpen this relation and show that differential T-duality pairs are the same as our differential T-duality correspondences. Their relation to topological T-duality correspondences is then a consequence thereof. In order to proceed, it is necessary to consider an equivalence relation on the set of all differential T-duality pairs over $X$. Unfortunately, Kahle-Valentino do not introduce such relation. Apparently, the most natural definition is the following. 

\begin{definition}
Two differential T-duality pairs $(\xi,\hat \xi,K,\tau)$ and $(\xi',\hat \xi',K',\tau')$ over $X$ are \emph{equivalent} if $K'=K$ and there exist  isomorphisms $ p:\xi \to \xi'$ and $\hat p:\hat \xi \to \hat \xi'$ in $\mathcal{H}^2(X)$ such that the diagram
\begin{equation*}
\alxydim{}{\xi \cup \hat \xi \ar[rr]^{ p \cup\hat p} \ar[dr]_{\tau} && \xi' \cup \hat \xi' \ar[dl]^{\tau'} \\ &\mathcal{I}_K&}
\end{equation*}
in $\mathcal{H}^4(X)$ is commutative.
The set of equivalence classes of differential T-duality pairs is denoted by $\DTP(X)$. 
\end{definition}

Note that the projection to the objects $\xi,\hat\xi$ gives a well-defined map
\begin{equation*}
\DTP(X) \to \hat\h^2(X) \times \hat \h^2(X)\text{.}
\end{equation*}
Below, we will prove the following result.

\begin{proposition}
\label{prop:dttanddtc}
There is a canonical bijection between equivalence classes of differential T-duality correspondences and equivalence classes of differential T-duality pairs, 
\begin{equation*}
\DTC(X) \cong \DTP(X)\text{,}
\end{equation*}
such that the diagram
\begin{equation*}
\alxydim{@C=0em@R=3em}{\DTC(X) \ar[rr] \ar[dr] && \DTP(X) \ar[dl] \\ & \hat\h^2(X) \times \hat\h^2(X)}
\end{equation*}
is commutative.
\end{proposition}

Because of the cup product, it is necessary to work with \emph{extended} Deligne cohomology, i.e., in degree $p$ with the sheaf complex
\begin{equation}
\label{Deligne-complex}
\Z \to \underline{\R} \to \underline{\Omega}^1 \to \underline{\Omega}^2 \to ... \to  \underline{\Omega}^{p}\text{,}
\end{equation}
whereas before we worked with the quasi-isomorphic complex  $\sheaf{\T} \to \sheaf{\Omega}^1 \to ... \to \sheaf{\Omega}^{p}$. In order to be more precise, let us denote the complex \cref{Deligne-complex} by $\mathcal{D}^q(p)$, so that, for instance, $\mathcal{D}^{-1}(p)=\Z$ and $\mathcal{D}^{0}(p)=\sheaf \R$.  The Deligne coboundary operator on the corresponding \v Cech double complex $C^r(\mathcal{D}^{q}(p))$ is defined to be $\mathrm{D}^{r,q}:= (-1)^{q+1}\delta^{r} + \mathrm{d}^{q}$, where $\mathrm{d}^{-1}$ is the inclusion $\Z \incl \sheaf\R$.  
The cup product of extended Deligne cocycles 
\begin{equation*}
\xi=(A^{p-1}_i,A^{p-2}_{i_1i_2},...,A^{0}_{i_1,...,i_p},m_{i_1,...,i_{p+1}})
\quand
\hat\xi=(\hat A^{q-1}_i,\hat A^{q-2}_{i_1i_2},...,\hat A^{0}_{i_1,...,i_q},\hat m_{i_1,...,i_{q+1}})
\end{equation*}
is defined in the usual way \cite[\S 1.5]{brylinski1}\cite[Sec. 2.2]{Gomi2006} by
\begin{multline}
\label{eq:cup}
\xi \cup \hat\xi :=(A^{p-1}_{i_1}\wedge \mathrm{d}\hat A_{i_1}^{q-1}\;,\;...\;,\;A^0_{i_1,...,i_p}\wedge \mathrm{d}\hat A^{q-1}_{i_{p+1}}\;,
\\m_{i_1,...,i_{p+1}}\hat A^{q-1}_{i_{p+1}}\;,\;...\;,\;m_{i_1,...,i_{p+1}}\hat A^0_{i_{p+1},...,i_{p+q}}\;,\;m_{i_1,...,i_{p+1}}\hat m_{i_{p+1},...,i_{p+q+1}})\text{.}
\end{multline}

Of most importance for us is the cup product of two objects $\xi,\hat\xi\in \mathcal{H}^2(X)$. Namely, for $\xi=(A_{i},a_{ij},m_{ijk})$ and $\hat\xi=(\hat A_i,\hat a_{ij},\hat m_{ijk})$ we obtain
\begin{equation}
\label{eq:xihatxi}
\xi\cup \hat\xi=(A_{i}\wedge \hat F\;,\;a_{ij}\hat F\;,\;m_{ijk}\hat A_{k}\;,\;m_{ijk}\hat a_{kl}\;,\;m_{ijk}\hat m_{klp})\text{.}
\end{equation}
Unfortunately, I have not been able to find a description for the cup product of morphisms in $\mathcal{H}^p(X)$. Kahle-Valentino just claim in \cite[\S A.2]{Kahle} that the cup product \quot{extends} to morphism, but do not explain how, whereas  the obvious attempt, namely to apply formula \cref{eq:cup} to cochains, does not work. Concretely, we need the cup product of two morphisms  $[\eta]:\xi \to \xi'$ and that $[\hat \eta]:\hat \xi \to \hat \xi'$ in $\mathcal{H}^2(X)$, i.e., $\xi'=\xi+\mathrm{D}\eta$ and $\hat \xi'=\hat \xi + \mathrm{D}\hat\eta$. Suppose $\eta=(p_i,z_{ij})$ and $\hat\eta=(\hat p_i,\hat z_{ij})$.
The only way that I was able to produce a 3-cochain $\eta\cup\hat \eta$ such that $\xi'\cup\hat\xi' =\xi \cup\hat\xi +\mathrm{D}(\eta\cup\hat \eta)$ is
\begin{equation}
\label{eq:assumptioncup}
\eta\cup \hat\eta = (p_i\hat F\;,\;z_{ij}\hat A_j\;,\; z_{ij}\hat a_{jk}-m'_{ijk}\hat p_k\;,\;m'_{ijk}\hat z_{kl}+z_{ij}\hat m_{jkl})\text{.}
\end{equation}
In the following, I assume that this is the correct cup product of morphisms in $\mathcal{H}^2(X)$.

Exploring the notion of a differential T-duality pair further, we spell out in the following what a geometric trivialization of $\xi \cup \hat\xi$ from \cref{eq:xihatxi} is. It consists of:
\begin{enumerate}[(a)]

\item 
a 3-form $K\in \Omega^3(X)$

\item
2-forms $R_i \in \Omega^2(U_i)$, such that 
\begin{equation}
\label{eq:geotriv:1}
A_i \wedge \hat F = K + \mathrm{d}R_i\text{.}
\end{equation}

\item
1-forms $\varphi_{ij}\in \Omega^1(U_i \cap U_j)$ such that
\begin{equation}
\label{eq:geotriv:2}
a_{ij}\hat F=-R_j+R_i + \mathrm{d}\varphi_{ij}.
\end{equation}

\item
functions $b_{ijk}:U_i \cap U_j \cap U_k \to \R$ such that
\begin{equation}
\label{eq:geotriv:3}
m_{ijk}\hat A_k = \varphi_{ij}+\varphi_{jk}-\varphi_{ik}+\mathrm{d}b_{ijk}.
\end{equation}

\item
numbers $q_{ijkl}\in \Z$ satisfying
\begin{equation}
\label{eq:geotriv:4}
m_{ijk}\hat a_{kl} = q_{ijkl} + b_{ijk}+b_{ikl}-b_{ijl}-b_{jkl}
\end{equation}
and
\begin{equation}
\label{eq:geotriv:5}
m_{ijk}\hat m_{klp}=q_{ijkl}-q_{ijkp}+q_{ijlp}-q_{iklp}+q_{jklp}\text{.}
\end{equation}

\end{enumerate}

At this point, it makes sense to discuss the action of $\hat\h^3(X)$ on differential T-duality pairs, which is induced by the above-mentioned action of $\hat\h^3(X)$ on all geometric trivializations of $\xi \cup \hat\xi$. 
Here, this action takes the form
\begin{equation*}
\hat\h^3(X) \times \DTP(X) \to \DTP(X)
\end{equation*}
and is given, using the above description of geometric trivializations, by the formula
\begin{equation*}
((B_i,A_{ij},c_{ijk},s_{ijkl}), (K,R_i,\varphi_{ij},b_{ijk},q_{ijkl})) \mapsto (K+\mathrm{d}B_i,R_i+B_i,\varphi_{ij}+A_{ij},b_{ijk}+c_{ijk},q_{ijkl}+s_{ijkl})\text{.}
\end{equation*}
Note that $K$ is shifted by the globally defined 3-form $H=\mathrm{d}B_i$, the curvature. 
It is clear that this action restricts to the fibres of the map $\DTP(X) \to \hat\h^2(X) \times \hat\h^2(X)$ and is transitive in each fibre. 

As in \cref{sec:diffTdualitycocycles}, see \cref{prop:actionfreeandtransitive}, we consider the pair $(F,\hat F) \in \Omega^2(X) \times \Omega^2(X)$ determined by an element of $\hat\h^2(X) \times \hat\h^2(X)$, and the subgroup $\mathcal{F}_{F,\hat F} \subset \hat\h^3(X)$.

\begin{lemma}
\label{lem:actionondtp}
The subgroup $\mathcal{F}_{F,\hat F}$ acts trivially, and the quotient $\hat\h^3(X)/\mathcal{F}_{F,\hat F}$ acts freely and transitively in the fibre over $(\xi,\hat\xi)$.
\end{lemma}

\begin{proof}
We have to show that $(K,R_i,\varphi_{ij},b_{ijk},q_{ijkl})$ and $(K,R_i+\hat y F + y\hat F,\varphi_{ij},b_{ijk},q_{ijkl})$ define the same morphism. We consider the automorphism of $\xi=(A_{i},a_{ij},m_{ijk})$ given by $(y,0)$, this works as $\mathrm{D}(y,0)=(0,0,0)$, and similarly, the automorphism of $\hat\xi$ given by $(\hat y,0)$. According to \cref{eq:assumptioncup}, we have
\begin{equation*}
(y,0)\cup (\hat y,0)=(y\hat F\;,\;0\;,\; -m_{ijk}\hat y,\;0)
\end{equation*}
We can change this by the coboundary of
$(-\hat y A_i,-\hat ya_{ij},0)$, 
which is $(-\hat y F,0,\hat y m_{ijk},0)$.
Thus,
\begin{equation*}
[(y,0)\cup (\hat y,0)] = [(y \hat F-\hat y F,0,0,0)]\text{.}
\end{equation*}
This proves that $\mathcal{F}$ acts trivially. Conversely, if 
\begin{equation*}
(K,R_i,\varphi_{ij},b_{ijk},q_{ijkl})
\quand
(K+\mathrm{d}B_i,R_i+B_i,\varphi_{ij}+A_{ij},b_{ijk}+c_{ijk},q_{ijkl}+s_{ijkl})
\end{equation*}
are equivalent, we have to show that $(B_i,A_{ij},c_{ijk},s_{ijkl})\sim (\hat y F-y\hat F,0,0,0)$. The proof of this is very similar to the one given in \cref{prop:actionfreeandtransitive}, and omitted for brevity. 
\end{proof}

Finally, we are in a position to give the proof of \cref{prop:dttanddtc}.
Under \cref{th:localtoglobalequivalence,prop:LDdifftohatHbij}, it remains to provide  a bijection
\begin{equation}
\label{eq:adctodtp}
\hat\h^1(X,\mathbb{TD}_{\kappa}) \to \DTP(X). 
\end{equation}
Let $( A_i,\hat A_i,R_i,\varphi_{ij} ,a_{ij},\hat a_{ij},m_{ijk},\hat m_{ijk},t_{ijk})$ be an adjusted differential $\mathbb{TD}$-cocycle. We set $\xi=(A_i,a_{ij},-m_{ijk})$ and $\hat \xi=(\hat A_i,\hat a_{ij},-\hat m_{ijk})$: this ensures that the diagram in \cref{prop:dttanddtc} will commute, while the signs  accounts for the different conventions used in non-abelian cohomology and Deligne cohomology. We define
\begin{equation}
\label{eq:defK}
K:= \mathrm{d}R_i +  A_i \wedge \hat F\text{;}
\end{equation}
this is the  3-curvature of  \cref{eq:adjusted3curvature}, and hence a globally defined 3-form. Thus,  passing to $-R_i$, we have \cref{eq:geotriv:1}. We may then use the given 1-form $\varphi_{ij}$, and note  that \cref{eq:diffcc:3} results into \cref{eq:geotriv:2}. 
Next, we choose real-valued functions $b_{ijk}$ that represent the given $\T$-valued functions $-t_{ijk}$; then, \cref{eq:diffcc:4} results into \cref{eq:geotriv:3}. 
Finally, we consider \cref{eq:TDcoc},
\begin{equation*}
t_{ikl}+t_{ijk}-m_{ijk}\hat a_{kl} = t_{ijl}+t_{jkl}\text{,}
\end{equation*}
which is an equation of $\T$-valued functions.
Substituting the lifts $b_{ijk}$ reveals $q_{ijkl}\in \Z$ such that
\begin{equation*}
-b_{ikl}-b_{ijk}-m_{ijk}\hat a_{kl} = -b_{ijl}-b_{jkl} + q_{ijkl}\text{,}
\end{equation*}
this is \cref{eq:geotriv:4}.
Finally, \cref{eq:geotriv:5} is a straightforward calculation. Summarizing,
$(K,R_i,\varphi_{ij},b_{ijk},q_{ijkl})$ is a geometric trivialization of $\xi \cup \hat\xi$.

Next we consider an equivalence between adjusted differential cocycles
\begin{equation*}
( A_i,\hat A_i,R_i,\varphi_{ij} ,a_{ij},\hat a_{ij},m_{ijk},\hat m_{ijk},t_{ijk})
\quand 
( A_i',\hat A_i',R_i',\varphi_{ij}' ,a_{ij}',\hat a_{ij}',m_{ijk}',\hat m_{ijk}',t_{ijk}')
\end{equation*} 
established by a tuple $(\phi_i,p_i,\hat p_i,z_{ij},\hat z_{ij},\tilde e_{ij})$.
Then, $\eta:=(-p_i,z_{ij})$ and $\hat\eta:=(-\hat p_i,\hat z_{ij})$ are morphisms in $\mathcal{H}^2(X)$ between $\xi$ and $\xi':=( A_i',a'_{ij},-m'_{ijk})$, and $\hat\xi$ and $\hat\xi'=(\hat A'_i,\hat a_{ij}',-\hat m'_{ijk})$, respectively. We have to show that
\begin{equation*}
(K,R_i,\varphi_{ij},b_{ijk},q_{ijkl})\sim (K',R_i',\varphi_{ij}',b_{ijk}',q_{ijkl}')+\eta \cup\hat \eta\text{.} 
\end{equation*}
We claim that both cochains differ in fact by the coboundary of $(\phi_i,z_{ij}\hat p_j+f_{ij},-r_{ijk}+z_{ij}\hat z_{jk})$, where $f_{ij}$ is a real-valued lift of $\tilde e_{ij}$, and $r_{ijk}\in \Z$  are the numbers that emerge from the $\T$-valued cocycle condition \cref{eq:TDcoceq} under this lift.  The claim is straightforward to check using \cref{eq:assumptioncup}.

By now we have constructed a well-defined map \cref{eq:adctodtp}, such that it preserves the fibres of the projections to $\hat\h^2(X) \times \hat\h^2(X)$. It is easy to see that our map \cref{eq:adctodtp} is equivariant w.r.t. to the actions of $\hat\h^3(X)/\mathcal{F}_{F,\hat F}$ in each fibre. Since these actions are free and transitive on both sides (\cref{lem:actionondtp,prop:actionfreeandtransitive}), it follows that \cref{eq:adctodtp} is a bijection. This proves \cref{prop:dttanddtc}.

\section{Examples of geometric T-duality }

\label{sec:examples}

\setsecnumdepth{2}

We first consider in \cref{sec:trivialBfield} the situation of a general principal $\T^{n}$-bundle $E$, a general metric, and trivial B-field, and present a construction of a T-dual geometric T-background. In \cref{sec:hopftriv} we specialize to the case that $E$ is the Hopf fibration, in which we explicitly compute the dual metric and dual bundle gerbe. In \cref{sec:hopf} we keep the Hopf fibration but consider a non-trivial B-field, whose Dixmier-Douady class is a generator of $\h^3(S^3,\Z)$. We prove that this geometric T-background is self-dual.      

\subsection{A torus bundle with trivial B-field}

\label{sec:trivialBfield}

We consider a geometric T-background $(E,g,\mathcal{G})$ over a smooth manifold $X$, whose bundle gerbe is the trivial one, i.e., $\mathcal{G}=\mathcal{I}_0$. In this section, we  explicitly construct a geometric T-duality correspondence whose left leg is $(E,g,\mathcal{I}_0)$. 

We let $(\omega,g',h)$ be the triple corresponding to $g$ under \cref{th:kaluzaklein}, and we let $F \in \Omega^2(X)$ be the curvature of the connection $\omega$. We consider the trivial bundle $\hat E := X \times \T^{n}$, and equip it with the trivial connection, $\hat\omega := \theta$. We let $\hat g$ be the invariant metric on $\hat E$ that corresponds to the triple $(\hat\omega,g',h^{-1})$. Next we construct the bundle gerbe $\hat{\mathcal{G}}$ over $\hat E$.

The surjective submersion is $Y:=E \times \T^{n} \to X \times \T^{n}$. The curving is 
\begin{equation*}
\Psi := \pr_E^{*}\omega \wedge \pr_{\T^{n}}^{*}\theta\in \Omega^2(Y)\text{.}
\end{equation*}
The 2-fold fibre product is $Y^{[2]} = E^{[2]} \times \T^{n}$. Note that we have a smooth map $g:E^{[2]} \to \T^{n}$, $e_2=e_1g(e_1,e_2)$, and  $\pr_2^{*}\omega = \pr_1^{*}\omega + g^{*}\theta$. Thus, we see that
\begin{equation*}
\pr_2^{*}\Psi - \pr_1^{*}\Psi=g^{*}\theta \dot\wedge \pr_{\T^{n}}^{*}\theta
\end{equation*} 
on $Y^{[2]}$.
Comparing with \cref{re:curvpoi}, the right hand side is the curvature of the pullback of the Poincaré bundle $\poi$ along the map $\tilde g: E^{[2]} \times \T^{n}  \to \T^{2n}: (e_1,e_2,a)\mapsto (a,g(e_1,e_2))$. 
 Thus, we readily define 
\begin{equation*}
P:=\tilde g^{*}\poi
\end{equation*}
as the principal $\T$-bundle with connection of $\hat{\mathcal{G}}$. Over $Y^{[3]}=E^{[3]} \times \T^{n}$, we have an isomorphism
\begin{multline*}
(\pr_{23}^{*}P\otimes \pr_{12}^{*}P)_{(e_1,e_2,e_3,a)} = \poi_{a,g(e_2,e_3)} \otimes \poi_{a,g(e_1,e_2)} \\\stackrel{\varphi_r}{\longrightarrow} \poi_{a,g(e_2,e_3)g(e_1,e_2)}=\poi_{a,g(e_1,e_3)}=(\pr_{13}^{*}P)_{(e_1,e_2,e_3,a)}\text{,}
\end{multline*}
where $\varphi_r$ was defined in \cref{sec:poincare}.
This isomorphism satisfies the associativity condition over $Y^{[4]}$ due to the commutativity of the analog of  \cref{eq:varphiass} for $\varphi_r$.

\begin{remark}
\label{re:dualcupproduct}
If $n=1$, then $\hat{\mathcal{G}}$ is precisely the \emph{cup product bundle gerbe} $\pr_X^{*}E \cup \pr_{\T}$, where $\pr_X: \hat E \to X$ and $\pr_{\T}: \hat E \to \T^1$ are the projections; explicitly, $\pr_X^{*}E$ is a principal $\T$-bundle over $\hat E$ with connection, and $\pr_{\T}$ is a $\T$-valued function on $\hat E$.  A description of the cup product of such structures, resulting in a bundle gerbe with connection, has been given by Johnson in \cite{johnson1}. Our construction above (for $n=1$) reproduced exactly that description. Johnson also proved that the cup product of a principal $\T$-bundle with connection and a $\T$-valued function coincides with the cup product in Deligne cohomology \cite{johnson1}.    
\end{remark}

We will now construct a geometric T-duality correspondence between  the  geometric T-backgrounds $(E,g,\mathcal{G})$ and $(\hat E,\hat g,\hat{\mathcal{G}})$. On correspondence space $E \times_X \hat E$ we need to find a connection-preserving isomorphism $\mathcal{D}:\pr^{*}\mathcal{G}\to \hat\pr^{*}\hat{\mathcal{G}} \otimes \mathcal{I}_{\rho_{g,\hat g}}$, where \begin{equation*}
\rho_{g,\hat g} = \hat\pr^{*}\hat \omega \dot\wedge \pr^{*}\omega = -\Psi\text{.}
\end{equation*} 
We note that $\hat\pr^{*}\hat{\mathcal{G}}$ is trivializable since its surjective submersion has a section $\sigma$ along $\hat\pr$, namely, the identity, $\sigma=\id_{E \times \T^{n}}$:
\begin{equation*}
\alxydim{}{ & E \times \T^{n}=Y\mquad \ar[d]  \\ E \times_X \hat E = E \times \T^{n}  \ar@{=}[ur] \ar[r] & X \times \T^{n} = \hat E\text{.}\mquad}
\end{equation*}
It induces a trivialization $\mathcal{S}: \hat\pr^{*}\hat{\mathcal{G}} \to \mathcal{I}_{\Psi}$, and $\mathcal{D}$ may be defined as
\begin{equation*}
\alxydim{@C=4em}{\pr^{*}\mathcal{G} = \mathcal{I}_0 =\mathcal{I}_{\Psi} \otimes \mathcal{I}_{-\Psi} \ar[r]^-{\mathcal{S}^{-1} \otimes \id} & \hat\pr^{*}\hat{\mathcal{G}} \otimes \mathcal{I}_{\rho_{g,\hat g}}\text{.}}
\end{equation*}
Thus, $\mathcal{D}$ is a geometric correspondence. It remains to check that it is a geometric T-duality correspondence. 

Conditions \cref{def:gtdc:1*,def:gtdc:2*} of \cref{def:gtdc} hold by construction of the metric $\hat g$. In order to check condition \cref{def:gtdc:3*}, we consider an open subset $U \subset X$ that admits a trivialization $\varphi: U \times \T^{n} \to E|_U$. On the dual side, we choose the identity trivialization, $\hat\varphi=\id$. We put $B:= 0$ and $\mathcal{T}:= \id$, as a trivialization of $\mathcal{I}_0=\varphi^{*}\mathcal{G} \to \mathcal{I}_B$. Note that $\hat\varphi^{*}\hat{\mathcal{G}}=\hat{\mathcal{G}}|_{U \times \T^{n}}$. Thus, the surjective submersion of $\hat\varphi^{*}\hat{\mathcal{G}}$ has a global section, $\tau:=(\varphi,\pr_{\T^{n}}):U \times \T^{n} \to E \times \T^{n}$. It induces a trivialization $\hat{\mathcal{T}}: \hat\varphi^{*}\hat{\mathcal{G}} \to \mathcal{I}_{\tau^{*}\Psi}$. We put 
\begin{equation*}
\hat B := \tau^{*}\Psi = \varphi^{*}\omega \dot\wedge \pr_{\T^{n}}^{*}\theta\text{.}
\end{equation*}  
Now we work over $U \times \T^{2n}$, where we find the diagram
\begin{equation*}
\alxydim{}{&&&E \times \T^{n} \ar[dd]\\U \times \T^{2n} \ar[dr]_{\hat \pr} \ar[rr]^-{\Phi} && E \times_X \hat E  \ar@/^1pc/[ur]_{\sigma}   \ar[dr]_{\hat\pr} \\ & U \times \T^{n} \ar@/^2pc/[uurr]^{\tau} \ar[rr]_{\hat\varphi} && \hat E }
\end{equation*}
whose rectangular part is commutative, but the sections differ. This means that the induced trivializations $\Phi^{*}\mathcal{S}$ and $\hat\pr^{*}\hat{\mathcal{T}}$ differ by the $\T$-bundle with connection
\begin{equation*}
(\sigma \circ \Phi, \tau \circ \hat\pr)^{*}P\text{;}
\end{equation*}
a discussion of this fact can be found in \cite[Lem. 3.2.3]{waldorf10}.
We readily compute the map
\begin{equation*}
k:=\tilde g \circ (\sigma \circ \Phi, \tau \circ \hat\pr): U \times \T^{2n} \to \T^{2n}:(x,a,\hat a) \mapsto (\hat a,\hat a-a)\text{.}
\end{equation*}
We note that $k^{*}\poi \cong \poi_{3,3-2}\cong \poi_{3,3} \otimes \poi_{3,-2}\cong \poi_{2,3}$, using the results of \cref{sec:poincare}.
The 2-isomorphism $\hat\pr^{*}\hat{\mathcal{T}} \cong \Phi^{*}\mathcal{S} \otimes k^{*}\poi$ implies that the relevant isomorphism of  \cref{def:gtdc:3*} \cref{def:gtdc:3c*},
\begin{align*}
&\alxydim{@C=3.5em}{\mathcal{I}_{\pr^{*}B} \ar[r]^-{\pr^{*}\mathcal{T}^{-1}} &  \pr^{*}\varphi^{*}\mathcal{G} \ar[r]^-{\Phi^{*}\mathcal{D}} & \hat \pr^{*}\hat\varphi^{*}\hat{\mathcal{G}} \otimes \mathcal{I}_{\Phi^{*}\rho} \ar[r]^-{\hat \pr^{*}\hat{\mathcal{T}} \otimes \id} &  \mathcal{I}_{\hat \pr^{*}\hat B + \Phi^{*}\rho}}
\end{align*}
corresponds to the principal $\T$-bundle $k^{*}\poi\cong \pr_{\T^{2n}}^{*}\poi$. This completes the proof the we have a geometric T-duality correspondence. 
In particular, by \cref{lem:Bs}, the Buscher rules hold locally.

\setsecnumdepth{2}
\subsection{The Hopf fibration with a trivial B-field}

\label{sec:hopftriv}

In this section we apply the  construction of the previous \cref{sec:trivialBfield} to the example where the torus bundle $E$ is the Hopf fibration $E:=S^3 \to S^2$. This reproduces a result from the PhD thesis of Kunath \cite[\S 3.4, \S 4.4]{Kunath2021}, where that  case has been discussed separately. 

We denote the round metric on the $n$-sphere by $g_n$; the metric on $E$ is $g=g_{3}$, which is indeed $\T$-invariant. Then, the dual torus bundle is $\hat E :=S^2 \times \T$. This was probably the first observation of a topology change, and made in \cite{Alvarez1994}. There, the following result has been proved, by applying locally the Buscher rules. Here, we re-derive it by applying the general procedure of \cref{sec:trivialBfield}. 

\begin{lemma}
The dual metric is $\hat g=\frac{1}{4}g_{2} \oplus g_{1}$. 
\end{lemma}

\begin{proof}
We claim that $g'=\frac{1}{4}g_2$ and $h=g_1$.  Then, \cref{rem:kaluzakleintrivialbundle}
applied to $\hat E$ and the trivial connection $A=0$ yields the lemma. The claim can be proved in an explicit model for the Hopf fibration. We model $p:S^3 \to S^2$ as the restriction to unit vectors of the map
\begin{equation*}
\R^4 \to \R^3: (x_0,x_1,x_2,x_3) \mapsto (2(x_0x_2+x_1x_3),2(x_1x_2-x_0x_3),x_0^2+x_1^2-x_2^2-x_3^2)\text{.}
\end{equation*} 
The action of $z\in \T$ sends $(x_0,x_1,x_2,x_3)$ to $(x_0',x_1',x_2',x_3')$, where
\begin{equation*}
x_0' + \im x_1' := z\cdot (x_0+\im x_1)
\quand
x_2' + \im x_3' := z\cdot (x_2+\im x_3)\text{.}
\end{equation*}
The tangent space $T_xS^3$ at $x\in S^3\subset \R^4$ is $x^{\perp}\subset \R^4$. The round metric $g_3$ is given by the standard inner product on $\R^4$, i.e., $g_3(v,w):=v\cdot w$. (One can see now directly that it is $\T$-invariant.) The differential of the bundle projection $p$ at $x=(x_0,x_1,x_2,x_3)$ is
\begin{equation*}
T_xp =2 \begin{pmatrix}x_2 & x_3 & x_0 & x_1 \\
-x_3 & x_2 & x_1 & -x_0 \\
x_0 & x_1 & -x_2 & -x_3 \\
\end{pmatrix}\text{.}
\end{equation*}
One computes
\begin{equation*}
V_x=\mathrm{Kern}(T_xp)=\left \langle (-x_1,x_0,-x_3,x_2)  \right \rangle
\end{equation*}
and thus,
\begin{equation*}
h_{p(x)}(r,s):=g_3((-rx_1,rx_0,-rx_3,rx_2),(-sx_1,sx_0,-sx_3,sx_2))=rs\text{.}
\end{equation*}
In particular, this metric does not depend on the base point $p(x)$. 
We observe that $T_xp\cdot T_xp^{\mathrm{tr}}=4E_4$, where $E_4$ denotes the unit matrix. We know that $T_xp|_{H_x}:H_x\to T_{p(x)}X$ is an isomorphism, and so $\frac{1}{4}T_xp^{\mathrm{tr}}$ is a right inverse. Thus,
\begin{equation*}
g(v,w):=\textstyle g_3(\frac{1}{4}T_xp^{\mathrm{tr}}(v),\frac{1}{4}T_xp^{\mathrm{tr}}(w))=\frac{1}{4}vw=\frac{1}{4}g_2(v,w)\text{.}
\end{equation*}
This proves the claim.
\end{proof}

By \cref{re:dualcupproduct}, the dual bundle gerbe $\hat{\mathcal{G}}$ is the cup product $\hat{\mathcal{G}} = \pr_{S^2}^{*}E \cup \pr_{S^1}$ of the principal $\T$-bundle $\pr_{S^2}^{*}E$ and the $\T$-valued function $\pr_{S^1}$. Summarizing, we have the following result.

\begin{proposition}
\label{lem:hopf:gerbe}
Let $E:= S^3 \to S^2$ be the Hopf fibration, $g:= g_3$ be the round metric, and $\mathcal{G}=\mathcal{I}_0$ be the trivial bundle gerbe. Then, there exists a geometric T-duality correspondence between $(E,g,\mathcal{G})$ and the geometric T-background $(S^2 \times S^1,\frac{1}{4}g_2\oplus g_1,\hat{\mathcal{G}})$, where $\hat{\mathcal{G}} = \pr_{S^2}^{*}E \cup \pr_{S^1}$ is the cup product bundle gerbe. In particular, the Dixmier-Douady class of $\hat{\mathcal{G}}$ is a cup product in singular cohomology,  
\begin{equation*}
\mathrm{DD}(\hat{\mathcal{G}}) = \pr_{S^2}^{*}c_1 \cup \pr_{S^1}^{*}\theta\text{,}
\end{equation*}
where $c_1$ is the first Chern class of the Hopf fibration, a generator of $\in \h^2(S^2,\Z)$, and $\theta \in \h^1(S^1,\Z)$ is a generator. Thus, $\mathrm{DD}(\hat{\mathcal{G}})$ is a generator of $\h^3(S^2 \times S^1,\Z) \cong \Z$. Moreover, the H-flux of $\hat{\mathcal{G}}$ is
\begin{equation*}
\hat H = \pr_{S^2}^{*}F \wedge \pr_{S^1}^{*}\theta\text{,} 
\end{equation*}    
where $F\in \Omega^2(S^2)$ is the curvature of the Kaluza-Klein connection corresponding to the metric $g=g_3$.
\end{proposition}

As remarked above, the dual metric has been computed in \cite{Alvarez1994}. 
The formula for the dual H-flux has been proved
in the setting of T-duality with H-flux by Bouwknegt-Evslin-Mathai  \cite{Bouwknegt}. Our unifying setting of geometric T-duality correspondences implies
both results.

\subsection{The Hopf fibration with the basic gerbe}

\label{sec:hopf}

The Hopf fibration $E:=S^3 \to S^2$ carries a canonical non-trivial bundle gerbe with connection, namely, the \emph{basic gerbe} $\mathcal{G}_{bas}$ over $\su 2$, under the canonical diffeomorphism $S^3 \cong \su 2$. In this section, we consider this bundle gerbe, while we keep $E$ equipped with the round metric $g_3$ as in \cref{sec:hopftriv}. 

\begin{proposition}
\label{prop:selfdual}
The geometric T-background $(S^3,g_3,\mathcal{G}_{bas})$ is self-dual under geometric T-duality. 
\end{proposition}

In the setting of T-duality with H-flux, the self-duality of $(S^3,g_3,H)$, where $H \in \Omega^3(S^3)$ is the curvature of the basic gerbe, i.e., the canonical 3-form, was known before; \cref{prop:selfdual} upgrades this to geometric T-duality.

In the remainder of this section we prove \cref{prop:selfdual}. We recall that the diffeomorphism between $\su 2$ and $S^3$ is
\begin{equation*}
\su 2 \to S^3: \begin{pmatrix}a & -\bar b \\
b & \bar a \\
\end{pmatrix} \mapsto (a,b)\text{;}
\end{equation*} 
here, the resulting element $(a,b)\in \C^2$ is identified with  $(\mathrm{Re}(a),\mathrm{Im}(a),\mathrm{Re}(b),\mathrm{Im}(b))\in S^3 \subset \R^4$. It is  well-known that the round metric $g_3$ on $S^3$ corresponds to the Killing form $B(Y_1,Y_2) = 4\mathrm{tr}(Y_1Y_2)$ on the Lie algebra $\mathfrak{su}(2)$. More precisely, we have under above diffeomorphism
\begin{equation*}
g_3=-\frac{1}{8}B\text{.}
\end{equation*}

Under the diffeomorphism with $\su 2$, the principal $\T$-action of the Hopf fibration is a map $\tau: \su 2 \times \T \to \su 2$, and it  is given by matrix multiplication along the group homomorphism 
\begin{equation*}
\zeta: \T \to \su 2: z \mapsto \begin{pmatrix}z & 0 \\
0 & \bar z \\
\end{pmatrix}
\end{equation*}
In other words, we have $\tau = m \circ\ (\id \times \zeta)$, where $m$ denotes the multiplication map of $\su 2$. 
We also remark that $\zeta(\T)$ is a maximal torus. 
Now we consider the basic bundle gerbe $\mathcal{G}_{bas}$ over $\su 2$ \cite{gawedzki1,meinrenken1}. Its canonical connection has the curvature $H=\frac{1}{6}\left \langle \theta\wedge[\theta\wedge\theta]  \right \rangle$, where $\left \langle  -,- \right \rangle$ is the basic inner product, which, in case of $\su 2$, is $\left \langle -,-  \right \rangle =-\mathrm{tr}(-\cdot -)=-\frac{1}{4}B$. We recall that $\mathcal{G}_{bas}$ also has a canonical multiplicative structure \cite{carey4,waldorf5}, consisting of a connection-preserving isomorphism
\begin{equation*}
\mathcal{M}:\pr_1^{*}\mathcal{G} \otimes \pr_2^{*}\mathcal{G} \to m^{*}\mathcal{G} \otimes \mathcal{I}_{\rho}
\end{equation*}
over $\su 2 \times \su 2$, where $\rho = \frac{1}{2}\left \langle \pr_1^{*}\theta\wedge \pr_2^{*}\bar\theta  \right \rangle \in \Omega^2(\su 2 \times \su 2)$; here $\bar\theta$ is the \emph{right}-invariant Maurer-Cartan form. Additionally, there is an \quot{associator}, a connection-preserving 2-isomorphism
\begin{equation*}
\alpha: \mathcal{M}_{1+2,3} \circ (\mathcal{M}_{1,2} \otimes \id) \Rightarrow \mathcal{M}_{1,2+3} \circ (\id \otimes \mathcal{M}_{2,3}) 
\end{equation*}  
over $\su 2^3$, which in turn satisfies a pentagon axiom over $\su 2^4$.

We consider another multiplicative bundle gerbe, but over the Lie group $\T$. The underlying bundle gerbe with connection is the  trivial one, $\mathcal{I}_0$. It is equipped with a multiplicative structure using the method of \cite[Ex. 1.4 (b)]{waldorf5}. Its multiplication isomorphism
\begin{equation*}
\poi:\pr_1^{*}\mathcal{I}_0 \otimes \pr_2^{*}\mathcal{I}_0 \to m^{*}\mathcal{I}_0 \otimes \mathcal{I}_{\Omega}
\end{equation*}
over $\T^2$ is given by the Poincaré bundle $\poi$ over $\T^2$, under the equivalence of \cref{lem:gerbehombundle}. Its associator is 
\begin{equation*}
\alxydim{@C=4em}{\poi_{1+2,3} \otimes \poi_{1,2} \ar[r]^-{\varphi_l^{-1} \otimes \id} & \poi_{1,3} \otimes \poi_{2,3} \otimes \poi_{1,2} \ar[r]^{\mathrm{flip}} & \poi_{1,2} \otimes \poi_{1,3} \otimes \poi_{2,3} \ar[r]^-{\varphi_r \otimes \id} & \poi_{1,2+3} \otimes \poi_{2,3}\text{,} }
\end{equation*}
and one can easily check that the pentagon  condition over $\T^4$ is satisfied. The bundle gerbe $\mathcal{I}_0$ together with the multiplicative structure will be denoted by $\mathcal{I}_0^{\mathcal{P}}$.

\begin{lemma}
\label{lem:basicandpoincare}
We have $\zeta^{*}\mathcal{G}_{bas} \cong \mathcal{I}_0^{\mathcal{P}}$ as multiplicative bundle gerbes with connection.
\end{lemma}

\begin{proof}
One considers for multiplicative bundle gerbes with connection the pair $(H,\rho)$ consisting of the curvature $H$ of the bundle gerbe and the 2-form $\rho$ of their multiplicative structure. One can check that $\zeta^{*}(H,\rho) = (0,\Omega)$. 
By \cite[Prop. 2.4]{waldorf5} the pair $(H,\rho)$ characterizes the multiplicative bundle gerbe uniquely up to isomorphism provided that $\h^4(BG,\Z)$ is torsion-free.  This is the case when $G=\T$, as the cohomology of $\T$  is a polynomial ring.
\end{proof}

In the following we choose an isomorphism $\mathcal{T}:\zeta^{*}\mathcal{G}_{bas} \to \mathcal{I}_0^{\poi}$ of multiplicative bundle gerbes with connection (it is unique up to unique 2-isomorphism).
The multiplicative structure $\mathcal{M}$ of $\mathcal{G}_{bas}$ then induces an isomorphism $\mathcal{M}'$
\begin{multline*}
\tau^{*}\mathcal{G}_{bas} = (\id \times \zeta)^{*}m^{*}\mathcal{G}_{bas} \stackrel{\mathcal{M}}{\cong} (\id \times \zeta)^{*}(\pr_1^{*}\mathcal{G}_{bas} \otimes \pr_2^{*}\mathcal{G}_{bas} \otimes \mathcal{I}_{-\rho})
\\=\pr_1^{*}\mathcal{G}_{bas} \otimes \pr_2^{*}\zeta^{*}\mathcal{G}_{bas} \otimes \mathcal{I}_{-(\id \times \zeta)^{*}\rho}\stackrel{\mathcal{T}}{\cong} \pr_1^{*}\mathcal{G}_{bas}\otimes \mathcal{I}_{-(\id \times \zeta)^{*}\rho}\text{.}
\end{multline*}
over $\su 2 \times \T$. Next we infer that $\su 2 \times \T$ is canonically diffeomorphic to the correspondence space for the self-dual situation: the diffeomorphism is
\begin{equation*}
\Psi: \su 2 \times \T \to S^3 \times_{S^2} S^3: (X,z) \mapsto (X,X\zeta(z))\text{.}
\end{equation*}
Note that $\pr \circ \Psi = \pr_1$ and $\hat\pr \circ \Psi = \tau$. Thus, pulling back the isomorphism $\mathcal{M}'$
along $\Psi^{-1}$, we obtain a candidate for the isomorphism $\mathcal{D}$. 
We first verify that the 2-form is correct, i.e. 
\begin{equation*}
\Psi^{*}\rho_{g_3,g_3} = (\id \times\zeta)^{*}\rho\text{.}
\end{equation*}
This can be checked explicitly using the given definitions.
By this, we have a geometric correspondence.

Conditions \cref{def:gtdc:1*,def:gtdc:2*} of \cref{def:gtdc} are obviously satisfied, since we have the same bundle and metric on both sides. 
It remains to verify condition \cref{def:gtdc:3*}. Consider an open set $U \subset S^2$ with a local trivialization $\varphi: U \times S^1 \to S^3|_U$. We denote by $s:U \to S^3\cong \su 2: x \mapsto \varphi(x,1)$ the corresponding section. For dimensional reasons, there exists a trivialization $\mathcal{S}: s^{*}\mathcal{G}_{bas} \to \mathcal{I}_\lambda$, where $\lambda\in \Omega^2(S^2)$.  Note that $\varphi(x,z)=s(x)\zeta(z)$, or, $\varphi=m \circ (s \times \zeta)$. Hence, we may produce a trivialization
\begin{equation*}
\mathcal{U}: \varphi^{*}\mathcal{G}_{bas} = (s \times\zeta)^{*}m^{*}\mathcal{G}_{bas} \cong \pr_1^{*}s^{*}\mathcal{G}_{bas} \otimes \pr_2^{*}\zeta^{*}\mathcal{G}_{bas} \otimes \mathcal{I}_{-(s \times \zeta)^{*}\rho}\cong \mathcal{I}_{\pr_1^{*}\lambda-(s\times\zeta)^{*}\rho}
\end{equation*}  
with $B:= \pr_1^{*}\lambda-(s\times\zeta)^{*}\rho$.
We choose the same trivializations $\varphi$ and $\mathcal{U}$ on both sides.

We observe that there is a commutative diagram
\begin{equation*}
\alxydim{}{& U \times \T^2 \ar[dr]^{\Phi} \ar[dl]_{\psi} \\ \su 2\times \T \ar[rr]_{\Psi} && S^3 \times_{S^2} S^3\text{,} }
\end{equation*}
where $\psi(x,z_1,z_2) := (s(x)\zeta(z_1),z_2-z_1)$, and $\Phi=(\varphi,\varphi)$.
Over $U \times \T^2$ we then have to consider the isomorphism
\begin{equation*}
(\mathcal{U}_{1,3} \otimes \id)\circ \psi^{*}\Psi^{*}\mathcal{D} \,\circ\, \mathcal{U}^{-1}_{1,2} 
\end{equation*}
Substituting the definitions of $\mathcal{U}$ and $\mathcal{D}$, it turns our that all occurrences of $\mathcal{M}$, and both occurrences of $\mathcal{S}$ cancel. Remaining are the contributions of $\mathcal{T}$, which are $\mathcal{T}_1$, $\mathcal{T}_2^{-1}$ and $\mathcal{T}_{2-1}$. By \cref{lem:basicandpoincare}, this gives the Poincaré bundle.
This proves \cref{def:gtdc:3*}, and completes the proof of \cref{prop:selfdual}.

\bibliographystyle{kobib}
\bibliography{kobib}

\end{document}

%% file: styles.tex
\usepackage{amsthm}
\usepackage{thmtools}
\usepackage{graphicx}
\usepackage[margin=2cm,font=normalsize,labelfont=bf]{caption}
\usepackage{verbatim}
\usepackage[shortlabels]{enumitem}
\usepackage{fancyhdr}
\usepackage[]{geometry}
\usepackage{xstring}
\usepackage{needspace}
\usepackage{color}
\usepackage{amstext}
\usepackage{nameref}
\usepackage[hypertexnames=false]{hyperref}
\usepackage{mathdots}
\usepackage{soul}
\usepackage{chngcntr}
\usepackage[section]{placeins}

\makeatletter
\newcounter{denseversion}
\newcounter{comments}
\newcounter{authorcounter}
\newcounter{adresscounter}

\def\title#1{\gdef\@title{#1}}
\def\@title{}
\def\subtitle#1{\gdef\@subtitle{#1}}
\def\@subtitle{}

\def\authortagsused{0}
\def\adresstag#1{\if!#1!\else$^{\;#1\;}$\fi}
\def\@authorsep#1{
  \ifnum\value{authorcounter}=#1 and \else\unskip, \fi
}

\renewcommand{\author}[2][]{
  \stepcounter{authorcounter}
  \if!#1!\else\gdef\authortagsused{1}\fi
  \ifnum\value{authorcounter}=1
    \def\@authorstringa{#2\adresstag{#1}}
    \def\@authorstringb{#2}
    \def\@authorstringc{#2\adresstag{#1}}
  \else
    \ifnum\value{authorcounter}=2
      \g@addto@macro\@authorstringa{\@authorsep{2}#2\adresstag{#1}}
      \g@addto@macro\@authorstringb{\@authorsep{2}#2}
      \g@addto@macro\@authorstringc{\\#2\adresstag{#1}}
    \else
      \g@addto@macro\@authorstringa{\@authorsep{3}#2\adresstag{#1}}
      \g@addto@macro\@authorstringb{\@authorsep{3}#2}
      \g@addto@macro\@authorstringc{\\#2\adresstag{#1}}
    \fi
  \fi}
\def\@author{\ifnum\value{denseversion}=0\@authorstringa\else\@authorstringb\fi}

\def\@adressstringa{}
\def\@adressstringb{}
\newcommand{\adress}[2][]{
  \stepcounter{adresscounter}
  \ifnum\value{adresscounter}=1
    \g@addto@macro\@adressstringa{\ifnum\authortagsused=0\def\br{\\}\else\def\br{, }\fi\adresstag{#1}#2}
    \g@addto@macro\@adressstringb{\def\br{\\}\adresstag{#1}\parbox[t]{14cm}{#2}}
  \else
    \g@addto@macro\@adressstringa{\\[\bigskipamount]\adresstag{#1}#2}
    \g@addto@macro\@adressstringb{\\[\medskipamount]\adresstag{#1}\parbox[t]{14cm}{#2}}
  \fi}

\def\preprint#1{\gdef\@preprint{#1}}
\def\@preprint{}
\def\keywords#1{\gdef\@keywords{#1}}
\def\@keywords{}
\def\msc#1{\gdef\@msc{#1}}
\def\@msc{}
\def\email#1{
   \gdef\@email{#1}
   \g@addto@macro\@authorstringc{ {\it (#1)}}}
\def\@email{}
\def\dedication#1{\gdef\@dedication{#1}}
\def\@dedication{}

\def\mybaselinestretch#1{
  \gdef\@mybaselinestretch{#1}
  \renewcommand{\baselinestretch}{\@mybaselinestretch}}

\def\myparskip#1{
  \gdef\@myparskip{#1}
  \setlength{\parskip}{\@myparskip}}

\mybaselinestretch{1.2}
\myparskip{1.5ex}

\binoppenalty=\maxdimen
\relpenalty=\maxdimen

\normalfont
\normalsize
\newlength{\@listleftmargin}
\def\setenumstandard{
  \setlist{leftmargin=\@listleftmargin,itemsep=0pt,topsep=0pt,partopsep=0pt,parsep=\@myparskip}
  \setlist[enumerate]{align=left,labelsep=*,leftmargin=\@listleftmargin,itemsep=0pt,topsep=0pt,partopsep=0pt,parsep=\@myparskip}
}

\def\denseversion{
  \setcounter{denseversion}{1}
  \newgeometry{left=3cm,right=3cm,top=3cm}
  \mybaselinestretch{1.1}
  \myparskip{0.8ex}
  \normalfont
  \def\possiblelinebreak{}
  \fancyfoot[C]{\itshape{--$\,\,$\thepage$\,\,$--}}}

\def\possiblelinebreak{\\}

\renewcommand{\emph}[1]{\def\reserved@a{it}\ifx\f@shape\reserved@a\ul{#1}\else\textit{#1}\fi}

\def\setcrefnames{}
\newcommand{\mytableofcontents}{
   \ifnum\value{denseversion}=0
     \tableofcontents
     \setcrefnames 
   \else
     \renewcommand{\baselinestretch}{1.1}
     \setlength{\parskip}{0ex}
     \normalfont
     \begingroup
     \def\addvspace##1{\vskip0.4em}
     \tableofcontents
     \setcrefnames 
     \endgroup
     \renewcommand{\baselinestretch}{\@mybaselinestretch}
     \setlength{\parskip}{\@myparskip}
     \normalfont
   \fi}

\newlength{\zeilenlaenge}
\def\putindent#1{
  \settowidth{\zeilenlaenge}{#1}
  \ifnum\zeilenlaenge>\textwidth
    #1
  \else
    \noindent #1
  \fi
}

\hypersetup{
    linktocpage = true,
    colorlinks,
    linkcolor=[RGB]{0,0,120},
    citecolor=[RGB]{0,0,120},
    urlcolor=[RGB]{0,0,120}
}

\def\pdfdaten{
  \hypersetup{
    pdftitle = {\@title},
    pdfauthor = {\@author},
    pdfkeywords = {\@keywords},    
    bookmarksopen = true,
    bookmarksopenlevel = 1
  }}  

\def\showkeywords{\begin{flushleft}\footnotesize\textbf{Keywords}: \@keywords\end{flushleft}}
\def\showmsc{\begin{flushleft}\footnotesize\textbf{MSC 2010}: \@msc\end{flushleft}}

\def\tocsection#1{\section*{#1}\addcontentsline{toc}{section}{#1}}

\def\mytitle{}
\def\zmptitle{
  \begin{tabular}{cc}
    \begin{minipage}[c]{0.4\textwidth}
      \begin{flushleft}
        \includegraphics[width=110pt]{../../tex/zmp}
      \end{flushleft}  
    \end{minipage}&
    \begin{minipage}[c]{0.55\textwidth}
      \begin{flushright}
      {\small\sf\@preprint}
      \end{flushright}
    \end{minipage}
  \end{tabular}
  \vskip 2cm}

\def\maketitle{
  \pdfdaten
  \noindent
  \mytitle
  \begin{center}
    \LARGE\@title\\
    \if!\@subtitle!\else\smallskip\LARGE\@subtitle\\\fi
    \bigskip
    \if!\@author!\else\bigskip\large\@author\\\fi
    \ifnum\value{denseversion}=0
      \if!\@adressstringa!\else\bigskip\normalsize\@adressstringa\\\fi
      \if!\@email!\else\ifnum\value{authorcounter}=1\bigskip\normalsize\textit{\@email}\\\else\fi\fi
    \else
    \fi
    \if!\@dedication!\else\bigskip\normalsize{\@dedication}\\\fi
  \end{center}
  \ifnum\value{denseversion}=0\vskip 1.5cm\else\vskip0.5cm\fi}

\def\kobib#1{
  \begin{raggedright}
  \ifnum\value{denseversion}=0\else\small\fi
  \Oldbibliography{#1/kobib}
  \bibliographystyle{#1/kobib}
  \end{raggedright}
  \ifnum\value{denseversion}=0\else
      \noindent
      \if!\@authorstringc!\else
        \ifnum\authortagsused=0\ifnum\value{authorcounter}>1\normalsize\@authorstringc\\[\medskipamount]\else\fi\else\normalsize\@authorstringc\\[\medskipamount]\fi
      \fi
      \if!\@adressstringb!\else\normalsize\@adressstringb\\{}\fi
      \ifnum\authortagsused=0
        \ifnum\value{authorcounter}=1
          \if!\@email!\else\linebreak\normalsize\textit{\@email}\\{}\fi
        \else
        \fi
      \else
      \fi
  \fi}

\let\Oldbibliography\bibliography
\def\bibliography#1{
  \begin{raggedright}
  \ifnum\value{denseversion}=0\else\small\fi
  \Oldbibliography{#1}
  \end{raggedright}
  \ifnum\value{denseversion}=0\else
      \medskip
      \noindent
      \if!\@authorstringc!\else
        \ifnum\authortagsused=0\ifnum\value{authorcounter}>1\normalsize\@authorstringc\\[\medskipamount]\else\fi\else\normalsize\@authorstringc\\[\medskipamount]\fi
      \fi
      \if!\@adressstringb!\else\normalsize\@adressstringb\\{}\fi
      \ifnum\authortagsused=0
        \ifnum\value{authorcounter}=1
          \if!\@email!\else\linebreak\normalsize\textit{\@email}\\{}\fi
        \else
        \fi
      \else
      \fi
  \fi
}

\newenvironment{commentfigure}{}
\newenvironment{sidewayscommentfigure}{\begin{minipage}}{\end{minipage}}
\newenvironment{displaycomment}{\begin{list}{}{\rightmargin=1cm\leftmargin=1cm}\item\sf\begin{small}}{\end{small}\end{list}}

\def\tocmark#1{}

\def\draftstamp#1{
  \def\tocmark##1{
    \ifnum\c@secnumdepth=0\section{##1}\fi
    \ifnum\c@secnumdepth=1\subsection{##1}\fi
    \ifnum\c@secnumdepth=2\subsubsection{##1}\fi
    \ifnum\c@secnumdepth=3\subsubsection{##1}\fi
  }
  \ifnum\value{comments}=0
    \gdef\@draft{DRAFT - Edited on \today\ by #1 - Comments are not displayed}
  \else
    \gdef\@draft{DRAFT - Edited on \today\ by #1 - Comments are displayed}
  \fi
  \fancyhead[C]{\footnotesize\tt\textcolor{red}{\@draft}}}

\def\skript{
  \renewenvironment{displaycomment}{}{}
  \ifnum\value{comments}=0
    \renewenvironment{example*}{\comment}{\endcomment}
    \renewenvironment{remark*}{\comment}{\endcomment}
  \else\fi
  \parindent=0mm	
}

\newcounter{sectioning}
\def\tsection#1{\ifnum\value{sectioning}=0\section{#1}\fi}
\def\lsection#1{
  \ifnum\value{sectioning}=1
    \clearpage
    \def\thesection{\lektionname~\arabic{section}:}
    \section{#1}
    \def\thesection{\arabic{section}}
  \fi}
\def\tsubsection#1{\ifnum\value{sectioning}=0\subsection{#1}\fi}
\def\lsubsection#1{\ifnum\value{sectioning}=1\subsection{#1}\fi}
\def\tsubsubsection#1{\ifnum\value{sectioning}=0\subsubsection{#1}\fi}
\def\lsubsubsection#1{\ifnum\value{sectioning}=1\subsubsection{#1}\fi}

\mybaselinestretch{}
\setenumstandard

\makeatother

%% file: makros.tex
\usepackage{amssymb}
\usepackage{amsmath}
\usepackage{amstext}
\usepackage{mathrsfs}
\usepackage{euscript}
\usepackage{color}
\usepackage[all,cmtip,color,arrow]{xy}
\usepackage{xstring}
\usepackage{slashed}
\usepackage{tensor}
\usepackage{tikz}

\entrymodifiers={+!!<0pt,\fontdimen22\textfont2>}

\def\Z {\mathbb{Z}}

\def\R {\mathbb{R}}
\def\C {\mathbb{C}}
\def\T {\mathbb{T}}
\def\im{\mathrm{i}}
\def\id{\mathrm{id}}

\def\hc#1{\mathrm{h}_{#1}}
\def\h {\mathrm{H}}

\def\subset{\subseteq}

\def\sep{\;|\;}

\renewcommand{\varepsilon}{\epsilon}

\newcommand{\incl}{\hookrightarrow}

\def\vektor#1{\begin{pmatrix}#1\end{pmatrix}}

\makeatletter
\renewenvironment{proof}[1][\nameProof]
  {\par\pushQED{\qed}%
   \normalfont \topsep6\p@\@plus6\p@\relax
   \trivlist
   \item[\hskip\labelsep
         \itshape
         #1\@addpunct{.}]
  \leavevmode}
  {\popQED\endtrivlist\@endpefalse}
\makeatother

\def\notebox#1#2{\begin{minipage}[b]{#1}\sloppy\renewcommand{\baselinestretch}{0.8}\footnotesize \begin{center}#2\end{center}\end{minipage}}

\def\mquad{\hspace{-2em}}
\def\mqquad{\hspace{-4em}}

\newcommand{\arr}[1][r]{\ar@<0.7ex>[#1]\ar@<-0.7ex>[#1]}
\newcommand{\arrr}[1][r]{\ar@<1.4ex>[#1]\ar[#1]\ar@<-1.4ex>[#1]}
\newcommand{\arrrr}[1][r]{\ar@<2.1ex>[#1]\ar@<-2.1ex>[#1]\ar@<0.7ex>[#1]\ar@<-0.7ex>[#1]}

\def\stackref#1#2{\stackrel{\text{\ref{#1}}}{#2}}
\def\eqref#1{\stackref{#1}{=}}

\newlength{\myeqt} 
\newlength{\myeqs} 
\newlength{\myeqm} 
\newlength{\myeqn} 
\newcommand\symtext[3][\myeqn]{
  \settowidth{\myeqt}{#2}
  \settowidth{\myeqs}{$#3$}
  \addtolength{\myeqs}{\the\myeqm}
  \ifdim\myeqt>\myeqs
    \stackrel{\hspace{-#1}\notebox{#1}{\medskip #2 \\ $\downarrow$\smallskip}\hspace{-#1}}{#3}
  \else
    \stackrel{\text{#2}}{#3}
  \fi}
\newcommand\eqcref[2][\myeqn]{\symcref[#1]{#2}{=}}

\newcommand\symcref[3][\myeqn]{\symtext[#1]{\cref{#2}}{#3}}

\def\brackets#1{\IfStrEq{#1}{-}{}{(#1)}}
\def\subindex#1{\IfStrEq{#1}{-}{}{_{#1}}}

\newcommand{\alxydim}[2]{\begin{aligned}\xymatrix#1{#2}\end{aligned}}

\makeatletter
\def\bigset#1#2{\left\lbrace\;\begin{minipage}[c]{#1}\begin{center}#2\end{center}\end{minipage}\;\right\rbrace}

\makeatother

\newlength{\myl}
\newcommand\sheaf[1]{\unitlength 0.1mm
  \settowidth{\myl}{$#1$}
  \addtolength{\myl}{-0.8mm}
  \begin{picture}(0,0)(0,0)
  \put(2,0){\text{\underline{\hspace{\myl}}}}
  \end{picture}#1\hspace{-0.15mm}}

\def\ddt#1#2#3{\left.\frac{\mathrm{d}^{\IfStrEq{#1}{1}{}{#1}}}{\mathrm{d}#2}\IfStrEq{#2}{#3}{\right.}{\right|_{#3}}}

\newcommand{\ueins}{{\mathrm{U}}(1)}

\newcommand{\su}[1]{{\mathrm{SU}}\brackets{#1}}

\def\hom{\mathcal{H}\!om}

\def\pr{{\mathrm{pr}}}

\newlength{\widthtmp}
\def\length#1{\settowidth{\widthtmp}{#1}\the\widthtmp}

\def\ttimes#1#2{\hspace{-0.15em}\tensor[_{#1}]{\times}{_{#2}}}

\def\buntech#1#2{\mathcal{B}\hspace{-0.01em}un_{\hspace{0.05em}#1}^{#2}}
\def\trivlin{\mathbf{I}}

\def\buncon#1#2{\buntech{#1}{\nabla}\hspace{-0.05em}\brackets{#2}}

\def\bunconflat#1#2{\buntech#1{\nabla_{\!0}}\hspace{-0.05em}\brackets{#2}}


%% file: hyphenation_en.tex
\usepackage[latin1]{inputenc}
\usepackage[english]{babel}
\usepackage[english]{cleveref}

\def\refname{References}

\def\quot#1{``#1''}

\def\quand{\quad\text{ and }\quad}
\def\quomma{\quad\text{, }\quad}

\def\quith{\quad\text{ with }\quad}

\def\nameTheorem{Theorem}
\def\nameTheorems{Theorems}
\def\nameDefinition{Definition}
\def\nameDefinitions{Definitions}
\def\nameCorollary{Corollary}
\def\nameCorollaries{Corollaries}
\def\nameProposition{Proposition}
\def\namePropositions{Propositions}
\def\nameConstruction{Construction}
\def\nameConstructions{Constructions}
\def\nameLemma{Lemma}
\def\nameLemmas{Lemmas}
\def\nameRemark{Remark}
\def\nameRemarks{Remarks}
\def\nameProblem{Problem}
\def\nameProblems{Problems}
\def\nameExample{Example}
\def\nameExamples{Examples}
\def\nameExercise{Exercise}
\def\nameExercises{Exercises}

\def\nameProof{Proof}
\def\nameEquation{Eq.}
\def\nameEquations{Eqs.}
\def\nameSection{Section}
\def\nameSections{Sections}
\def\nameAppendix{Appendix}
\def\nameAppendixs{Appendices}
\def\nameFigure{Figure}
\def\nameFigures{Figures}